\documentclass[twoside,a4paper,11pt]{amsart}

\usepackage[T1]{fontenc}
\usepackage{mathpazo} 
\usepackage{eulervm}

\usepackage[utf8]{inputenc}

\usepackage{fullpage}
\usepackage{amscd}
\usepackage{amsmath}
\usepackage{amssymb}
\usepackage{amsthm}
\usepackage{color}
\usepackage{graphicx}
\usepackage{hyperref}
\usepackage[noabbrev,capitalize]{cleveref}
\usepackage{mathrsfs}
\usepackage{mathtools}
\usepackage{pict2e}
\usepackage{stackrel}
\usepackage[T1]{fontenc}
\usepackage{tikz}
\usepackage{tikz-cd}
\usetikzlibrary{decorations.pathmorphing,decorations.markings,arrows,calc,shapes.geometric,arrows.meta,positioning}
\usepackage[utf8]{inputenc}
\usepackage{wasysym}
\usepackage[all]{xy}
\usepackage{xcolor}
\usepackage{xy}
\usepackage{CJKutf8}
\usepackage{ipaex-type1}
\usepackage{epigraph}

\hypersetup{
    colorlinks,
    linkcolor={red!50!black},
    citecolor={blue!50!black},
    urlcolor={blue!80!black}
}

\usepackage{palatino}
\usepackage{mathpazo} 
\allowdisplaybreaks

\newtheorem{lemma}{Lemma}[section]
\newtheorem{theorem}[lemma]{Theorem}
\newtheorem{Corollary}[lemma]{Corollary}
\newtheorem{Proposition}[lemma]{Proposition}

\newtheorem*{Notation}{Notation}

\newtheorem{theoremintro}{Theorem}

\theoremstyle{definition}
\newtheorem{Definition}[lemma]{Definition}
\newtheorem{Remark}[lemma]{\sc Remark}
\newtheorem{Counterexample}[lemma]{\sc Counter-Example}
\newtheorem{Example}[lemma]{\sc Example}

\def\colim{\mathop{\mathrm{colim}}}

\newcommand{\palg}{\PP\textsf{-}\mathsf{alg}}

\newcommand{\partialop}{\mathsf{pOp}}

\newcommand{\partialcoop}{\mathsf{pCoop}}

\newcommand{\smod}{\mathbb{S}\textsf{-}\mathsf{mod}}
\newcommand{\kmod}{\mathbb{K}\textsf{-}\mathsf{mod}}

\newcommand{\ac}{\scriptstyle \textrm{!`}}

\newcommand{\qi}{\xrightarrow{ \,\smash{\raisebox{-0.65ex}{\ensuremath{\scriptstyle\sim}}}\,}}

\newcommand{\Cobar}{\Omega}

\newcommand{\C}{\mathcal{C}}

\newcommand{\kk}{\mathbb{K}}

\newcommand{\PP}{\mathcal{P}}
\newcommand{\Q}{\mathcal{Q}}
\newcommand{\D}{\mathcal{D}}
\newcommand{\G}{\mathcal{G}}
\newcommand{\ucom}{u\mathcal{C}om}
\newcommand{\ucomd}{u\mathcal{C}om^{*}}

\newcommand{\Assc}{c\mathcal{A}ss}
\newcommand{\Liec}{c\mathcal{L}ie}
\newcommand{\cP}{(\PP,d_\PP, \Theta_\PP)}

\input cyracc.def
\font\tencyr=wncysc10
\def\cyr{\tencyr\cyracc}
\def\diracComb{\mbox{\cyr SH}}

\author{Victor Roca i Lucio}
\title{Curved Operadic Calculus}
\date{\today}

\address{Victor Roca i Lucio, Ecole Polytechnique Fédérale de Lausanne, EPFL,
CH-1015 Lausanne, Switzerland}
\email{\href{mailto:victor.rocalucio@epfl.ch}{victor.rocalucio@epfl.ch}}

\subjclass[2020]{Primary 18M70; Secondary 18M60, 18C35.}

\keywords{Curved operads, curved algebras, universal algebra, homotopical algebra.}

\thanks{The author was partially supported by the project ANR-20-CE40-0016 HighAGT funded by the Agence Nationale pour la Recherche.}

\begin{document}
	
\begin{abstract}
Curved algebras are a generalization of differential graded algebras which have found numerous applications recently. The goal of this foundational article is to introduce the notion of a curved operad, and to develop the operadic calculus at this new level. The algebraic side of the curved operadic calculus provides us with universal constructions: using a new notion of curved operadic bimodules, we construct curved universal enveloping algebras. Since there is no notion of quasi-isomorphism in the curved context, we develop the homotopy theory of curved operads using new methods. This approach leads us to introduce the new notion of a curved absolute operad, which is the notion Koszul dual to counital cooperads non-necessarily conilpotent, and we construct a complete Bar-Cobar adjunction between them. We endow curved absolute operads with a suitable model category structure. We establish a duality square of duality functors which intertwines this complete Bar-Cobar construction with the Bar-Cobar adjunction between unital operads and conilpotent curved cooperads. This allows us to compute minimal cofibrant resolutions for various curved absolute operads. Using the complete Bar construction, we show a general Homotopy Transfer Theorem for curved algebras. Along the way, we construct the non-necessarily conilpotent cofree cooperad. 
\end{abstract}

\maketitle

\setcounter{tocdepth}{1}

\tableofcontents

\section*{Introduction}

\textbf{Global picture.} The theory of operads shifts the point of view of universal algebra: instead of working by hand with specific types of algebras, one works with the operads that encode them. Examples of algebraic structures encoded by operads include associative algebras, Lie algebras, Poisson algebras, Batalin--Vilkovisky algebras, etc. One calls \textit{operadic calculus} the set of techniques that allow us to work with operads themselves. There is a purely algebraic side to the operadic calculus. For instance, one might consider morphisms between operads. A morphism between operads induces a structure of an operadic bimodule. Via the theory of operadic bimodules, one can construct universal functors between the categories of algebras over operads. Thus this theory gives, for any morphism of operads, a universal adjunction between their respective categories of algebras. The universal enveloping algebra of a Lie algebra can be recovered in this way. This approach allows vast generalizations, and new universal enveloping algebras can be constructed for other types of algebraic structures in this way. It can also generalize well-known theorems such as the Poincaré--Birkoff--Witt theorem in a functorial way, see \cite{Tamaroff2020}. 

\medskip

The operadic calculus has also an homotopical side. When operads themselves live in the category of differential graded $\mathbb{S}$-modules, they admit a notion of weak equivalence given by arity-wise quasi-isomorphisms. Understanding the homotopy theory of operads gives results on the homotopy theory in the category of algebras it encodes. Two weakly equivalent operads encode the same homotopy category of algebras. For any operad $\mathcal{P}$, algebras over a cofibrant resolution of $\mathcal{P}$ provide us with a suitable notion of a $\mathcal{P}$-\textit{algebra up to homotopy}. This recovers the seminal notions of $\mathcal{A}_\infty$-algebras, $\mathcal{L}_\infty$-algebras, $\mathcal{C}_\infty$-algebras as particular examples. These resolutions also allow one to construct universal André-Quillen cohomology theories for algebras over an operad, see \cite{JoanQuillen}. The main tools for studying the homotopy theory of operads and computing cofibrant resolutions are the Bar-Cobar adjunction and the Koszul duality theory. Combining both the algebraic and the homotopical aspects of the operadic calculus provides us with new tools to solve problems, see \cite{LodayVallette12}. For example, in \cite{campos2020lie}, R. Campos, D. Petersen, D. Robert-Nicoud, and F. Wierstra show that a nilpotent Lie algebra is completely characterized up to isomorphism by its universal enveloping algebra, seen as an associative algebra. The proof of this purely algebraic result requires a combination of both of the aforementioned methods. Another example is given by \cite{robertnicoud2020higher}, where D. Robert-Nicoud and B. Vallette develop the integration theory of $\mathcal{L}_\infty$-algebras using methods coming from the operadic calculus. One of the main motivations of this article was to lay down the operadic tools required to generalize the results of \textit{loc.cit} to the case of \textit{curved} $\mathcal{L}_\infty$-algebras. See \cite{integration}.

\medskip

In the mist of algebraic structures, there are the so-called \textit{curved algebras}. The prototype of curved algebras are curved associative algebras. These are graded associative algebras $(A, \mu_A)~,$ endowed with a derivation $d$ of degree $-1$ and a distinguished element $\theta$ of degree $-2$ called the \textit{curvature}, such that 
\[
d^2(-) = \mu_A(\theta,-) - \mu_A(-,\theta)~. \quad \quad (*)
\]
Here, the element $\theta$ is the obstruction for the derivation $d$ to square to zero. When $\theta$ is non-trivial, this type of algebras do not have underlying homology groups, thus there is no notion of quasi-isomorphism for them in general. In these examples, the distinction between homological algebra and the more general concept of homotopical algebra becomes apparent. Other examples include curved Lie algebras, curved $\mathcal{A}_\infty$-algebras, curved $\mathcal{L}_\infty$-algebras and so on. These types of algebras are playing an increasingly important role in various areas of mathematics: curved $\mathcal{A}_\infty$-algebras in Floer cohomology in symplectic geometry \cite{FOOO7}, curved $\mathcal{L}_\infty$-algebras in derived differential geometry \cite{behrend2021derived}, derived deformation theory \cite{calaque2021lie} and $\mathcal{L}_\infty$-spaces \cite{costello2011geometric}. 

\medskip

\textbf{Main results.} The goal of this foundational article is to settle the operadic calculus for \textit{curved operads}. Curved operads appear naturally when one tries to encode types of curved algebras with operad-like structures: to encode the curvature relation $(*)$ at the algebra level, one needs to add a curvature on the operad level. Working with curved operads forces the underlying category to be the category of pre-differential graded (pdg) modules, which are given by graded modules with a degree $-1$ endomorphism. So we start by settling the basic properties related to this new framework. In order to develop the algebraic side of the curved operadic calculus, we introduce the notion of a curved operadic bimodule and we develop the subsequent theory. There are obstructions to this generalization: the "free algebra" over a curved operad does not, in general, satisfy the curvature relation. This implies that a curved operad is not a "naive" curved left module over herself. Nevertheless, we develop the theory of curved bimodules for curved operads. This allows us to construct universal functors between the categories of curved algebras. For instance, we construct for the first time the universal curved enveloping algebra for curved Lie algebras and the universal curved enveloping $\mathcal{A}_\infty$-algebra for curved $\mathcal{L}_\infty$-algebras. 

\medskip

Once this algebraic framework is established, the rest of the article is devoted to generalizing the homotopical side of the operadic calculus to the curved setting. We start \textit{ab initio} from a conceptual point of view. We use the groupoid-colored formalism developed by B. Ward in \cite{ward19}. We introduce the unital groupoid-colored operad $u\mathcal{O}$ that encodes unital partial operads as its algebras and counital partial cooperads as its coalgebras. By partial operad, we mean the definition of operads introduced by M. Markl in \cite{Markl96} in terms of partial composition maps $\{\circ_i\}$. The notion of partial cooperads is the dual notion, defined in terms of partial decomposition maps $\{\Delta_i\}$. We generalize the main point of the inhomogeneous Koszul duality of Hirsh--Milles in \cite{HirshMilles12} to extend them to the groupoid-colored framework. This allows us to compute the conilpotent curved groupoid-colored dual cooperad $u\mathcal{O}^{\ac} \cong c\mathcal{O}^\vee$ (up to suspension). We show that (curved) coalgebras over $c\mathcal{O}^\vee$ correspond to conilpotent curved partial cooperads. 

\begin{theoremintro}[Theorem \ref{Koszulity of uO}]
The unital groupoid-colored operad $u\mathcal{O}$ is a Koszul operad and its Koszul dual curved groupoid-colored cooperad is given by $c\mathcal{O}^\vee$ up to suspension.
\end{theoremintro}

In particular, this result implies that there is a Koszul curved twisting morphism 
\[
\kappa: c\mathcal{O}^\vee \longrightarrow u\mathcal{O}~.
\]
From this curved twisting morphism, one obtains a first Bar-Cobar adjunction using the classical methods of \cite{LodayVallette12} and \cite{grignou2019}. This adjunction recovers the Bar-Cobar adjunction between unital partial operads and conilpotent curved partial cooperads constructed by B. Le Grignou in \cite{grignou2021}. In \textit{loc.cit.}, the author endows the category of unital operads with a model structure where weak-equivalences are given by arity-wise quasi-isomorphisms, and then transfers it along this adjunction to conilpotent curved coaugmented cooperads. This endows conilpotent curved coaugmented cooperads with a meaningful model structure, and the Bar-Cobar adjunction is shown to be a Quillen equivalence. The above theorem provides a new proof that the Bar-Cobar adjunction recovered is indeed a Quillen equivalence.

\medskip

Our goal here is to construct another adjunction using the curved twisting morphism $\kappa$, in order to obtain a Koszul dual notion of counital partial cooperads. In order so, we generalize the results of D. Lejay and B. Le Grignou in \cite{grignoulejay18} to the groupoid-colored case. In \textit{op.cit}, from a curved twisting morphism the authors construct a "complete" Bar-Cobar adjunction between the category of coalgebras over an operad and the category of curved algebras over a curved cooperad. The idea is that the Koszul dual of \textit{non-conilpotent} types of coalgebras are \textit{absolute} types of algebras, which are algebraic structures endowed with well-defined infinite sums of operations without any having an underlying topology. 

\medskip

In our case, we get an adjunction between the category of coalgebras over the groupoid-colored operad $u\mathcal{O}$ and the category of curved algebras over the curved groupoid-colored cooperad $c\mathcal{O}^\vee$, which we call the \textit{complete Bar-Cobar adjunction}. Coalgebras over $u\mathcal{O}$ are simply counital partial cooperads. Whereas curved algebras $c\mathcal{O}^\vee$ give rise to new objects in the operadic calculus, which we call \textit{curved absolute partial operads}. There is also a more basic notion of \textit{absolute partial operad}. Both of them can be though as partial operads where infinite sums of partial compositions have a well-defined image. These notions are further characterized in Appendix \ref{Appendix B}. Note that (curved) absolute partial operads appear precisely when arity $0$ and $1$ pehnomena are taken into account, and coincide with usual notion of a partial operad otherwise. 

\medskip

Summarizing, we obtain a complete Bar-Cobar adjunction
\[
\begin{tikzcd}[column sep=7pc,row sep=3pc]
            \mathsf{dg}~\mathsf{upCoop} \arrow[r, shift left=1.1ex, "\widehat{\Omega}"{name=F}] &\mathsf{curv}~\mathsf{abs}~\mathsf{pOp}~. \arrow[l, shift left=.75ex, "\widehat{\mathrm{B}}"{name=U}]
            \arrow[phantom, from=F, to=U, , "\dashv" rotate=-90]
\end{tikzcd}
\]

Notice first that this adjunction interrelates "counits" with "curvature", which are known to be Koszul dual. See \cite{PolischukPositselski05}. But its main novelty lies in the fact that it also lifts the usual assumption of \textit{conilpotency} on the cooperad side. Lifting conilpotency on one side of the Koszul duality is what makes infinite sums appear in the other side. This is the conceptual explanation for the existence of these curved \textit{absolute} partial operads. Absolute partial operads admit a canonical filtration, dual to the coradical filtration for partial cooperads. And like the coradical filtration, this filtration comes from the structure of the object. By definition, infinite sums of partial compositions are well-defined. But the topology induced by this filtration might not be Hausdorff. We restrict ourselves to the sub-category of those which are complete. Notice that unlike many other approaches to curved objects, where one \textit{changes the base category} from graded modules to filtered/complete graded modules in order to deal with the infinite sums that appear, here we introduce new types of algebraic structures which admit these infinite sums \textit{without enriching further the underlying category of graded modules}. And the complete filtrations that appear, come from the structure of these new objects, and are therefore \textit{canonical}.

\medskip

Considering these new algebraic objects, namely complete curved absolute partial operads, is what allows us to obtain a well-behaved homotopy theory. Since the notion of a quasi-isomorphism does not exist in the curved context, our approach is to obtain a notion of weak equivalences via a transfer theorem from a Koszul dual category. For this purpose, the category of counital partial cooperads is too narrow. We define the notion of counital partial cooperads up to homotopy. Then, we endow the category of counital partial cooperads up to homotopy with strict morphisms with a model category structure where weak equivalences are given by arity-wise quasi-isomorphisms. Finally, we construct another complete Bar-Cobar adjunction and transfer this model structure to the category of complete curved absolute partial operads. 

\begin{theoremintro}[Theorem \ref{thm: model structure on curved operads}]
The category of complete curved absolute partial operads admits a model structure transferred along the adjunction
\[
\begin{tikzcd}[column sep=7pc,row sep=3pc]
            \mathsf{upCoop}_\infty \arrow[r, shift left=1.1ex, "\widehat{\Omega}_\iota"{name=F}] &\mathsf{curv}~\mathsf{abs}~\mathsf{pOp}^{\mathsf{comp}}~, \arrow[l, shift left=.75ex, "\widehat{\mathrm{B}}_\iota"{name=U}]
            \arrow[phantom, from=F, to=U, , "\dashv" rotate=-90]
\end{tikzcd}
\]
where the model category structure considered on the left-hand side has arity-wise quasi-isomorphisms as weak-equivalences and monomorphisms as cofibrations.
\end{theoremintro}

We then relate the "classical" Bar-Cobar adjunction with the complete Bar-Cobar adjunction constructed here. For this purpose, we construct a pair of duality adjunctions that relate these Bar-Cobar adjunctions. First we do this in the algebraic context, where we prove the following result.

\begin{theoremintro}[Theorem \ref{thm: carré magique}]
The following square of adjunction 
\[
\begin{tikzcd}[column sep=5pc,row sep=5pc]
\mathsf{dg}~\mathsf{upOp}^{\mathsf{op}} \arrow[r,"\mathrm{B}^{\mathsf{op}}"{name=B},shift left=1.1ex] \arrow[d,"(-)^\circ "{name=SD},shift left=1.1ex ]
&\left(\mathsf{curv}~\mathsf{pCoop}^{\mathsf{conil}}\right)^{\mathsf{op}}  \arrow[d,"(-)^*"{name=LDC},shift left=1.1ex ] \arrow[l,"\Omega^{\mathsf{op}}"{name=C},,shift left=1.1ex]  \\
\mathsf{dg}~\mathsf{upCoop} \arrow[r,"\widehat{\Omega}"{name=CC},shift left=1.1ex]  \arrow[u,"(-)^*"{name=LD},shift left=1.1ex ]
&\mathsf{curv}~\mathsf{abs}~\mathsf{pOp}^{\mathsf{comp}}~, \arrow[l,"\widehat{\mathrm{B}}"{name=CB},shift left=1.1ex] \arrow[u,"(-)^\vee"{name=TD},shift left=1.1ex] \arrow[phantom, from=SD, to=LD, , "\dashv" rotate=0] \arrow[phantom, from=C, to=B, , "\dashv" rotate=-90]\arrow[phantom, from=TD, to=LDC, , "\dashv" rotate=0] \arrow[phantom, from=CC, to=CB, , "\dashv" rotate=-90]
\end{tikzcd}
\] 
commutes in the following sense: right adjoints going from top right corner to bottom left corner are naturally isomorphic. 
\end{theoremintro}

Here the functor $(-)^\circ$ is a direct generalization of the Sweedler duality functor defined in \cite{Sweedler69} for associative algebras, and the functor $(-)^\vee$ can be thought of as a topological dual. Then, by replacing (co)unital partial (co)operads on the left-hand side of the square with their "up to homotopy" counterparts, we construct another commuting square of duality adjunctions which are, in this case, all Quillen adjunctions. Using this duality square, we construct explicit cofibrant resolutions for a vast class of complete curved absolute partial operads. For instance, we construct a minimal cofibrant resolution of complete curved absolute partial operads encoding curved Lie algebras, which provides us with a suitable notion of curved Lie algebra up to homotopy. Applications of these results can be found in \cite{integration}, where we develop the integration theory of curved absolute $\mathcal{L}_\infty$-algebras, and apply it to deformation theory and rational homotopy theory. 

\medskip

Finally, using our complete Bar construction, we extend the classical Homotopy Transfer Theorem to the case of algebras over a general class of cofibrant complete curved absolute operads. Let $V$ and $H$ be two complete pdg modules. A \textit{homotopy contraction} amounts to the data of 
\[
\begin{tikzcd}[column sep=5pc,row sep=3pc]
V \arrow[r, shift left=1.1ex, "p"{name=F}] \arrow[loop left]{l}{h}
&H~, \arrow[l, shift left=.75ex, "i"{name=U}]
\end{tikzcd}
\]
where $p$ and $i$ are two morphisms of filtered pdg modules and $h$ is a morphism of filtered graded modules of degree $-1$, which satisfy standard conditions. This data allows us to construct a \textit{Van der Laan morphism} between the complete Bar constructions of the curved endomorphisms operads of $V$ and $H$. See \cite[Chapter 10]{LodayVallette12} for the classical Van der Laan morphism. Using this new morphism, we obtain the following result.

\begin{theoremintro}[Theorem \ref{thm: curved HTT}]
Let $\mathcal{C}$ be a dg counital partial cooperad. Let $H$ be a homotopy retract of $V$ and let 
\[
\varphi: \widehat{\Omega}\mathcal{C} \longrightarrow \mathrm{end}_V
\]
be a curved algebra structure on $V$. There is a transferred curved $\widehat{\Omega}\mathcal{C}$-algebra on $H$ constructed from the Van der Laan morphism.
\end{theoremintro}

In the particular case of complete curved $\mathcal{L}_\infty$-algebras, the transferred structure coincides with that constructed by K. Fukaya in \cite{FukayaHTT}. 

\medskip

Although (curved) absolute partial operads appear naturally when one considers the groupoid-colored point of view, there are still many open questions related to them. A crucial one is to have a meaningful notion of a (curved) \textit{absolute algebra} over a (curved) absolute partial operad. One can forget the "absolute" structure and only consider (curved) algebras on the underlying (curved) partial operad, or one can define them using filtered objects like we do Section \ref{Section: curved HTT}. These answers are not fully satisfactory, and expanding this theory should be the subject of future work. 

\subsection*{Layout}
This article is structured as follows: in Section 1, we recall the different definition of operads and cooperads, and we define their respective categories of algebras and coalgebras. This includes the notion of an algebra over a cooperad introduced by Le Grignou--Lejay which will play an important role in the subsequent sections. This section contains non-standard definitions even for experts in the operadic literature. In Section 2, we introduce the notion of a curved operad in the underlying category of pre-differential modules and establish basic properties. In Section 3, we develop the theory of curved bimodules over curved operads and we prove that they induce universal functors between categories of curved algebras. In Section 4, we recall different results on curved cooperads and their categories of curved coalgebras and curved algebras.

\medskip

In Section 5, we explain the formalism of algebras and coalgebras over operads and cooperads in the $\mathbb{S}$-colored case. We extend the inhomogeneous Koszul duality to this framework in order to prove that the unital $\mathbb{S}$-colored operad $u\mathcal{O}$, which encodes unital partial operads as its algebras, is a Koszul $\mathbb{S}$-colored operad. In Section 6, we use this $\mathbb{S}$-colored inhomogeneous Koszul duality to construct a new complete Bar-Cobar adjunction. In Section 7, we endow the category of complete curved absolute partial operads with a transferred model category structure from counital partial cooperads up to homotopy, considered with strict morphisms. In Section 8, we construct the duality adjunctions that intertwine the Bar-Cobar adjunction constructed before with the Bar-Cobar adjunction constructed by Le Grignou. We use this duality square to obtain explicit cofibrant resolutions. In Section 9, we apply the complete Bar-Cobar adjunction to obtain a general Homotopy Transfer Theorem for complete curved absolute partial operads. In the appendix, we define and characterize the key notion of \textit{absolute partial operads} and their curved counterparts.

\subsection*{Acknowledgments}
This work was carried out during my PhD at the Université Paris 13, and I would like to thank my former advisor Bruno Vallette for numerous discussions which led to this paper as well as for the careful readings of this paper. I would also like to thank Brice Le Grignou for the time he spent discussing with me, as well as the interesting ideas he shared with me, like the suggestion to use of counital partial cooperads up to homotopy in Section 7. I thank Mathieu Anel, Joan Bellier-Millès, Ricardo Campos, Damien Calaque, Geoffroy Horel, and all the members of the Séminaire Roberta for interesting discussions. Finally, I would also like to thank the referee of this paper for the numerous comments, remarks and corrections that helped greatly improve this paper.

\subsection*{Conventions}
Let $\mathbb{K}$ be a ground field of characteristic $0$. The ground category is the symmetric monoidal category $(\mathsf{gr}\textsf{-}\mathsf{mod}, \otimes, \mathbb{K})$ of graded $\mathbb{K}$-modules, with the tensor product $\otimes$ of graded modules given by 
\[
(A\otimes B)_n \coloneqq \displaystyle \bigoplus_{p+q=n}~A_p \otimes~ B_q ~.
\]
The tensor is taken over the base field $\mathbb{K}$, which will be implicit from now on. The isomorphism $\tau_{A,B}: A \otimes B \longrightarrow B \otimes A$ is given by the Koszul sign rule $\tau (a \otimes b) \coloneqq (-1)^{|a|.|b|} \hspace{1pt} b\otimes a$ on homogeneous elements. We work with the \textit{homological} degree convention, that is, the degree of the (pre)-differentials considered is $-1$. The suspension of a graded module $V$ is denoted by $sV$, given by $(sV)_{p} \coloneqq V_{p-1}~.$ A graded $\mathbb{S}$-module  $M$ is a collection $\{M(n)\}_{n \in \mathbb{N}}$ of graded $\kk[\mathbb{S}_n]$-modules. This category is denoted by $\mathsf{gr}~\mathbb{S}\textsf{-}\mathsf{mod}$. We denote $(\mathsf{gr}~\mathbb{S}\textsf{-}\mathsf{mod}, \circ, \I)$ the monoidal category of graded $\mathbb{S}$-modules endowed with the composition product $\circ$. We adopt the conventions of \cite{LodayVallette12}, unless stated otherwise. 

\section{Recollections on different types of (co)operads}\label{Section: Recollections}
In this section, we recall briefly the (co)monoidal definition of (co)operads. Operads were introduced to encode types of algebras and cooperads to encode types of conilpotent coalgebras. We explain briefly the dual point of view developed in \cite{grignoulejay18}, where operads encode non-necessarily conilpotent coalgebras, and where cooperads encode a new type of algebras. This type of algebras have well-defined infinite sums of operations \textit{by definition}. We review the partial definitions of (co)operads and compare them with the (co)monoidal definitions. The partial definitions are extremely useful to study operads themselves. These categories are the "natural habitat" of the operadic Bar-Cobar constructions defined later on. The material presented here contains non-standard results, thus we encourage experts in the field to read it as well.

\subsection{(Co)operads as (co)monoids and their algebras and coalgebras}\label{subsection : algebras and coalgebras}
We recall the basic definition of operads and cooperads as monoids and comonoids in $(\mathsf{gr}~\mathbb{S}\textsf{-}\mathsf{mod}, \circ, \I)$. We omit the epithet "graded" before each definition, although it is implicit that all our objects are graded.

\begin{Definition}[Operad]\label{def: operad}
An \textit{operad} $(\PP, \gamma_\PP, \eta)$ is the data of a unital monoid in $(\mathsf{gr}~\mathbb{S}\textsf{-}\mathsf{mod}, \circ, \I)$.   
\end{Definition}

\begin{Definition}[Cooperad]\label{def: cooperad}
A \textit{cooperad} $(\C,\Delta,\epsilon)$ is the data of a counital comonoid in $(\mathsf{gr}~\mathbb{S}\textsf{-}\mathsf{mod}, \circ, \I)~.$ 
\end{Definition}

\begin{Notation}
If there is no ambiguity, we will denote the data of an operad $(\PP, \gamma_\PP, \eta)$ simply by $\mathcal{P}$. Likewise, we will denote the data of a cooperad $(\C,\Delta,\epsilon)$ simply by $\mathcal{C}$. 
\end{Notation}

Theses definitions are well suited for defining algebras and coalgebras because of their "monadic" nature. To any graded $\mathbb{S}$-module one can associate an endofunctor in the category of graded modules via the \textit{Schur realization functor}
\[
\begin{tikzcd}[column sep=4pc,row sep=0.5pc]
&\mathscr{S} : \mathsf{gr}~\smod \arrow[r]
&\mathsf{End}(\mathsf{gr}~\mathsf{mod}) \\
&M \arrow[r,mapsto]
&\mathscr{S}(M)(-) \coloneqq \displaystyle \bigoplus_{n \geq 0} M(n) \otimes_{\mathbb{S}_n} (-)^{\otimes n}~.
\end{tikzcd}~
\]

\begin{lemma}
The Schur realization functor defines a strong monoidal functor between the monoidal categories $(\mathsf{gr}~\mathbb{S}\textsf{-}\mathsf{mod}, \circ, \I)$ and $(\mathsf{End}(\mathsf{gr}~\mathsf{mod}), \circ,\mathrm{Id})$, where the monoidal product on endofuctors is given by the composition of endofunctors. 
\end{lemma}

This implies that if $\PP$ is a operad, then $\mathscr{S}(\PP)$ is a monad and that if $\C$ is a cooperad, then $\mathscr{S}(\mathcal{C})$ is comonad. 

\begin{Definition}[$\PP$-algebras]
A $\mathcal{P}$\textit{-algebra} $A$ amounts to the data $(A,\gamma_A)$ of an algebra over the monad $\mathscr{S}(\PP)$.
\end{Definition}

\begin{Definition}[$\mathcal{C}$-coalgebras]
A $\mathcal{C}$\textit{-coalgebra} $C$ amounts to the data $(C,\Delta_C)$ of a coalgebra over the comonad $\mathscr{S}(\C)$.
\end{Definition}

\begin{Remark}
The notion of a coalgebra over a cooperad can only encapsulate types of coalgebras which are conilpotent in some sense. The reason is that the structural map
\[
\Delta_C: C \longrightarrow \bigoplus_{n \geq 0} \mathcal{C}(n) \otimes_{\mathbb{S}_n} C^{\otimes n}~,
\]
lands on the direct sum of a non-completed tensor product. Thus decompositions of elements in $C$ are supported in finitely many arities and in finitely many steps of the coradical filtration of $\mathcal{C}$.
\end{Remark}

There is another kind of Schur realization functor, which was introduced in \cite{grignoulejay18}. The \textit{dual Schur realization} functor $\widehat{\mathscr{S}}^c$ is given by

\[
\begin{tikzcd}[column sep=4pc,row sep=0.5pc]
\widehat{\mathscr{S}}^c : \mathsf{gr}~\smod^{\mathsf{op}} \arrow[r]
&\mathsf{End}(\mathsf{gr}~\mathsf{mod}) \\
M \arrow[r,mapsto]
&\widehat{\mathscr{S}}^c(M)(-) \coloneqq \displaystyle \prod_{n \geq 0} \mathrm{Hom}_{\mathbb{S}_n}(M(n),(-)^{\otimes n})~.
\end{tikzcd}
\]

\begin{lemma}[{\cite[Corollary~3.4]{grignoulejay18}}]\label{lemma: Schur lax contravariant functor}
The functor $\widehat{\mathscr{S}}^c(-): (\mathbb{S}\textsf{-}\mathsf{mod}, \allowbreak \circ, \I)^{\mathsf{op}} \longrightarrow (\mathsf{End}(\kmod), \circ,\mathsf{Id})$ can be endowed with a lax monoidal structure, that is, there exists a natural transformation 
\[
\varphi_{M,N}: \widehat{\mathscr{S}}^c(M) \circ \widehat{\mathscr{S}}^c(N) \longrightarrow \widehat{\mathscr{S}}^c(M \circ N)~.
\]
which satisfies associativity and unitality compatibility conditions with respect to the monoidal structures. Furthermore, $\varphi_{M,N}$ is a monomorphism for all graded $\mathbb{S}$-modules $M,N$.
\end{lemma}

This implies that if $\mathcal{C}$ is a cooperad, then $\widehat{\mathscr{S}}^c(\mathcal{C})$ is monad. Thus one can define algebras over a cooperad in this way. 

\begin{Definition}[$\C$-algebra]
A $\mathcal{C}$\textit{-algebra} $B$ amounts to the data $(B,\gamma_B)$ of an algebra over the monad $\widehat{\mathscr{S}}^c(\mathcal{C})$.
\end{Definition}

\begin{Remark}
The notion of an algebra over a cooperad defines a new type of algebraic structures. The reason is that the structural map
\[
\gamma_B: \displaystyle \prod_{n \geq 0} \mathrm{Hom}_{\mathbb{S}_n}(\mathcal{C}(n),B^{\otimes n}) \longrightarrow B 
\]
associates a element in $B$ to any infinite series of operations. Thus algebras over a cooperad are endowed with a notion of infinite summation without presupposing any underlying topology. 
\end{Remark}

The notion of an algebra over a cooperad admits a further description in the case where the cooperad is \textit{conilpotent}. Our definition of a conilpotent cooperad is explained in Section \ref{subsection: filtrations on partial (co)operads}. 

\medskip

Let $(\C^u,\Delta,\epsilon)$ be a conilpotent cooperad. Each term of the coradical filtration $\mathscr{R}_\omega \C$ defines a sub-cooperad, and there is a short exact sequence of $\mathbb{S}$-modules 
\[
\begin{tikzcd}
0 \arrow[r]
&\mathscr{R}_\omega \C \arrow[r,"\iota_\omega",hook]
&\C \arrow[r,"\pi_\omega"]
&\C / \mathscr{R}_\omega \C \arrow[r]
&0~.
\end{tikzcd}
\]

\begin{Definition}[Canonical filtration on a $\mathcal{C}$-algebra]
Let $(\C^u,\Delta,\epsilon)$ be a conilpotent cooperad and let $(B, \gamma_B)$ be a $\C$-algebra. The \textit{canonical filtration} of $B$ is the decreasing filtration of given by 
\[ 
\mathrm{W}_\omega B \coloneqq \mathrm{Im}\left(\gamma_B \circ \widehat{\mathscr{S}}^c(\pi_\omega)(\mathrm{id}_B): \widehat{\mathscr{S}}^c(\C / \mathscr{R}_\omega \C)(B) \longrightarrow B \right)
\]
where $\mathscr{R}_\omega \C$ denotes the $\omega$-th term of the coradical filtration, for all $\omega \geq 0~.$ Notice that we have 
\[
B = \mathrm{W}_0 B \supseteq \mathrm{W}_1 B \supseteq \mathrm{W}_2 B \supseteq \cdots \supseteq \mathrm{W}_\omega B \supseteq \cdots.
\]
\end{Definition}

\begin{Definition}[Completion of a $\C$-algebra]
Let $(B, \gamma_B)$ be a $\C$-algebra. Its \textit{completion} is given by 
\[ 
\widehat{B} \coloneqq \lim_{\omega \in \mathbb{N}} B/\mathrm{W}_\omega B~,
\]
where the limit is taken in the category of $\C$-algebras.
\end{Definition}

It comes equipped with a canonical morphism of $\C$-algebras $\varphi_B: B \longrightarrow \widehat{B}~.$ 

\begin{Definition}[Complete $\mathcal{C}$-algebra]
The $\C$-algebra $(B,\gamma_B)$ is said to be \textit{complete} if $\varphi_B$ is an isomorphism of $\C$-algebras. 
\end{Definition}

\begin{Proposition}[{\cite[Proposition~4.24]{grignoulejay18}}]
Let $(B, \gamma_B)$ be a $\C$-algebra. The canonical morphism 
\[
\varphi_B: B \longrightarrow \widehat{B}
\]
is an epimorphism.
\end{Proposition}
 
\begin{Remark}\label{Remark: varphi is an epimorphism}
Conceptually, this comes from the fact that any $\mathcal{C}$-algebras already carries a meaningful notion of infinite summation. Thus "nothing is added" when one applies the completion functor. On the other hand, the topology induced by the canonical filtration of a $\mathcal{C}$-algebra might not be Hausdorff. Meaning that the canonical morphism $\varphi_B$ might not be a monomorphism. This completion functor should be considered a \textit{sifting functor}.
\end{Remark}

\begin{Proposition}[{\cite[Proposition 4.21, Proposition 4.25]{grignoulejay18}}]\label{prop: complete C algebras are a reflexiv subcat}
Let $(\C^u,\Delta,\epsilon)$ be a conilpotent cooperad. Any free $\C$-algebra is complete. Furthermore, the category of complete $\C$-algebras forms a reflexive subcategory of the category of $\C$-algebras, where the reflector is given by the completion.
\end{Proposition}

\begin{Remark}
Contrary to $\mathcal{C}$-coalgebras, which are always conilpotent, there are examples of $\mathcal{C}$-algebras which are not complete. See for instance \cite[Section 4.5]{grignoulejay18}.
\end{Remark}

Even though that, for an operad $\PP$, its dual Schur functor $\widehat{\mathscr{S}}^c(\mathcal{P})$ fails to be a comonad, one can still define a notion of a coalgebra over an operad. 

\begin{Definition}[$\mathcal{P}$-coalgebra]
A $\mathcal{P}$-coalgebra $D$ amounts to the data $(D, \Delta_D)$ of a graded module $D$ endowed with a structural map
\[
\Delta_D: D \longrightarrow \displaystyle \prod_{n \geq 0} \mathrm{Hom}_{\mathbb{S}_n}(\mathcal{P}(n),D^{\otimes n})~,
\]
such that the following diagram commutes 
\[
\begin{tikzcd}[column sep=4.5pc,row sep=3pc]
D \arrow[r,"\Delta_D"] \arrow[d,"\Delta_D",swap] 
&\widehat{\mathscr{S}}^c(\PP)(D) \arrow[r,"\widehat{\mathscr{S}}^c(\mathrm{id})(\Delta_D)"]
&\widehat{\mathscr{S}}^c(\PP)(D) \circ \widehat{\mathscr{S}}^c(\PP)(D) \arrow[d,"\varphi_{\PP,\PP}(D)"] \\
\widehat{\mathscr{S}}^c(\PP)(D) \arrow[rr,"\widehat{\mathscr{S}}^c(\gamma)(\mathrm{id})"]
&
&\widehat{\mathscr{S}}^c(\PP \circ \PP)(D)~.
\end{tikzcd}
\]
\end{Definition}

\begin{Remark}
The data of a $\mathcal{P}$-coalgebra $D$ amounts to the data of a morphism of operads $\mathcal{P} \longrightarrow \mathrm{Coend}_D$, where $\mathrm{Coend}_D$ stands for the coendomorphisms operads of $D$. 
\end{Remark}

\begin{Notation}\label{not: le dirac}
Let $f: X \longrightarrow Y$ be a map of degree $0$ and $g: X \longrightarrow Y$ be a map of degree $p$. We denote
\[
\diracComb_n(f,g) \coloneqq \sum_{i=0}^n f^{\otimes i-1} \otimes g \otimes f^{\otimes n-i} : X^{\otimes n} \longrightarrow Y^{\otimes n}
\]
the resulting $\mathbb{S}_n$-equivariant map of degree $p$. Let $M$ be an graded $\mathbb{S}$-module. It induces first a map of degree $p$
\[
\bigoplus_{n \geq 0} M(n) \otimes_{\mathbb{S}_n} X^{\otimes n} \longrightarrow \bigoplus_{n \geq 0} M(n) \otimes_{\mathbb{S}_n} Y^{\otimes n}
\]
by applying $\mathrm{id}_M \otimes \diracComb_n(f,g)$ at each arity. By a slight abuse of notation, this map will be denoted by $\mathscr{S}(\mathrm{id}_M)(\diracComb(f,g))~.$ Likewise, it induces a second map of degree $p$
\[
\prod_{n \geq 0} \mathrm{Hom}_{\mathbb{S}_n}(M(n),X^{\otimes n}) \longrightarrow \prod_{n \geq 0} \mathrm{Hom}_{\mathbb{S}_n}(M(n),Y^{\otimes n})
\]
by applying $\mathrm{Hom}(id_M,\diracComb_n(f,g))$ at each arity. By a slight abuse of notation, this map will be denoted by $\widehat{\mathscr{S}}^c(\mathrm{id}_M)(\diracComb(f,g))~.$
\end{Notation}

\subsection{Partial definition of (co)operads}
In order to work with operads and cooperads \textit{themselves}, it is convenient to use another definition.

\begin{Definition}
Let $n,k \in \mathbb{N}$, for all $1 \leq i \leq n$, we define maps:
\begin{equation}\label{morph: rho i}
\circ_i: \mathbb{S}_n \times \mathbb{S}_k \longrightarrow \mathbb{S}_{n+k-1}~.
\end{equation}
For $(\tau,\sigma) \in \mathbb{S}_n \times \mathbb{S}_k$, $\sigma \circ_i \tau$ is given by the unique permutation in $\mathbb{S}_{n+k-1}$ which first acts as $\tau$ on the set $\{1,\cdots,n+k-1\} - \{i,\cdots, i+k-1\}$ and sends $\{i,\cdots, i+k-1\}$ to $\{\tau(i), \cdots, \tau(i) + k -1 \}$; then acts as $\sigma$ on the set $\{\tau(i), \cdots, \tau(i) + k -1 \}$.
\end{Definition}

\begin{Definition}[Partial operad]\label{def: partialoperad}
A \textit{partial operad} $(\PP, \{\circ_i\})$ is the data of a graded $\mathbb{S}$-module $\PP$ endowed with a family of \textit{partial composition maps} of degree $0$:
\[
\circ_i : \PP(n) \otimes \PP(k) \longrightarrow \PP(n+k-1)
\]
for $1 \leq i \leq n$, subject to the following conditions.

\begin{enumerate}

\item It satisfies the \textit{sequential} axiom: for $1 \leq i \leq n, 1 \leq j \leq k$, the following diagram commutes
\[
\begin{tikzcd}[column sep=4.5pc,row sep=3pc]
\PP(n) \otimes \PP(k) \otimes \PP(m) \arrow[r,"\circ_i ~\otimes~ \mathrm{id}_{\PP(m)}"] \arrow[d,"\mathrm{id}_{\PP(n)}~ \otimes~ \circ_j",swap] 
&\PP(n+k-1) \otimes \PP(m) \arrow[d,"\circ_{i+j-1} "] \\
\PP(n) \otimes \PP(k+m-1) \arrow[r,"\circ_i"]
&\PP(n+k+m-2)~.
\end{tikzcd}
\]
\item It satisfies the \textit{parallel} axiom: for $1 \leq i < j \leq n$, the following diagram commutes

\[
\begin{tikzcd}[column sep=4.5pc,row sep=3pc]
\PP(n) \otimes \PP(k) \otimes \PP(m) \arrow[d,"\circ_i  ~\otimes~ \mathrm{id}_{\PP(m)}",swap] \arrow[r,"\circ_j ~\otimes~ \mathrm{id}_{\PP(k)}"]
&\PP(n+m-1) \otimes \PP(k) \arrow[d,"\circ_i"] \\
\PP(n+k-1) \otimes \PP(m) \arrow[r,"\circ_{j+k-1}"]
&\PP(n+k+m-2)~.
\end{tikzcd}
\]

\item The maps
\[
\circ_i : \PP(n) \otimes \PP(k) \longrightarrow \PP(n+k-1)
\]
satisfy the following condition: let $(\tau,\sigma) \in \mathbb{S}_n \times \mathbb{S}_k$, then the following diagram commutes
\[
\begin{tikzcd}[column sep=3pc,row sep=3pc]
\PP(n) \otimes \PP(k) \arrow[r,"\circ_i "] \arrow[d,"\tau ~\otimes~\sigma ",swap]
&\PP(n+k-1) \arrow[d,"\tau ~\circ_i~ \sigma " ] \\
\PP(n) \otimes \PP(k) \arrow[r,"\circ_{\tau(i)} "]
&\PP(n+k-1)~.
\end{tikzcd}
\]
\end{enumerate} 
\end{Definition}

\begin{Definition}[Unital partial operad]\label{def: unital partial operad}
A \textit{unital partial operad} $(\mathcal{P},\{\circ_i\},\eta)$ is the data of a partial operad $(\mathcal{P},\{\circ_i\})$ together with a morphism $\eta: \I \longrightarrow \PP$, such that $\circ_1(\eta(\mathrm{id}),-)$ and $\circ_i(-,\eta(\mathrm{id}))$ are the identity on $\PP(n)$.
\end{Definition}

\begin{Remark}
Unital partial operads are sometimes called Markl operads in the literature, see \cite{Markl96}.
\end{Remark}

Let $\mathscr{T}$ denote the tree monad on the category of graded $\mathbb{S}$-modules and let $\overline{\mathscr{T}}$ denote the reduced tree monad on the same category. For an explicit construction, see \cite[Section 5.6]{LodayVallette12}. 

\begin{Proposition}
There are isomorphisms of categories between:
\begin{enumerate}
\item The category of algebras over the tree monad $\mathscr{T}$ and the category of unital partial operads.

\item The category of algebras over the reduced tree monad $\overline{\mathscr{T}}$ and the category of partial operads.
\end{enumerate}
\end{Proposition}

\begin{Proposition}\label{prop: unital partial operads are operads}
The category of operads defined as monoids is equivalent to the category of unital partial operads. 
\end{Proposition}

\begin{Remark}
Given a partial operad $(\mathcal{P},\{\circ_i\})$, one can freely adjoin a unit by setting $\mathcal{P}^u \coloneqq \mathcal{P} \oplus \I$ and considering the inclusion $\eta: \I \longrightarrow \mathcal{P}^u$. This allows us to associate to any partial operad an augmented operad defined as a monoid.
\end{Remark}

\begin{Definition}[Partial cooperad]\label{def: partialcoop}
A \textit{partial cooperad} $(\C, \{\Delta_i\})$ is the data of a graded $\mathbb{S}$-module $\C$ together with \textit{partial decomposition maps}: 
\[
\Delta_i: \C(n+k-1) \longrightarrow \C(n) \otimes \C(k)
\]
for $1 \leq i \leq n$, subject to the following conditions.

\begin{enumerate}
\item It satisfies the \textit{sequential} axiom: for $1 \leq i \leq n, 1 \leq j \leq k$, the following diagram commutes
\[
\begin{tikzcd}[column sep=4.5pc,row sep=3pc]
\C(n+k+m-2) \arrow[r,"\Delta_{i+j-1}"] \arrow[d,"\Delta_i ",swap]
&\C(n+k-1) \otimes \C(m) \arrow[d,"\Delta_i  \hspace{1pt} \otimes  \hspace{1pt} \mathrm{id}_{\C(m)}"] \\
\C(n) \otimes \C(k+n-1) \arrow[r,"\mathrm{id}_{\C(n)}  \hspace{1pt} \otimes  \hspace{1pt} \Delta_j "]
&\C(n) \otimes \C(k) \otimes \C(m)~.
\end{tikzcd}
\]
\item It satisfies the \textit{parallel} axiom: for $1 \leq i < j \leq n$, the following diagram commutes
\[
\begin{tikzcd}[column sep=4.5pc,row sep=3pc]
\C(n+k+m-2) \arrow[r,"\Delta_{j+k-1}"] \arrow[d,"\Delta_i ",swap]
&\C(n+k-1) \otimes \C(m) \arrow[d,"\Delta_i \hspace{1pt} \otimes  \hspace{1pt} \mathrm{id}_{\C(m)}"] \\
\C(n+m-1) \otimes \C(k) \arrow[r,"\Delta_j  \hspace{1pt} \otimes  \hspace{1pt} \mathrm{id}_{\C(k)}"]
&\C(n) \otimes \C(k) \otimes \C(m)~.
\end{tikzcd}
\]
\item The maps
\[
\Delta_i: \C(n+k-1) \longrightarrow \C(n) \otimes \C(k)
\]
satisfy the following condition: let $(\tau,\sigma) \in \mathbb{S}_n \times \mathbb{S}_k$, then the following diagram commutes
\[
\begin{tikzcd}[column sep=3pc,row sep=3pc]
\C(n+k-1) \arrow[r,"\Delta_{\tau^{-1}(i)} "] \arrow[d,"\tau ~\circ_i~\sigma ",swap]
&\C(n) \otimes \C(k)  \arrow[d,"\tau ~ \otimes ~ \sigma "]\\
\C(n+k-1) \arrow[r,"\Delta_i "]
&\C(n) \otimes \C(k)~.
\end{tikzcd}
\] 
\end{enumerate} 
\end{Definition}

\begin{Definition}[Counital partial cooperad]
A \textit{counital partial cooperad} $(\C, \{\Delta_i\},\epsilon)$ is the data of a partial cooperad $(\C, \{\Delta_i\})$ and a morphism of $\mathbb{S}$-modules $\epsilon: \C \longrightarrow \I$ such that:
\[
\begin{tikzcd}
\C(n) \arrow[r,"\Delta_i"] \arrow[rd,"\cong",swap]
&\C(n) \otimes \C(1) \arrow[d,"\mathrm{id}~\otimes~\epsilon"]
&\C(n) \arrow[r,"\Delta_1"] \arrow[rd,"\cong",swap]
&\C(1) \otimes \C(n) \arrow[d,"\epsilon~\otimes~\mathrm{id}"] \\
&\C(n) \cong \C(n) \otimes \kk
&
&\C(n) \cong  \kk \otimes \C(n)~.
\end{tikzcd}
\]
\end{Definition}

\subsection{Filtrations on partial (co)operads}\label{subsection: filtrations on partial (co)operads}
In order to compare the notion of a (possibly counital) partial cooperad with the notion of a cooperad defined as a comonoid, one needs to introduce filtrations. The tree monad has a natural weight grading $\mathscr{T}^{(\omega)}$ given by the number of internal edges $\omega$ of the rooted trees. Notice that: 
\[
\mathscr{T}^{(-1)}(M) = |~, \quad \mathscr{T}^{(0)}(M) = M~, \quad \text{and}~~ \mathscr{T}^{(1)}(M) = M \circ_{(1)} M~.
\]
Recall that the tree monad on a graded $\mathbb{S}$-module $\mathscr{T}(M)$ is given by the \textit{direct sum} over rooted trees where vertices are labeled by elements of $M$ in the following way: if $\tau$ is a rooted tree and $v$ is one of its vertices with $k$ incoming edges, then $v$ must be labeled with an element of $M(k)$. See for more details \cite[Section 5.6.1]{LodayVallette12}.

\begin{Definition}[Reduced completed tree monad]
The \textit{reduced completed tree endofunctor} $\overline{\mathscr{T}}^\wedge$ is the endofunctor of the category of graded $\mathbb{S}$-modules given by
\[
\overline{\mathscr{T}}^\wedge(M) \coloneqq \lim_{\omega} \overline{\mathscr{T}}(M)/\overline{\mathscr{T}}^{(\geq \omega)}(M)~,
\]
that is, the completion of the tree monad with respect to this weigh filtration. 
\end{Definition}

\begin{Remark}
The reduced complete tree endofunctor is isomorphic to the \textit{product} over rooted trees where vertices are labeled by elements of $M$ in the following way: if $\tau$ is a rooted tree and $v$ is one of its vertices with $k$ incoming edges, then $v$ must be labeled with an element of $M(k)$. In particular, if $M(0)=0$ and $M(1)=0$, then there are only a finite number of rooted trees in each arity and one has that 
\[
\overline{\mathscr{T}}^\wedge(M) \cong \overline{\mathscr{T}}(M)~.
\]
\end{Remark}

\begin{lemma}\label{Existence of delta C}
Let $(\C,\{\Delta_i\})$ be a partial cooperad. The partial decomposition maps $\{\Delta_i\}$ induce a morphism of graded $\mathbb{S}$-modules
\[
\Delta_\C: \C \longrightarrow \overline{\mathscr{T}}^{\wedge}(\C)~.
\]
\end{lemma}

\begin{proof}
The map $\Delta_\C$ is defined as follows: it maps to an operation $\mu \in \C(n)$ the sum of all rooted trees associated to the iterated partial decompositions of this operation. This sum might not be finite, but since the target is complete, this gives a well defined map of graded $\mathbb{S}$-modules.
\end{proof}

\begin{Definition}[Coradical filtration of a partial cooperad]\label{coradical filtration coop}
Let $(\C, \{ \Delta_i \})$ be a partial cooperad. The \textit{coradical filtration} of $\C$ is the increasing filtration given by 
\[
\mathscr{R}_\omega \C \coloneqq \mathrm{Ker}\left(\Delta_\C^{(\geq \omega)}: \C \longrightarrow (\overline{\mathscr{T}}^{\wedge})^{(\geq \omega)}(\C) \right)
\]
for all $\omega \geq 0$, where $(\overline{\mathscr{T}}^{\wedge})^{(\geq \omega)}$ denotes the elements of weight greater or equal to $\omega$. For all $\omega \geq 0$, $\mathscr{R}_\omega \C$ is given by operations whose $(\omega + 1)$-iterated partial decompositions are equal to zero, and it defines a partial sub-cooperad of $\C$. This gives a diagram
\[
0 = \mathscr{R}_{0} \C \subseteq \mathscr{R}_{1} \C \subseteq \cdots \subseteq \mathscr{R}_{\omega} \C \subseteq \cdots
\]
\end{Definition} 

\begin{Definition}[Conilpotent partial cooperad]\label{conilpartcoop}
Let $(\C, \{\Delta_i \})$ be a partial cooperad. It is \textit{conilpotent} if the canonical map from the colimit
\[ 
\psi: \colim_{\omega} \mathscr{R}_{\omega} \C \longrightarrow \C 
\]
is an isomorphism of partial cooperads. This is equivalent to the fact that the coradical filtration is exhaustive. 
\end{Definition}

The underlying endofunctor of the reduced tree monad $\overline{\mathscr{T}}$ has a comonad structure $(\Delta_{\overline{\mathscr{T}}},\epsilon_{\overline{\mathscr{T}}})$, where $\Delta_{\overline{\mathscr{T}}}$ is given by partitioning trees and $\epsilon_{\overline{\mathscr{T}}}$ by the inclusion of the corollas. See \cite[Section 5.8.8]{LodayVallette12} for more details. The endofunctor $\overline{\mathscr{T}}$ together with its comonad structure will be denoted $\overline{\mathscr{T}}^c$, and called \textit{the reduced tree comonad}.

\begin{Proposition}\label{prop: reduced tree comonad gives conil part coop}
The category of conilpotent partial cooperads is isomorphic to the category of coalgebras over the reduced tree comonad $\overline{\mathscr{T}}^c$.
\end{Proposition}

\begin{proof}
Given a partial cooperad $(\C,\{\Delta_i\})$, the morphism $\Delta_\C$ given in Lemma \ref{Existence of delta C} factors through the $\overline{\mathscr{T}}^c(\C)$ if and only if the coradical filtration is exhaustive. The axioms of a partial cooperad imply that this map defines a coalgebra structure on $\overline{\mathscr{T}}^c(\C)$. The other way around is straightforward.
\end{proof}

\begin{Proposition}
Let $(\C,\{\Delta_i\})$ be a conilpotent partial cooperad. The partial decompositions $\{\Delta_i\}$ induce a coassociative morphism of graded $\mathbb{S}$-modules $\Delta: \C \longrightarrow \C \circ \C$.
\end{Proposition}

\begin{proof}
Let $(\C,\{\Delta_i\})$ be a conilpotent partial cooperad. Then it is a coalgebra over the reduced tree comonad $\overline{\mathscr{T}}^c$. The morphism of $\mathbb{S}$-modules
\[
\Delta_\C: \C \longrightarrow \overline{\mathscr{T}}^c(\C)
\]
can be projected onto the sub-$\mathbb{S}$-module $\C \circ \C \hookrightarrow \overline{\mathscr{T}}^c(\C)$. The axioms of a partial cooperad ensure that this projection is coassociative. 
\end{proof}

\begin{Corollary}\label{corollary: adding a counit}
Let $(\C,\{\Delta_i\})$ be a conilpotent partial cooperad. Let $\C^u \coloneqq \C \oplus \I$ and $\epsilon$ be the projection onto $\I$. Then $(\C^u, \Delta, \epsilon)$ forms a cooperad defined as a comonoid. It defines a functor 
\[
\mathsf{Conil}: \mathsf{gr}~\mathsf{pCoop}^\mathsf{conil} \longrightarrow \mathsf{gr}~\mathsf{Coop}
\]
from the category of conilpotent partial cooperads to the category of cooperads defined as comonoids.
\end{Corollary}

\begin{Remark}
Counital partial cooperads do not, in general, induce a cooperad defined as a comonoid. For example, the graded $\mathbb{S}$-module $u\mathcal{C}om^*$ given by the linear dual of the operad $u\mathcal{C}om$ admits a counital partial cooperad structure but does not admit a cooperad structure.
\end{Remark}

\begin{Definition}[Conilpotent cooperad]\label{def: conilpotent cooperad}
Let $(\C,\Delta,\epsilon)$ be a cooperad. It is said to be \textit{conilpotent} if it is in the essential image of the functor $\mathsf{Conil}$ defined in Corollary \ref{corollary: adding a counit}. 
\end{Definition}

\begin{Remark}
This definition of conilpotency for a cooperad defined as a comonoid $(\C,\Delta,\epsilon)$ is not standard. It is equivalent to the definition of conilpotency given in \cite[Definition 4.3]{grignoulejay18}, but does not coincide with the definition given in \cite[Section 5.8.5]{LodayVallette12}, although it turns out to be equivalent.
\end{Remark}

\begin{Remark}
For a similar discussion concerning coproperads, see \cite[Section 2.2]{hoffbeck2019properadic}.
\end{Remark}

Now we look at the dual picture: the canonical filtrations one can induce on partial operads. These are induced by the weight filtration on the reduced tree monad $\overline{\mathscr{T}}$.

\begin{Definition}[Canonical filtration on a partial operad] 
Let $(\PP,\{\circ_i\})$ be a partial operad and let $\gamma_\PP: \overline{\mathscr{T}}(\PP) \longrightarrow \PP$ be its $\overline{\mathscr{T}}$-algebra structure. Its \textit{canonical filtration} is the decreasing filtration given by
\[ 
\mathscr{F}_{\omega} \PP \coloneqq \mathrm{Im}\left(\gamma_\PP^{(\geq \omega)}: \overline{\mathscr{T}}^{(\geq \omega)}(\PP) \longrightarrow \PP \right)
\]
for all $\omega \geq 0$, where $\overline{\mathscr{T}}^{(\geq \omega)}(\PP)$ denotes the elements of weight greater or equal to $\omega~.$ Each $\mathscr{F}_{\omega} \PP$ defines an operadic ideal of $\PP$ since they are stable under partial composition. Notice that: 
\[
\PP = \mathscr{F}_{0} \PP \supseteq \mathscr{F}_{1} \PP \supseteq \cdots \supseteq \mathscr{F}_{\omega} \PP \supseteq \cdots.
\]
\end{Definition}

\begin{Definition}[Nilpotent partial operad]
Let $(\PP,\{\circ_i\})$ be a partial operad. It is said to be \textit{nilpotent} if there exists an $\omega \geq 1$ such that 
\[
\PP/\mathscr{F}_{\omega} \PP \cong \PP~.
\]
The partial operad is said to be $\omega_0$\textit{-nilpotent} if $\omega_0$ is the smallest integer such that the above isomorphism exits.
\end{Definition}

\begin{Definition}[Completion of a partial operad]\label{def: completion functor for operads}
Let $(\PP,\{\circ_i\})$ be a partial operad, its \textit{completion} $\widehat{\PP}$ is given by the following limit
\[
\widehat{\PP} \coloneqq \lim_{\omega} \PP/\mathscr{F}_{\omega} \PP 
\]
taken in the category of partial operads. 
\end{Definition}

Any morphism of partial operads $f: (\PP,\{\circ_i\}_\PP) \longrightarrow (\Q,\{\circ_i\}_\Q)$ is continuous with respect to the canonical filtration (i.e $f(\mathscr{F}_{\omega} \PP) \subseteq \mathscr{F}_{\omega} \Q$). Thus the completion of partial operads is functorial. For every $\omega \geq 1$, there are projection morphisms of partial operads $\varphi_\omega: \PP \twoheadrightarrow \PP/\mathscr{F}_{\omega} \PP$ which induce a canonical morphism of partial operads: 

\[
\varphi: \PP \longrightarrow \widehat{\PP}. 
\]
\begin{Definition}[Complete partial operad]\label{def: complete operads}
Let $(\PP,\{\circ_i\})$ be a partial operad. It is \textit{complete} if the canonical morphism $\varphi: \PP\longrightarrow \widehat{\PP}$ is an isomorphism of partial operads. 
\end{Definition} 

\begin{Example}
Any nilpotent partial operad is complete. Any complete partial operad is the limit of a tower of nilpotent partial operads.
\end{Example}

\subsection{Convolution partial operad}
Finally, we introduce the convolution partial operad and recall what a twisting morphism is. 

\begin{Definition}[Convolution partial operad]
Let $(\C,\{\Delta_i\})$ be a partial cooperad and let $(\mathcal{P},\{\circ_i\})$ be a partial operad. The \textit{convolution partial operad} of $\C$ and $\PP$ is given by the graded $\mathbb{S}$-module 
\[
\mathcal{H}om(\C,\PP)(n) \coloneqq \mathrm{Hom}_{\mathsf{gr}~\mathsf{mod}}(\C(n), \PP(n))~,
\]
with its natural $\mathbb{S}_n$-action. It can be endowed with the partial composition given by
\[
\begin{tikzcd}[column sep=3pc,row sep=0pc]
\alpha \circ_i \beta : \C(n+k-1) \arrow[r,"\Delta_{i}"]
&\C(n) \otimes \C(k) \arrow[r,"\alpha \hspace{1pt} \otimes \hspace{1pt} \beta"]
&\PP(n) \otimes \PP(k) \arrow[r,"\circ_i"]
&\PP(n+k-1)~.
\end{tikzcd}
\]
\end{Definition}

\begin{Notation}
We denote by $\mathrm{OrPar}(i_1, \cdots, i_k)$ the set of ordered partitions of $\{1,\cdots, n\}$, where $n = i_1 + \cdots + i_k$, whose $j$-th part has $i_j$ elements. Any partition $P$ in $\mathrm{OrPar}(i_1, \cdots, i_k)$ defines a $(i_1, \cdots, i_k)$-unshuffle, see \cite[Section 5.4.2]{LodayVallette12}. 
\end{Notation}

\begin{Definition}[Totalization of a partial operad]
Let $(\PP,\{\circ_i\})$ be a partial operad, the \textit{totalization} of $\PP$ given by 
\[
\prod_{n \geq 0} \PP(n)^{\mathbb{S}_n}~.
\]
It can be endowed with a pre-Lie algebra structure by setting
\[
\mu \star \nu \coloneqq \sum_{i=1}^n \sum_{\sigma_P} (\mu \circ_i \nu)^{\sigma_P}~,
\]
where $\mu$ is in $\PP(n)$, $\nu$ is in $\PP(m)$, and where the second sum ranges over unshuffles $\sigma_P$ associated to all ordered partitions $P$ in $\mathrm{OrPar}(1,\cdots, 1, \underbrace{n-i+1}_{i\text{-th position}}, 1, \cdots, 1)$.
\end{Definition}

\begin{Remark}
One can check by direct computation that the axioms of a partial operad make the associator of $\star$ right symmetric. We refer to \cite[Section 5.4.3]{LodayVallette12} for more details on the totalization of a partial operad.
\end{Remark}

\begin{Definition}[Twisting morphism]
Let $(\C,\{\Delta_i\})$ be a partial cooperad and let $(\mathcal{P},\{\circ_i\})$ be a partial operad. Let 
\[
\mathfrak{g}_{\C,\PP} \coloneqq \prod_{n \geq 0} \mathrm{Hom}_{\mathbb{S}_n}(\C(n), \PP(n))
\]
be the pre-Lie algebra given by the totalization of the convolution partial operad of $\C$ and $\PP$. A \textit{twisting morphism} is a Maurer-Cartan element $\alpha$ of $\mathfrak{g}_{\C,\PP}$, that is, a morphism $\alpha: \C \longrightarrow \PP$ graded $\mathbb{S}$-modules of degree $-1 $ satisfying :
\[
\alpha \star \alpha = 0~.
\]
The set of twisting morphisms between $\C$ and $\PP$ is denoted by $\mathrm{Tw}(\C,\PP)~.$
\end{Definition}

\begin{Remark}
Everything stated in this section can be generalized \textit{mutatis mutaindis} to the context of differential graded modules.
\end{Remark}

\section{Curved operads and their curved algebras}\label{Section: Curved Operads}
In the section, we introduce a new type of operad-like structure, called curved operads. They naturally encode curved algebras. The proper framework for this theory will be the underlying category of pre-differential graded modules. This framework, and the existence of the curved endomorphisms operad, was discovered independently of \cite{JoanCurved}.

\begin{Definition}[Pre-differential graded module]
A \textit{pre-differential graded module} (pdg module for short) $(V,d_V)$ is the data of a graded module $V$ together with a linear map $d_V: V \longrightarrow V$ of degree $-1$. A \textit{morphism} $f: (V,d_V) \longrightarrow (W,d_W)$ is a morphism of graded modules $f:V \longrightarrow W$ that commutes with the pre-differentials.
\end{Definition}

Pre-differential graded modules form a symmetric monoidal category $(\mathsf{pdg}~\mathsf{mod}, \otimes, \mathbb{K})~,$ where the pre-differential on the graded tensor product $A \otimes B$ is given by 
\[
d_{A \otimes B} (a \otimes b) \coloneqq d_A(a) \otimes b + (-1)^{|a|}\hspace{1pt} a \otimes d_B(b)~.
\]
The symmetric monoidal category $\mathsf{pdg}~\mathsf{mod}$ will be our \textit{underlying} category. Notice that we \textbf{do not} ask the condition $d_V^2 =0$. The category of dg modules is a full subcategory of the category of pdg modules.

\medskip

\subsection{First definitions} The data of a pdg operad is the data $(\PP,\gamma,\eta,d_\PP)$ of a graded operad $(\PP,\gamma,\eta)$ together with a derivation $d_\PP$ of degree $-1$. Similarly, a pdg partial operad amounts to a graded partial operad $(\PP,\{\circ_i\})$ equipped with a degree $-1$ derivation $d_\PP$. Again, a pdg cooperad $(\C,\Delta, \epsilon, d_\C)$ amounts to a graded cooperad $(\C,\Delta, \epsilon)$ with a coderivation $d_\C$ of degree $-1$ and a pdg partial cooperad $(\C,\{\Delta_i\},d_\C)$ amounts to a graded partial cooperad $(\C,\{\Delta_i\})$ with a coderivation $d_\C$ of degree $-1$. The notions of conilpotentcy for partial cooperads or completeness for partial operads are the same in this framework.

\begin{Notation}
Since we already use $\circ$ for the composition product of $\mathbb{S}$-modules, we will sometimes denote the composition of two morphisms by $\cdot$ when the context might be ambiguous.
\end{Notation}

\begin{Definition}[Curved operad]
A \textit{curved operad} $(\mathcal{P},\gamma,\eta,d_\PP, \Theta_\mathcal{P})$ amounts to the data of a pdg operad $(\PP,\gamma,\eta,d_\PP)$ together with a morphism of pdg $\mathbb{S}$-modules $\Theta: (\I,0) \longrightarrow (\PP,d_\PP)$ of degree $-2$, such that the following diagram commutes: 
\[
\begin{tikzcd}[column sep=7pc,row sep=3pc]
\PP \arrow[r,"\mathrm{diag}"] \arrow[rrd,"(d_\PP)^2", bend right =10]
&\PP \oplus \PP \cong (\I \circ \PP) \oplus (\PP \circ \I) \arrow[r,"(\Theta_\PP ~ \circ~ \mathrm{id})~ -~(\mathrm{id}~ \circ'~ \Theta_\PP)"] 
& \PP \circ_{(1)} \PP \arrow[d,"\gamma_{(1)}"]\\
&
&\PP~,
\end{tikzcd}
\]
where $\mathrm{diag}$ is given by $\mathrm{diag}(\mu) \coloneqq (\mu,\mu)$. 
\end{Definition}

The data of the morphism of pdg $\mathbb{S}$-modules $\Theta_\PP$ is equivalent to a element $\Theta_\PP(\mathrm{id}) \coloneqq \theta$ in $\PP(1)_{-2}$ such that $d_\PP(\theta)=0$. The commutativity of the diagram amounts to the following condition: for every $\mu$ in $\PP(n)$, we ought to have
\[
\vcenter{\hbox{
\begin{tikzpicture}
\node at (3.25,1.5) {$d_\PP^2(\mu)$};

\node at (4.5,1.5) {$=$};

\node at (10,1.5) {$\displaystyle - \sum_{i=1}^n$};

\draw(13,-0.5)--(13,0.5);
\draw(13,0.5)--(12.5,1.5);
\draw(13,0.5)--(13.5,1.5);
\draw(13,0.5)--(13,2.5);
\draw(13,0.5)--(11.5,1.5);
\draw(13,0.5)--(14.5,1.5);
\node at (13,2) {$\bullet$};
\node at (13.3,0.4) {$\mu$};
\node at (13.3,2) {$\theta$};
\node at (13,2.75) {$i$};

\node at (12.5,1.7) {$\cdots$};
\node at (13.5,1.7) {$\cdots$};
\node at (11.5,1.7) {$1$};
\node at (14.5,1.7) {$n$};

\draw(7,-0.5)--(7,0.5);
\draw(7,0.5)--(6.5,1.5);
\draw(7,0.5)--(8.5,1.5);
\draw(7,0.5)--(7,1.5);
\draw(7,0.5)--(7.5,1.5);
\draw(7,0.5)--(5.5,1.5);
\draw(7,0.5)--(8.5,1.5);
\node at (7,1.7) {$\cdots$};
\node at (5.5,1.7) {$1$};
\node at (8.5,1.7) {$n$};
\node at (7.3,0.4) {$\mu$};
\node at (7.3,-0.2) {$\theta$};
\node at (7,-0.2) {$\bullet$};
\end{tikzpicture}.}}
\]

This states that the square of the pre-differential $d_\PP^2$ is equal to the operadic Lie bracket $[\theta,-]$ on the totalization of $\PP$. By a slight abuse of notation, we denote this equality by 
\[
d_\PP^2 = \gamma_{(1)} \cdot (\Theta_\PP \circ \mathrm{id} - \mathrm{id} \circ' \Theta_\PP)~,
\]
forgetting $\mathrm{diag}$ and the identifications made.

\begin{Notation}
A curved operad $(\PP,\gamma,\eta,d_\PP,\Theta_\PP)$ will be denoted by $(\PP,d_\PP,\Theta_\PP)$ for short, making the composition map and the unit implicit. When regarded as an element of $\mathcal{P}(1)$, the curvature is denoted $\theta$. Likewise for pdg operads, which will be sometimes denoted by $(\PP,d_\PP)$ when the context is clear.
\end{Notation}

\begin{Definition}
A \textit{morphism} of curved operads $f: (\PP,d_\PP,\Theta_\PP) \longrightarrow \allowbreak (\D,d_\D,\Theta_\D)$ amounts to the data of a pdg operad morphism $f: (\PP,d_\PP) \longrightarrow (\D,d_\D)$ that preserves the curvatures $f \circ \Theta_\PP = \Theta_\D$. 
\end{Definition}

Let $(V,d_V)$ be a pdg module, we consider the endomorphism pdg $\mathbb{S}$-module given by: 
\[
\mathrm{End}_V(n) \coloneqq \mathrm{Hom}(V^{\otimes n},V)~.
\]
Endowed with the composition of functions, and the pre-differential $\partial \coloneqq [d_V,-]~,$ it forms a pdg operad, called the \textit{endomorphism operad} of $V$.

\begin{lemma}\label{lemma: endo curved is curved}
Equipped with the curvature $\Theta_V: \I \longrightarrow \text{End}_V$ given by $\Theta_V(\mathrm{id}) = d_V^2$, the data $(\text{End}_V,\partial, \Theta_V)$ forms a curved operad, called the \textit{curved endomorphism operad} associated to $(V,d_V)$. 
\end{lemma}

\begin{proof}
It is straightforward to check that $\partial_V^2=[d_V^2,-]$ and that $[d_V,d_V^2]=0$.
\end{proof}

\begin{Definition}[Pdg $\PP$-algebra]
Let $(\PP,d_\PP)$ be a pdg operad. A \textit{pdg} $\PP$-\textit{algebra} $(A,\gamma_A,d_A)$ amounts to the data of a graded $\PP$-algebra structure $\gamma_A: \mathscr{S}(\PP)(A)\longrightarrow A$ on a pdg module $(A,d_A)$ such that 
\begin{equation}\label{derivation de P alg}
\gamma_A \cdot \left( \mathscr{S}(d_\PP)(\mathrm{id}) + \mathscr{S}(\mathrm{id})(\diracComb(\mathrm{id},d_A))\right) = d_A \cdot \gamma_A~. 
\end{equation}
A \textit{morphism} of pdg $\PP$-algebras $g: (A,\gamma_A,d_A) \longrightarrow (B,\gamma_B,d_B)$ amounts to a morphism of pdg modules $f:(A,d_A) \longrightarrow (B,d_B)$ that commutes with the $\PP$-algebra structures. The category of pdg $\PP$-algebras is denoted by $\mathsf{pdg}~\palg$.
\end{Definition}

\begin{Remark}
The notation $\mathscr{S}(\mathrm{id})(\diracComb(\mathrm{id},d_A))$ is explained in \ref{not: le dirac}.
\end{Remark}

\begin{Definition}[Curved $\PP$-algebra]
Let $(\PP,d_\PP,\Theta_\mathcal{P})$ be a curved operad. A \textit{curved} $\PP$-\textit{algebra} $(A,\gamma_A,d_A)$ amounts to the data of a pdg $\PP$-algebra structure $(A,\gamma_A,d_A)$ such that the following diagram commutes:
\begin{equation}\label{curved}
\begin{tikzcd}[column sep=5pc,row sep=2pc]
\mathscr{S}(\I)(A) \arrow[r,"\mathscr{S}(\Theta_\PP)(\mathrm{id})"] \arrow[rd, "d_A^2",swap]
&\mathscr{S}(\PP)(A) \arrow[d,"\gamma_A"] \\
&A~.
\end{tikzcd}
\end{equation}
Otherwise stated, we have that $\gamma_A(\theta,a) = d_A^2(a)$ for any $a$ in $A$. Morphisms of curved $\PP$-algebras are just morphisms of pdg $\PP$-algebras. The category of curved $\PP$-algebras will be denoted $\mathsf{curv}~\PP\text{-}\mathsf{alg}~.$
\end{Definition}

\begin{Remark}
Notice that a curved operad $(\PP,d_\PP,\Theta_\mathcal{P})$ is the data of a pdg operad $(\PP,d_\PP)$ together with extra \textit{structure} given by the curvature $\Theta_\mathcal{P}~.$ On the other hand, the data of a curved algebra structure over a curved operad is the same as the data of a pdg algebra structure over its underlying pdg operad, but this structure has to satisfy an extra \textit{property} given by diagram \ref{curved}.
\end{Remark}

There is an obvious inclusion functor $\mathsf{Inc}: \mathsf{curv}~\PP\text{-}\mathsf{alg} \hookrightarrow \mathsf{pdg}~\palg$ which is fully faithful.

\begin{Proposition}[{\cite[Proposition C.31]{JoanCurved}}]\label{reflexive}
Let $(\PP,d_\PP,\Theta_\mathcal{P})$ be a curved operad. The inclusion functor has a left adjoint 
\[
\mathsf{Curv}: \mathsf{pdg}~\palg \longrightarrow \mathsf{curv}~\PP\text{-}\mathsf{alg}~.
\]
Hence $\mathsf{curv}~\PP\text{-}\mathsf{alg}$ is a reflexive subcategory of $\mathsf{pdg}~\palg$. For a pdg $\PP$-algebra $(A,\gamma_A,d_A)$, its image under this functor is given by the following quotient: 
\[
\mathsf{Curv}(A) \coloneqq \frac{A}{(\mathsf{Im}(\gamma_A(\theta,-) - d_A^2(-)))}~.
\]
where $(\mathsf{Im}(\gamma_A(\theta,-) - d_A^2(-)))$ is the ideal generated by $\mathsf{Im}(\gamma_A(\theta,-) - d_A^2(-))$. Its pdg $\PP$-algebra structure is induced by $\gamma_A$ and $d_A$.
\end{Proposition}

\begin{proof}
First, lets show that $\mathsf{Im}(\gamma_A(\theta,-) - d_A^2(-))$ is stable under the pre-differential $d_A$. We have that $d_A(\gamma_A(\theta,a) - d_A^2(a)) = \gamma_A(d_\PP(\theta),a) + \gamma_A(\theta,d_A(a)) - d_A^3(a)~,$ which is again in $\mathsf{Im}(\gamma_A(\theta,-) - d_A^2(-))$ since $\gamma_A(d_\PP(\theta),a)=0~.$ Hence the quotient forms a pdg $\PP$-algebra, which is curved by definition. 

\medskip

Let $A$ be a pdg $\PP$-algebra and $B$ be a curved $\PP$-algebra. Since the pdg $\PP$-algebra structure on $B$ satisfies diagram \ref{curved}, any pdg $\PP$-algebra morphism $f: A \longrightarrow B$ factors through $\mathsf{Curv}(A)$. This gives the adjunction isomorphism by the universal property of the quotient. 
\end{proof}

\begin{Corollary}\label{bicomplete}
Let $(\PP,d_\PP,\Theta_\mathcal{P})$ be a curved operad. The category of curved $\PP$-algebras is complete and cocomplete. 
\end{Corollary}

\begin{proof}
The category $\mathsf{pdg}~\palg$ is both complete and cocomplete, since $\mathsf{curv}~\PP\text{-}\mathsf{alg}$ is a reflexive subcategory of $\mathsf{pdg}~\palg$ by Proposition \ref{reflexive}.
\end{proof}

\begin{lemma}
Let $(\PP,d_\PP,\Theta_\mathcal{P})$ be a curved operad and let $(A,d_A)$ be a pdg module. The data of a curved $\PP$-algebra structure $\gamma_A$ on $(A,d_A)$ is equivalent to the data of a morphism of curved operads $\Gamma_A :(\PP,d_\PP,\Theta_\PP) \longrightarrow (\mathrm{End}_A, \partial, \Theta_{A})~.$
\end{lemma}

\begin{proof}
The data of a graded $\PP$-algebra structure $\gamma_A: \mathcal{S}(\PP)(A) \longrightarrow A$ is equivalent to the data of a morphism of graded operads $\Gamma_A : \PP \longrightarrow \mathrm{End}_A~.$ One can check that condition of equation \ref{derivation de P alg} is equivalent to $\Gamma_A$ being a morphism of pdg operads. The commutativity of the diagram \ref{curved} means that $\gamma_A (\mathscr{S}(\Theta_\PP)(\mathrm{id})) = d_A^2~,$ which is equivalent to $\Gamma_A (\Theta_\PP) = \Theta_{A}~.$
\end{proof}

Given a pdg module $(V,d_V)$, we can endow the \textit{coendomorphism operad} $\mathrm{Coend}_V$ with a curved operad structure using the same curvature $\Theta_V(\mathrm{id}) \coloneqq d_V^2$. 

\begin{Definition}[Curved $\mathcal{P}$-coalgebra]
A \textit{curved} $\PP$-\textit{coalgebra} structure on $(V,d_V)$ is the data of a morphism of curved operads $\Delta_V: (\PP,d_\PP,\Theta_\mathcal{P}) \longrightarrow (\mathrm{Coend}_V,\partial,\Theta_V)~.$ 
\end{Definition}

\begin{Remark}
The notion of a curved $\PP$-algebra can encode types of curved coalgebras which are not necessarily conilpotent.
\end{Remark}

\begin{Remark}\label{extends to curved case}
It is straightforward to introduce the analogous notion of a curved partial operad, since the condition imposed on the pre-differential involves the partial compositions of the operad and not the total compositions. The category of curved partial operads is denoted by $\mathsf{curv}~\mathsf{pOp}$. Furthermore, it immediate to see that the category of curved unital partial operads is equivalent to the category of curved operads. See Proposition \ref{prop: unital partial operads are operads} for the non-curved statement.
\end{Remark}

\subsection{First examples} Here are some classical examples of curved algebras that one can encode with curved operads. See for instance \cite{PositselskiTwoKinds} for curved associative algebras. See for instance \cite{LazarevCocom} for curved Lie algebras. For a first treatment of curvature using operadic methods, see \cite{HirshMilles12}.

\begin{Definition}[Curved Lie algebra]\label{clie}
A \textit{curved Lie algebra} $(\mathfrak{g},[-,-],d_{\mathfrak{g}},\vartheta)$  is the data of a graded Lie algebra $(\mathfrak{g},[-,-])$ together with a pre-differential $d_{\mathfrak{g}}$ of degree $-1$ which is a derivation with respect to the bracket $[-,-]$ , and a morphism of graded modules $\vartheta: \mathbb{K} \longrightarrow \mathfrak{g}$ of degree $-2$ such that:
\[
d_{\mathfrak{g}}^2=[\vartheta(1), -]~, \quad \text{and} \quad d_{\mathfrak{g}}(\vartheta(1))=0~.
\]
A \textit{morphism} $f: (\mathfrak{g}, d_\mathfrak{g},\vartheta_\mathfrak{g}) \longrightarrow (\mathfrak{h},d_\mathfrak{h},\vartheta_\mathfrak{h})$ is the data of a graded Lie algebra morphism $f: \mathfrak{g} \longrightarrow \mathfrak{h}$ that commutes with the pre-differentials and such that $f(\vartheta_\mathfrak{g}) = \vartheta_\mathfrak{h}~.$
\end{Definition}

Recall that the classical partial operad $\mathcal{L}ie$, encoding Lie algebras, can be defined as the free partial operad generated by one binary skew-symmetric operation, modulo the operadic ideal generated by the Jacobi relation. See \cite[Section 13.2]{LodayVallette12} for a complete account.

\medskip 

Let $M$ be the pdg $\mathbb{S}$-module given by $(\mathbb{K}.\zeta,0, \mathbb{K}.\beta,0,\cdots)$ with zero pre-differential, where $\zeta$ is an arity $0$ operation of degree $-2$, and $\beta$ is an arity $2$ operation of degree $0$, basis of the signature representation of $\mathbb{S}_2$. 

\begin{Definition}[$\Liec$ operad]
The \textit{curved partial operad} $\Liec$ is given by the free pdg partial operad generated by $M$ modulo the operadic ideal generated by the Jacobi relation on the generator $\beta$. It is endowed with the curvature $\Theta_\mathcal{L}$ given by $\Theta_\mathcal{L}(\mathrm{id}) \coloneqq \beta \circ_1 \zeta~.$  
\end{Definition}

\begin{lemma}\label{lemmalie}
The data $(\Liec, 0, \Theta_\mathcal{L})$ forms a curved partial operad. Furthermore, the category of curved $\Liec$-algebras is isomorphic to the category of curved Lie algebras.
\end{lemma}

\begin{proof}
First, let us check that $(\Liec, 0, \Theta_\mathcal{L})$ is indeed a curved partial operad. The pre-differential is $0$, hence $d(\theta)=0$. Now we need to check that $[\theta,-] = 0$ . Since $[\theta,-]$ is a derivation, it is enough to check it on the two generators of $\Liec$. The result $[\theta, \vartheta] = 0$ follows from the anti-symmetry of $\beta$. A straightforward computation gives that $[\theta, \beta] = 0$, using the Jacobi relation.  

\medskip

A pdg algebra over the curved partial operad $\Liec$ amounts to a pdg-module $(A,d_A)$ endowed with a graded Lie bracket $[-,-] \coloneqq \gamma_A(\beta)$ and a morphism $\vartheta \coloneqq \gamma_A(\zeta): \kk \longrightarrow \mathfrak{g}$ such that $d_A(\vartheta(1))=0$.
They form a curved $\Liec$-algebra if and only if Diagram \ref{curved} commutes, which equivalent to  $[\vartheta,-] = d_A^2(-)$. Morphisms of curved $\Liec$-algebras must commute with the structure, i.e: the bracket and curvature $\vartheta$; and with the underlying pre-differentials. 
\end{proof}

\begin{Definition}[Curved associative algebra]\label{cass}
A \textit{curved associative algebra} $(A,\mu_A,d_A,\vartheta)$ amounts to the data a non-unital graded associative algebra $(A,\mu_A)$ together with a pre-differential $d_A$ of degree $-1$ which is a derivation with respect to the associative product $\mu_A$, and a morphism of graded modules $\vartheta: \mathbb{K} \longrightarrow A$ of degree $-2$ such that:
\[
d_A^2=\mu_A(\vartheta(1),-) - \mu_A(-, \vartheta(1))~, \quad \text{and} \quad d_A(\vartheta(1))=0~.
\]
A \textit{morphism} $f: (A,\mu_A, d_A,\vartheta_A) \longrightarrow (B,\mu_B, d_B, \vartheta_B)$ is the data of a morphism of graded non-unital associative algebras $f:(A,\mu_A) \longrightarrow (B,\mu_B)$ that commutes with the pre-differentials and preserves the curvatures $f(\vartheta_A) = \vartheta_B~.$
\end{Definition}

Let $N$ be the pdg $\mathbb{S}$-module given by $(\mathbb{K}.\phi,0,\mathbb{K}[\mathbb{S}_2].\mu,0,\cdots)$ with zero pre-differential, where $\phi$ is an arity $0$ operation of degree $-2$, and $\mu$ is an binary operation of degree $0$, basis of the regular representation of $\mathbb{S}_2$.

\begin{Definition}[$\Assc$ operad]
The \textit{curved partial operad} $\Assc$ is given by the free pdg partial operad generated by $N$ modulo the operadic ideal generated by the associativity relation on the generator $\mu$. It is endowed with the curvature $\Theta_\mathcal{A}$ given by $\Theta_\mathcal{A}(\mathrm{id}) \coloneqq \mu \circ_1 \phi - \mu \circ_2 \phi.$ 
\end{Definition}

\begin{lemma}
The data $(\Assc, 0, \Theta_\mathcal{A})$ forms a curved partial operad. Furthermore, the category of curved $\Assc$-algebras is isomorphic to the category of curved associative algebras.
\end{lemma}

\begin{proof}
The proof of this lemma follows the same steps as the proof of Lemma \ref{lemmalie}. 
\end{proof}

\begin{Remark}
See Appendix \ref{Appendix B} for the "absolute analogues" of these curved operads.
\end{Remark}

\begin{Proposition}\label{assliem}
The classical morphism of partial operads $\mathcal{L}ie \longrightarrow \mathcal{A}ss$ induced by $\beta \mapsto \mu - \mu^{(12)}~,$ extends to a morphism of curved partial operads $\Liec \longrightarrow \Assc$ by sending $\zeta$ to $\phi$.
\end{Proposition}

\begin{proof}
It is straightforward to check that the latter morphism commutes with the respective curvatures. 
\end{proof}

\begin{Corollary}\label{Antisymmetrization}
The morphism $\Liec \longrightarrow \Assc$ induces a functor
\[
\begin{tikzcd}[column sep=3pc,row sep=0pc]
\mathsf{Skew}: \mathsf{curv}~\Assc\text{-}\mathsf{alg} \arrow[r]
&\mathsf{curv}~\Liec\text{-}\mathsf{alg} \\
(A,\mu_A,d_A,\vartheta) \arrow[r,mapsto]
&(A,[-,-] \coloneqq \mu_A - \mu_A^{(12)},d_A,\vartheta)~.
\end{tikzcd}
\]
given by the skew-symmetrization of the associative product. 
\end{Corollary} 

\begin{proof}
Given a morphism of curved partial operads $\Gamma_A: \Assc \longrightarrow \mathrm{End}_A$, one pulls back along $\Liec \longrightarrow \Assc$.
\end{proof}

\begin{Remark}
In the next section we will construct the left adjoint of this functor using curved bimodules.
\end{Remark}

\section{Curved bimodules and universal functors}
The goal of this section is to define the "curved" generalization of modules over operads. Their appeal comes from the fact that operadic modules encode universal functors relating categories of algebras over operads. The generalization is not immediate since in a curved operad, the curvature condition intertwines the underlying category with the operadic structure. Nevertheless, there exists a good notion of bimodule that encodes universal functors between the categories of curved algebras over curved operads. Note that in order for a bimodule to define a functor between categories of algebras over operads, we need to use the monoidal definition of operads, like in \cite{Fresse09}. If one is working with curved partial operads, one can simply add an augmented unit to fit into this framework.

\medskip

\subsection{Curved bimodules} One major difference between curved operads and classical operads is the following fact. Let $\cP$ be a curved operad and $(V,d_V)$ be a pdg module. Then $\mathscr{S}(\PP)(V)$ endowed with its canonical $\PP$-algebra structure does not form a curved $\PP$-algebra in general. The pre-differential on $\mathscr{S}(\PP)(V)$ is given by: 
\[
d_{\mathscr{S}(\PP)(V)} = \mathscr{S}(d_\PP)(\mathrm{id}_V) + \mathscr{S}(id_\PP)(\diracComb(\mathrm{id}_V,d_V))~.
\]
Hence we have that 
\[
d_{\mathscr{S}(\PP)(V)}^2 = \mathscr{S}(d_\PP^2)(\mathrm{id}_V) + \mathscr{S}(\mathrm{id}_\PP)(\diracComb(\mathrm{id}_V,d_V^2))~,
\]
because of the Koszul sign convention. Since $\PP$ is a curved operad:
\[
d_\PP^2 = \gamma_{(1)} \cdot (\Theta_\PP \circ \mathrm{id}_\PP - \mathrm{id}_\PP \circ' \Theta_\PP)~.
\]
Therefore
\begin{align*}
d_{\mathscr{S}(\PP)(V)}^2 = &~\mathscr{S}(\gamma)(\mathrm{id}_V) \cdot \Big( \mathscr{S}(\Theta_\PP) \circ \mathscr{S}(\mathrm{id}_\PP)(\mathrm{id}_V) - \mathscr{S}(\mathrm{id}_\PP) \circ \mathscr{S}(\diracComb(\mathrm{id}_\PP,\Theta_\PP))(\mathrm{id}_V) \Big) \\ 
&+ \mathscr{S}(\mathrm{id}_\PP)(\diracComb(\mathrm{id}_V,d_V^2))
\end{align*}
and does not satisfy the condition imposed by the diagram \ref{curved}. In general
\[
d_{\mathscr{S}(\PP)(V)}^2 \neq \mathscr{S}(\gamma)(\mathrm{id}_V) \cdot \Big(\mathscr{S}(\Theta_\PP) \circ \mathscr{S}(\mathrm{id}_\PP)(\mathrm{id}_V)\Big)~,
\]
even if we impose the extra condition that $d_V^2=0~.$

\begin{Proposition}[{\cite[Proposition C.31]{JoanCurved}}]\label{freecurvedalg}
Let $\cP$ be a curved operad and $(V,d_V)$ be a pdg module. The free curved $\PP$-algebra is given by:
\[
\mathsf{F}(\PP)(V) \coloneqq \frac{\mathscr{S}(\PP)(A)}{\mathsf{Im}\Big(\mathscr{S}(\mathrm{id}_\PP)(d_V^2) - \mathscr{S}(\Theta_\PP)(\mathrm{id}_V)\Big)}~.
\]
\end{Proposition}

\begin{proof}
One can check that $\mathsf{F}(\PP)(V)$ is equal to the composition of the free pdg $\PP$-algebra functor $\mathscr{S}(\PP)(V)$ followed by the reflector $\mathsf{Curv}$ constructed in Proposition \ref{reflexive}.
\end{proof}

This type of construction is encoded by a left $\PP$-module in the classical case. The above Proposition shows that in the curved case, modules over operads are more intricate, since it is not obvious that this quotient can be interpreted as a "curved" left $\PP$-module. We begin by recalling standard definitions.

\begin{Definition}[Left and right $\PP$-modules]
Let $(\PP, \gamma,\eta, d_\PP)$ be a pdg operad. 
\begin{enumerate}
\item A \textit{left pdg} $\PP$-\textit{module} $(N,\lambda,d_N)$ is the data of a pdg $\mathbb{S}$-module $(N,d_N)$ together with a morphism of pdg $\mathbb{S}$-modules $\lambda: \PP \circ N \longrightarrow N$, such that the following diagram
\[
\begin{tikzcd}[column sep=3pc,row sep=3pc]
\PP \circ \PP \circ N \arrow[r, "\gamma_\PP ~ \circ ~ \mathrm{id}"] \arrow[d,swap,"\mathrm{id} ~ \circ ~ \lambda"]
&\PP \circ N \arrow[d,"\lambda"]\\
\PP \circ N \arrow[r,"\lambda"]
&N 
\end{tikzcd}
\]
commutes and such that $\lambda \cdot (\eta \circ \mathrm{id}_N) = \mathrm{id}_N~.$ 

\item A \textit{right pdg} $\PP$-\textit{module} $(M,\rho,d_M)$ is the data of a pdg $\mathbb{S}$-module $(M,d_M)$ together with a morphism of pdg $\mathbb{S}$-modules $\rho: M \circ \PP \longrightarrow \PP$, such that the following diagram
\[
\begin{tikzcd}[column sep=3pc,row sep=3pc]
M \circ \PP \circ \PP \arrow[r, "\mathrm{id} ~ \circ ~ \gamma"] \arrow[d,swap,"\rho ~ \circ ~ \mathrm{id}"]
&M \circ \PP \arrow[d,"\rho"]\\
M \circ \PP  \arrow[r,"\rho"]
&M
\end{tikzcd}
\]
commutes and such that $\rho \cdot (\mathrm{id}_M \circ \rho) = \mathrm{id}_M~.$ 
\end{enumerate}
\end{Definition}

\begin{Remark}
One can identify left pdg $\PP$-modules concentrated in arity $0$ with the category of pdg $\PP$-algebras.
\end{Remark} 

\begin{Definition}[Relative composition product]
Let $(\PP, d_\PP)$ be a pdg operads. Let $(N,\lambda_N,d_N)$ be a left pdg $\PP$-module and let $(M,\rho_M,d_M)$ be a right pdg $\PP$-module, the \textit{relative composition product} $M \circ_\PP N$, given by the following coequalizer
\[
M \circ_\PP N \coloneqq
\begin{tikzcd}[column sep=4.5pc,row sep=0pc]
\mathsf{Coeq}~ \Big( M \circ \PP \circ N \arrow[r,shift left=.75ex,"\rho_M ~ \circ ~ \mathrm{id}"] \arrow[r,shift right=.75ex,swap,"\mathrm{id} ~ \circ ~ \lambda_N"]
&M \circ N \Big) ~,
\end{tikzcd}
\]
in the category of pdg $\mathbb{S}$-modules. 
\end{Definition} 

\begin{Definition}[Operadic bimodule]
Let $(\PP, d_\PP)$ and $(\Q, d_\Q)$ be two pdg operads. A \textit{pdg} $(\PP, \Q)$-\textit{bimodule} is the data $(M,\lambda,\rho,d_M)$ of a pdg $\mathbb{S}$-module $(M,d_M)$ together with two morphisms of pdg $\mathbb{S}$-modules $\lambda: \PP \circ M \longrightarrow M$ and $\rho: M \circ \Q \longrightarrow M$ which endow $M$ with a left $\PP$-module structure and a right $\PP$-module structure. Those structures are compatible with each other in the following sense: 
\[
\begin{tikzcd}[column sep=3pc,row sep=3pc]
\PP \circ M \circ \Q \arrow[d,swap, "\mathrm{id} ~ \circ ~ \rho "] \arrow[r,"\lambda ~ \circ ~ \mathrm{id}"]
&M \circ \Q \arrow[d,"\rho"] \\
\PP \circ M  \arrow[r,"\lambda"]
&M~.
\end{tikzcd}
\]
\end{Definition}

Operadic bimodules encode functors between the categories of algebras over operads, see \cite[Chapter 9]{Fresse09} for a detailed account. 

\begin{Definition}[Relative Schur functor of a bimodule]
Let $(M,\lambda,\rho,d_M)$ be a pdg $(\PP,\Q)$-bimodule. The \textit{relative Schur functor} $\mathscr{S}_\Q(M)(-)$ associated to $M$ is given, for $(A,\gamma_A,d_A)$ a pdg $\Q$-algebra, by the following coequalizer:
\[
\mathscr{S}_\Q(M)(A) \coloneqq
\begin{tikzcd}[column sep=4.5pc,row sep=0pc]
\mathsf{Coeq}~ \Big( \mathscr{S}(M) \circ \mathscr{S}(\Q)(A) \arrow[r,shift left=.75ex,"\mathscr{S}(\rho)(\mathrm{id}_A)"] \arrow[r,shift right=.75ex,swap,"\mathscr{S}(\mathrm{id}_M)(\gamma_A)"]
&\mathscr{S}(M)(A) \Big) ~.
\end{tikzcd}
\]
It defines a functor: 
\[
\mathscr{S}_\Q(M)(-): \mathsf{pdg}~\Q\text{-}\mathsf{alg} \longrightarrow \mathsf{pdg}~\PP\text{-}\mathsf{alg}~.
\]
\end{Definition}

Under the identification of left pdg $\Q$-modules concentrated in arity $0$ with pdg $\Q$-algebras, $\mathscr{S}_\Q(M)(A)$ is in fact given by the relative composition product $M \circ_\Q A$ of the left pdg $\Q$-module $(A,\gamma_A,d_A)$ with the $(\PP, \Q)$-bimodule $(M,\lambda,\rho,d_M)$.

\begin{Definition}[Curved Bimodules]\label{curvedbimodule}
Let $(\PP, d_\PP,\Theta_\PP)$ and $(\Q, d_\Q,\Theta_\Q)$ be two curved operads. A \textit{curved} $(\PP, \Q)$-\textit{bimodule} is the data of a pdg $(\PP, \Q)$-bimodule $(M,\lambda_M,\rho_M,d_M)$ such that the following diagram commutes:
\[
\begin{tikzcd}[column sep=7pc,row sep=3pc]
M \arrow[r,"\mathrm{diag}"] \arrow[rrd,swap,"d_M^2", bend right = 10]
&(\I \circ M) \oplus (M \circ \I) \arrow[r,"(\Theta_\PP~ \circ~ \mathrm{id}) - (\mathrm{id}~ \circ' ~\Theta_\Q)"] 
&(\PP \circ M) \oplus (M \circ_{(1)} \Q) \arrow[d,"\lambda + \rho_{(1)}"]\\
&
&M~.
\end{tikzcd}
\]
This equality is denoted by $d_M^2 = \lambda \cdot (\Theta_\PP \circ \mathrm{id}) - \rho_{(1)} \cdot (\mathrm{id} \circ' \Theta_\Q)~.$ 
\end{Definition}

\begin{Remark}
A curved generalization of infinitesimal bimodules for curved partial operads is immediate since the curvature is an infinitesimal notion. For instance, curved infinitesimal bimodules are the coefficients for the André-Quillen cohomology of a curved operad, which can be defined using the same methods as for operads and standard bimodules. See \cite{MerkulovVallette09I} for more details.
\end{Remark}

\begin{Proposition} \label{functorbimod}
Any curved $(\PP, \Q)$-bimodule $(M,\lambda,\rho,d_M)$ induces a functor 
\[
\mathscr{S}_\Q(M)(-): \mathsf{curv}~\Q\text{-}\mathsf{alg} \longrightarrow \mathsf{curv}~\PP\text{-}\mathsf{alg}~,
\]
given by the relative Schur functor associated to $M$ as a pdg $(\PP, \Q)$-bimodule.
\end{Proposition}

\begin{proof}
Let $(A,\gamma_A,d_A)$ be a curved $\Q$-algebra. Let $\pi_A: \mathscr{S}(M)(A) \twoheadrightarrow \mathscr{S}_\Q(M)(A)$ be the projection map. The pre-differential of $\mathscr{S}_\Q(M)(A)$ is given by the image in the quotient of 
\[
\mathscr{S}(d_M)(\mathrm{id}_A) + \mathscr{S}(\mathrm{id}_M)(\diracComb(\mathrm{id}_A,d_A))~,
\]
which we will denote $d$ for simplicity. The pdg $\PP$-algebra structure on $\mathscr{S}_\Q(M)(A)$ is given by the map $\pi_A \cdot (\mathscr{S}(\lambda)(\mathrm{id}))~.$ We need to check that $\mathscr{S}_\Q(M)(A)$ is indeed a curved $\PP$-algebra, i.e: that the following diagram commutes:
\[
\begin{tikzcd}[column sep=5.5pc,row sep=3pc]
\mathscr{S}(\I)(\mathscr{S}_\Q(M)(A)) \arrow[r,"\mathscr{S}(\Theta_\PP)(\mathrm{id})"] \arrow[rd,"d^2",swap]
&\mathscr{S}(\PP)(\mathscr{S}_\Q(M)(A)) \arrow[d,"\pi_A \cdot (\mathscr{S}(\lambda_M)(\mathrm{id}))"]\\
&\mathscr{S}_\Q(M)(A)~.
\end{tikzcd}
\]
We know that $d^2$ will be induced by $\mathscr{S}(d_M^2)(\mathrm{id}_A) + \mathscr{S}(\mathrm{id}_M)(\diracComb(\mathrm{id}_A,d_A^2))$. On one hand, since $M$ is a curved $(\PP, \Q)$-bimodule, $d_M^2 = \lambda \cdot (\Theta_\PP \circ \mathrm{id}) - \rho_{(1)} \cdot (\mathrm{id} \circ' \Theta_\Q)~.$ On the other hand, since $A$ is a curved $\Q$-algebra, $d_A^2 = \gamma_A \cdot (\mathscr{S}(\Theta_\Q)(\mathrm{id}))~.$ Therefore: 
\begin{align*}
d^2 = &~\pi_A \cdot (\mathscr{S}(\lambda)(\mathrm{id}_A)) \cdot (\mathscr{S}(\Theta_\PP) \circ \mathscr{S}(\mathrm{id}_M)(\mathrm{id}_A)) - \pi_A \cdot (\mathscr{S}(\rho)(\mathrm{id}_A)) \cdot (\mathscr{S}(\diracComb(\mathrm{id}_M,\Theta_\Q)(A)) \\
	  &+\pi_A \cdot \mathscr{S}(\mathrm{id}_M)\Big(\diracComb(\mathrm{id}_A,\gamma_A \cdot (\mathscr{S}(\Theta_\Q)(\mathrm{id})))\Big)~.	 
\end{align*}
By definition of $\mathscr{S}_\Q(M)(A)~,$ we have that the right action of $\Q$ on $M$ given by $\rho$ is equal to the action of $\Q$ on $A$ given by $\gamma_A$, hence the last two terms are equal and therefore cancel each other. As a result:
\[
d^2 = \pi_A \cdot (\mathscr{S}(\lambda)(\mathrm{id}_A)) \cdot (\mathscr{S}(\Theta_\PP) \circ \mathscr{S}(\mathrm{id}_M)(\mathrm{id}_A))~,
\]
which is precisely the condition imposed by the above diagram. 
\end{proof}

\begin{Example}
Given a curved operad $(\PP, d_\PP,\Theta_\mathcal{P})$, its underlying pdg $\mathbb{S}$-module $(\PP, d_\PP)$ can be endowed with a canonical curved $(\PP,\PP)$-bimodule structure given by the composition of $\PP$. This curved bimodule encodes the identity endofunctor of curved $\PP$-algebras. 
\end{Example}

\begin{theorem}\label{Ind and Res adjunction}
Let $(\PP, d_\PP,\Theta_\PP)$ and $(\Q, d_\Q,\Theta_\Q)$ be two curved operads, and let $f: \PP \longrightarrow \Q$ be a morphism of curved operads. The morphism $f$ induces an adjunction at the level of curved algebras:
\[
\begin{tikzcd}[column sep=7pc,row sep=3pc]
            \mathsf{Ind}_f: \mathsf{curv}~\PP\text{-}\mathsf{alg} \arrow[r, shift left=1.1ex, ""{name=F}] & \mathsf{curv}~\Q\text{-}\mathsf{alg} : \mathsf{Res}_f \arrow[l, shift left=.75ex, ""{name=U}]
            \arrow[phantom, from=F, to=U, , "\dashv" rotate=-90]
\end{tikzcd}
\]
given by the restriction and the induction functors. 
\end{theorem}

\begin{proof}
The right adjoint is given by the restriction along $f$ as follows. Given a curved $\D$-algebra structure on a pdg module $(A,d_A)$, that is, a morphism of curved operads $\Gamma_A: \Q \longrightarrow \mathrm{End}_A~,$ one can always pre-compose $\Gamma_A$ with $f$. This curved operadic morphism $\Gamma_A \cdot f: \PP \longrightarrow \mathrm{End}_A$ endows $(A,d_A)$ with a curved $\PP$-algebra structure, denoted by $\mathsf{Res}_f(A)~.$

\medskip

On the other hand, the morphism $f$ allows us to endow $\Q$ with a curved $(\Q,\PP)$-bimodule structure, where right action of $\PP$ is given by: 
\[
\begin{tikzcd}[column sep = 4pc,row sep=2pc]
\rho :\Q \circ \PP \arrow[r,"\mathrm{id}_\Q ~\circ ~f"]
&\Q \circ \Q  \arrow[r,"\gamma_\Q"]
&\Q~,
\end{tikzcd}
\]
and the left action of $\Q$ is simply given by $\gamma_\Q~.$ Let us check that this endows $\Q$ with a curved bimodule structure. We have that:
\[
d_\Q^2 = \gamma_\Q \cdot (\Theta_\Q \circ \mathrm{id}) - (\gamma_\Q)_{(1)} \cdot (\mathrm{id} \circ' \Theta_\Q) = \gamma_\Q \cdot (\Theta_\Q \circ \mathrm{id}) - \rho_{(1)} \cdot (\mathrm{id} \circ' \Theta_\PP)
\]
since $f \cdot \Theta_\PP  = \Theta_\Q~.$ Thus, the curved $(\Q,\PP)$-bimodule $\Q$ induces a functor 
\[
\mathsf{Ind}_f(-) \coloneqq \mathscr{S}_\PP(\Q)(-): \mathsf{curv}~\PP\text{-}\mathsf{alg} \longrightarrow \mathsf{curv}~\Q\text{-}\mathsf{alg}~.
\]
Since morphism of curved algebras are morphisms of pdg algebras, this general operadic construction is still an adjunction. 
\end{proof}

One important application of this new result is the construction of the curved universal enveloping algebra of a curved Lie algebra. To the best of our knowledge, this construction is new. The morphism $\Liec \longrightarrow \Assc$, defined in Proposition \ref{assliem}, induces the following adjunction.

\begin{Corollary}[Curved universal enveloping algebra]
There an adjunction between the category of curved Lie algebras and the category of curved associative algebras
\[
\begin{tikzcd}[column sep=7pc,row sep=3pc]
            \mathfrak{U} : \mathsf{curv}~\Liec\text{-}\mathsf{alg} \arrow[r, shift left=1.1ex, ""{name=F}] & \mathsf{curv}~\Assc \text{-}\mathsf{alg} : \mathsf{Skew}~, \arrow[l, shift left=.75ex, ""{name=U}]
            \arrow[phantom, from=F, to=U, , "\dashv" rotate=-90]
\end{tikzcd}
\]
where $\mathfrak{U}$ denotes the curved universal enveloping algebra and $\mathsf{Skew}$ denotes the functor obtained by the skew-symmetrization of the associative product of Corollary \ref{Antisymmetrization}.
\end{Corollary}

\begin{proof}
Simply apply Theorem \ref{Ind and Res adjunction} to the morphism $f: \Liec \longrightarrow \Assc$, and set $\mathfrak{U} \coloneqq \mathsf{Ind}_f$ and $\mathsf{Anti} \coloneqq \mathsf{Res}_f~.$
\end{proof}

\begin{Proposition}
Let $\mathfrak{g}$ be a curved Lie algebra. Its universal enveloping algebra is isomorphic to
\[
\mathfrak{U}(\mathfrak{g}) \cong \frac{\overline{\mathcal{T}}(\mathfrak{g})}{\left(x \otimes y - (-1)^{|x|} y \otimes x - [x,y]\right)}~,
\]
where $\overline{\mathcal{T}}(-)$ denotes the non-unital tensor algebra, endowed with the curvature $\vartheta: \kk \longrightarrow \mathfrak{g} \hookrightarrow \mathfrak{U}(\mathfrak{g})$. 
\end{Proposition}

\begin{proof}
One can check this adjunction by hand. Indeed, the adjunction holds between between pdg $c\mathcal{L}ie$-algebras and pdg $c\mathcal{A}ss$-algebras. It is straightforward to check that $\mathfrak{U}(\mathfrak{g})$ endowed with the curvature $\vartheta$ forms a curved $c\mathcal{L}ie$-algebra. 
\end{proof}

\begin{Remark}[PBW property]
In \cite{Tamaroff2020}, the authors develop a general theory of \textit{PBW properties} for monads: a morphism $f: M \longrightarrow N$ has the PBW property if $N$ is free as a right $M$-module, where the structure is induced by the morphism $f$. This encompasses the classical PBW theorem, since we have an isomorphism
\[
\mathcal{A}ss \cong \mathcal{C}om \circ \mathcal{L}ie
\]
of right $\mathcal{L}ie$-modules. This gives back the isomorphism of vector spaces between the universal enveloping algebra of a Lie algebra and the symmetric algebra of a Lie algebra. 

\medskip

The same result hold in the curved setting, where there is an isomorphism 
\[
c\mathcal{A}ss \cong \mathcal{C}om \circ c\mathcal{L}ie
\]
of pdg right $c\mathcal{L}ie$-modules. This can be deduced from the classical PBW, since the morphism $c\mathcal{L}ie \longrightarrow c\mathcal{A}ss$ preserves the arity zero component; grading by the number of corks, the map induced map in the associated graded is an isomorphism, therefore the map was an isomorphism to begin with. This PBW property implies that the universal enveloping algebra of a curved Lie algebra is also isomorphic to the symmetric algebra of the curved Lie algebra as graded vector spaces.
\end{Remark}

This machinery can also be applied to curved mixed $\mathcal{L}_\infty$-algebras. These are curved $\mathcal{L}_\infty$-algebras with two pre-differentials, one coming from the underlying pdg module and another one from the structure, such that their difference satisfies the axioms of a classical curved $\mathcal{L}_\infty$-algebra. See \cite[Appendix]{integration} for more details. These algebras are exactly curved algebras over $\widehat{\Omega}(u\mathcal{C}om^*)$, where $\widehat{\Omega}$ is the complete Cobar construction of Section \ref{Section: Constructions Bar-Cobar operadiques}.

\begin{Corollary}[Curved $\mathcal{A}_\infty$ universal enveloping algebra]
There is the following adjunction between the category of curved mixed $\mathcal{L}_\infty$-algebras and the category of curved mixed $\mathcal{A}_\infty$ algebras given by:
\[
\begin{tikzcd}[column sep=7pc,row sep=3pc]
            \mathfrak{U} : \mathsf{curv.mix}~\mathcal{L}_\infty\text{-}\mathsf{alg} \arrow[r, shift left=1.1ex, ""{name=F}] & \mathsf{curv.mix}~\mathcal{A}_\infty \text{-}\mathsf{alg} : \mathsf{Anti}~, \arrow[l, shift left=.75ex, ""{name=U}]
            \arrow[phantom, from=F, to=U, , "\dashv" rotate=-90]
\end{tikzcd}
\]
where $\mathfrak{U}$ denotes the curved $\mathcal{A}_\infty$ universal enveloping algebra.
\end{Corollary}

\begin{proof}
There is a morphism of counital partial cooperads $\iota: u\mathcal{C}om^* \longrightarrow u\mathcal{A}ss^*$ given by the linear dual of the classical morphism of unital partial operads $u\mathcal{A}ss \twoheadrightarrow u\mathcal{C}om$. Therefore, using the complete Cobar construction of Section \ref{Section: Constructions Bar-Cobar operadiques}, we get a morphism of curved partial operads $\widehat{\Omega}(\iota): \widehat{\Omega}(u\mathcal{C}om^*) \longrightarrow \widehat{\Omega}(u\mathcal{A}ss^*)$. By Theorem \ref{Ind and Res adjunction} it induces an adjunction.
\end{proof}

\subsection{The analogues of left and right modules}One may ask whether there exists a natural notion of a left (resp. right) curved $\PP$-module. This is in fact not obvious. Since every algebra over an operad is a particular case of left module in the classical setting, a naive generalization could be: a left pdg $\PP$-module $(M,\lambda,d_M)$ is a  left curved $\PP$-module if the following diagram commutes
\[
\begin{tikzcd}[column sep=5pc,row sep=2pc]\label{curvedmod}
\I \circ M \arrow[r,"\Theta_\PP \hspace{1pt} \circ \hspace{1pt} \mathrm{id}"] \arrow[rd, "d_M^2",swap]
&\PP \circ M \arrow[d,"\lambda"]\\
&M~.
\end{tikzcd}
\]
But in this case $\mathscr{S}(M)(-)$ does not produce a functor $\mathsf{pdg}~\mathsf{mod} \longrightarrow \mathsf{curv}~\PP\text{-}\mathsf{alg}~,$ since $\mathscr{S}(M)(A)$ is not, in general, a curved $\PP$-algebra. One can notice  by doing a straightforward computation that for this notion of curved left $\PP$-module, $\mathscr{S}(M)(A)$ is a curved $\PP$-algebra if and only if $d_A^2=0~.$

\medskip

The conceptual explanation is that the data of a dg module is equivalent to the data of a curved algebra over the curved operad $(\I, 0, 0)$, where $\I$ is the trivial operad with the zero pre-differential and the zero curvature. The above-mentioned naive definition of a left curved $\PP$-module is in fact a curved $(\PP, \I)$-bimodule, which is coherent with the fact that this notion encodes functors $\mathsf{dg}~\mathsf{mod} \longrightarrow \mathsf{curv}~\PP\text{-}\mathsf{alg}~.$ Likewise, a symmetric naive definition of a right curved $\PP$-module would in fact be a $(\I, \PP)$-bimodule that would encode functors $\mathsf{curv}~\PP\text{-}\mathsf{alg} \longrightarrow \mathsf{dg}~\mathsf{mod}$ .

\begin{Remark}
Let $(\PP, \gamma,\eta, d_\PP, \Theta_\PP)$ be a curved operad. The unit $\eta: \I \longrightarrow \PP$ is not a morphism of curved operads. Otherwise, we could endow $\PP$ with a curved $(\PP, \I)$-bimodule structure and $\mathscr{S}(\PP)(A)$ would be a curved $\PP$-algebra for $(A,d_A)$ a dg module, which is not the case.
\end{Remark}

The natural question is then to find a curved operad that plays the same role as the unit operad $\I$ plays for classical operads in the theory of bimodules. Otherwise stated, a curved operad that encodes pdg modules as its curved algebras.

\begin{Definition}[$\mathcal{I}\mathcal{C}$ operad]
The pdg operad $\mathcal{IC}$ is the free pdg operad generated by an operation of arity $1$ and degree $-2$. It is given by the pdg $\mathbb{S}$-module $(0,\mathbb{K}.\mathrm{id} ~\oplus~ \bar{S}(\theta),0,\cdots)$ with zero pre-differential, where $\bar{S}(\theta)$ denotes the free non-unital commutative algebra generated by $\theta~.$
\end{Definition}

\begin{lemma}
The pdg operad $\mathcal{IC}$, endowed the curvature $\Theta_{\mathcal{IC}}(\mathrm{id}) \coloneqq \theta$ forms a curved operad. The category of curved $\mathcal{IC}$-algebras is isomorphic to the category of pdg modules. 
\end{lemma}

\begin{proof}
The commutator with $\theta$ is zero, hence $\mathcal{IC}$ is a curved operad. Let $(V,d_V)$ be a pdg module and let $\Gamma: \mathcal{IC} \longrightarrow \mathrm{End}_V$ be a morphism of curved operads. The image of $\theta$ determines $\Gamma$, since it is a morphism of curved operads then $\Gamma(\theta) = d_V^2~.$ 
\end{proof}

For any curved operad $(\PP, d_\PP, \Theta_\PP)$, there is an unique morphism of curved operads $\varphi_\PP: \mathcal{IC} \longrightarrow \PP$ given by $\varphi_\PP(\theta) = \theta_\PP$, where $\theta_\PP = \Theta_\PP(\mathrm{id})$. Therefore the pdg $\mathbb{S}$-module $\PP$ can be endowed canonically with a curved $(\mathcal{IC},\PP)$-bimodule structure and with a curved $(\PP,\mathcal{IC})$-bimodule structure.

\begin{Proposition}
Let $(\PP, d_\PP, \Theta_\PP)$ be a curved operad and let 
\[
\begin{tikzcd}[column sep=7pc,row sep=3pc]
            \mathsf{F}(\PP) : \mathsf{pdg}~\mathsf{mod} \arrow[r, shift left=1.1ex, ""{name=F}] & \mathsf{curv}~\PP \text{-}\mathsf{alg} : \mathsf{U}~, \arrow[l, shift left=.75ex, ""{name=U}]
            \arrow[phantom, from=F, to=U, , "\dashv" rotate=-90]
\end{tikzcd}
\]
be the free-forgetful adjunction of Proposition \ref{freecurvedalg}. The free functor $\mathsf{F}(\PP)$ is naturally isomorphic to the functor $\mathscr{S}_{\mathcal{IC}}(\PP)$ given by the canonical curved $(\PP,\mathcal{IC})$-bimodule structure on $\PP$. The forgetful functor $\mathsf{U}$ is naturally isomorphic to $\mathscr{S}_{\PP}(\PP)$ given by the canonical curved $(\mathcal{IC},\PP)$-bimodule structure on $\PP$.
\end{Proposition}

\begin{proof}
Let $(V,d_V)$ be a pdg module, $\mathscr{S}_{\mathcal{IC}}(\PP)(V)$ is defined by the same quotient as $\mathsf{F}(\PP)(V)$: the right action of $\mathcal{IC}$ on $\PP$ is given by $\theta_\PP$ and the left action of $\mathcal{IC}$ on $(V,d_V)$ is given by $d_V^2$, which are identified in $\mathscr{S}_{\mathcal{IC}}(\PP)(V)~.$ It is straightforward to check that $\mathscr{S}_{\PP}(\PP)(V)$ amounts to $(V,d_V)$ endowed with its pdg module structure.
\end{proof}

This indicates that curved $(\PP,\mathcal{IC})$-bimodules are a good analogue to left $\PP$-modules over operads in the case of curved operads; likewise, curved $(\mathcal{IC},\PP)$-bimodules are a good analogue to right $\PP$-modules. 

\begin{Remark}
To the best of our knowledge, these definitions also provide new notions of left and right modules for curved associative algebras when we consider them as curved operads concentrated in arity one. 
\end{Remark}

\section{Curved cooperads and curved partial cooperads}\label{Section: Curved cooperads}
In this section, we briefly recall the notion of a curved cooperad which first appeared in \cite{HirshMilles12}. Conilpotent curved partial cooperads will play the role of the Koszul dual of unital partial operads. The notions of coalgebras and algebras over cooperads also extend to the case of curved cooperads.

\begin{Definition}[Curved cooperad]\label{def curved cooperad}
A \textit{curved cooperad} $(\C,\Delta,\epsilon,d_\C,\Theta_\C)$ amounts to the data of a pdg cooperad $(\C,\Delta,\epsilon,d_\C)$ and a morphism of pdg $\mathbb{S}$-modules $\Theta_\C: (\C,d_\C) \longrightarrow (\I,0)$ of degree $-2$, such that the following diagram commutes: 
\[
\begin{tikzcd}[column sep=7.5pc,row sep=3pc]
\C \arrow[r,"\Delta_{(1)}"] \arrow[rrd,"d_\C^2", bend right =10]
&\C \circ_{(1)} \C \arrow[r,"(\mathrm{id}~ \circ ~ \Theta_\C)~-~(\Theta_\C~ \circ_{(1)}~ \mathrm{id})~"] 
&(\C \circ \I) \oplus (\I \circ \C) \cong  \C \oplus \C  \arrow[d,"\mathrm{proj}"]\\
&
&\C~,
\end{tikzcd}
\]
where $\mathrm{proj}$ is given by $\mathrm{proj}(\mu,\nu) \coloneqq \mu + \nu$. A \textit{morphism} of curved cooperads $f: (\C,,d_\C,\Theta_\C) \longrightarrow (\D,d_\D,\Theta_\D)$ is the data of a morphism of pdg cooperads $f: (\C,,d_\C) \longrightarrow (\D,d_\D)$ such that $\Theta_\D \circ f = \Theta_\C~.$
\end{Definition}

By a slight abuse of notation, we denote this equality by 
\[
d_\C^2 = (\mathrm{id} \circ_{(1)} \Theta_\C - \Theta_\C \circ \mathrm{id}) \cdot \Delta_{(1)}~,
\]
forgetting $\mathrm{proj}$ and the identifications made.

\begin{Remark}
One also defines \textit{curved partial cooperads} and \textit{curved counital partial cooperads} as one would expect, since the condition on the curvature only involves partial decompositions. In this context, the comparison results of Section \ref{Section: Recollections} between different types of cooperads extend without any difficulty to the curved case, \textit{mutatis mutandis}.
\end{Remark}

Any curved cooperad $(\C,d_\C,\Theta_\C)$ induces a comonad structure on its Schur functor $\mathscr{S}(\C)$, and a monad structure on its dual Schur functor $\widehat{\mathscr{S}}^c(\C)~.$ Hence we can define curved coalgebras and curved algebras over any given curved cooperad.

\begin{Definition}[Curved $\C$-coalgebra]\label{def curved conil coalg}
Let $(\C,d_\C,\Theta_\C)$ be a curved cooperad and let $(C,\Delta_C,d_C)$ be a pdg $\C$-coalgebra. It is a \textit{curved} $\C$-\textit{coalgebra} if the following diagram commutes:
\[
\begin{tikzcd}[column sep=3pc,row sep=3pc]
C  \arrow[r,"\Delta_C "] \arrow[rd,"d_C^2",swap]
&\mathscr{S}(\C)(C) \arrow[d,"\mathscr{S}(\Theta_\C)(\mathrm{id})"]\\
&C \cong \mathscr{S}(\I)(C)~.
\end{tikzcd}
\]
A \textit{morphism} of curved conilpotent $\C$-coalgebras $f: (C,\Delta_C,d_C) \longrightarrow (C',\Delta_{C'},d_{C'})$ is the data of a morphism of pdg $\C$-coalgebras. 
\end{Definition}

Thus the category of curved $\C$-coalgebras is a full-subcategory of the category of pdg $\mathcal{C}$-coalgebras. It is not obvious to us whether the category of curved $\C$-coalgebras admits a cofree object or not. It would be interesting to have a comonadicity result like in the case of curved algebras over a curved operad. Nevertheless, the following result guarantees the existence of limits and colimits.

\begin{theorem}[{\cite{grignou2019}}]
The category curved $\C$-coalgebras is a presentable category. In particular, it is both complete and cocomplete.
\end{theorem}

\begin{Definition}[Curved algebra over a cooperad]\label{def curved alg over a coop}
Let $(\C,d_\C,\Theta_\C)$ be a curved cooperad and let $(B,\gamma_B,d_B)$ be a pdg $\C$-algebra. It is a \textit{curved} $\C$\textit{-algebra} if the following diagram commutes: 
\[
\begin{tikzcd}[column sep=4pc,row sep=3pc]
B \cong \widehat{\mathscr{S}}^c(\I)(B) \arrow[r,"\widehat{\mathscr{S}}^c(\Theta_\C)(\mathrm{id}) "] \arrow[rd,"- d_B^2",swap]
&\widehat{\mathscr{S}}^c(\C)(B) \arrow[d,"\gamma_B"]\\
&B ~.
\end{tikzcd}
\]
A \textit{morphism} of curved $\C$-algebras $f: (B,\gamma_B,d_B) \longrightarrow (B',\gamma_{B'},d_{B'})$ is the data of a morphism of pdg $\C$-algebras. 
\end{Definition}

Thus the category of curved $\mathcal{C}$-algebras is also a full sub-category of the category of pdg $\mathcal{C}$-algebras. The following proposition gives a reflector.

\begin{Proposition}[{\cite[Theorem 7.5]{grignoulejay18}}]
Let $(\C,d_\C,\Theta_\C)$ be a curved cooperad. The inclusion functor 
\[
\mathsf{Inc}: \mathsf{curv}~\C\text{-}\mathsf{alg} \hookrightarrow \mathsf{pdg}~\C\text{-}\mathsf{alg}
\]
has a left adjoint 
\[
\mathsf{Curv}: \mathsf{pdg}~\C\text{-}\mathsf{alg} \longrightarrow \mathsf{curv}~\C\text{-}\mathsf{alg}~.
\]
Hence $\mathsf{curv}~\C\text{-}\mathsf{alg}$ is a reflexive subcategory of $\mathsf{pdg}~\C\text{-}\mathsf{alg}$. For a pdg $\C$-algebra $(B,\gamma_B,d_B)$, its image under this functor is given by the following quotient: 

\[
\mathsf{Curv}(B) \coloneqq \frac{B}{\left(\gamma_B \cdot \widehat{\mathscr{S}}^c(\Theta_\C)(\mathrm{id}) + d_B^2(-)\right)}~.
\]
where $\left(\gamma_B \cdot \widehat{\mathscr{S}}^c(\Theta_\C)(\mathrm{id}) + d_B^2(-)\right)$ denoted the ideal generated by $\mathsf{Im}\left(\gamma_B \cdot \widehat{\mathscr{S}}^c(\Theta_\C)(\mathrm{id}) + d_B^2(-)\right)~.$ Its pdg $\C$-algebra structure is induced by $\gamma_B$ and $d_B$.
\end{Proposition}

\begin{Remark}
Let $B$ be a $\mathcal{C}$-algebra, an \textit{ideal} is a subobject $I \hookrightarrow B$ such that the quotient $B/I$ is a $\mathcal{C}$-algebra, and such that the projection morphism is a morphism of $\mathcal{C}$-algebras. See \cite[Definition 4.1]{grignoulejay18}.
\end{Remark}

\begin{Corollary}
The category of curved $\C$-algebras is presentable. In particular it is complete and cocomplete.
\end{Corollary}

The notion of a conilpotent cooperad also generalizes to the case of curved cooperads.

\begin{Definition}[Conilpotent curved partial cooperad]
Let $(\C,\{\Delta_i\},d_\C,\Theta_\C)$ be a curved partial cooperad. It is \textit{conilpotent} if its underlying partial pdg cooperad $(\C,\{\Delta_i\},d_\C)$ is a conilpotent partial cooperad. 
\end{Definition}

There is a functor 
\[
\mathsf{Conil}: \mathsf{curv}~\mathsf{pCoop}^\mathsf{conil} \longrightarrow \mathsf{curv}~\mathsf{Coop}~,
\]
from the category of conilpotent curved partial cooperads to the category of curved cooperads defined as comonoids given by adding a counit as in Corollary \ref{corollary: adding a counit}. 

\begin{Definition}[Conilpotent curved cooperad]
A curved cooperad $(\C,d_\C,\Theta_\C)$ is said to be \textit{conilpotent} if it is in the essential image of the functor $\mathsf{Conil}~.$
\end{Definition}

In the case of a conilpotent curved cooperad, one can induce a canonical filtration on the category of curved $\mathcal{C}$-algebras.

\begin{Definition}[Complete curved algebra over a conilpotent cooperad]
Let $(\C^u,d_\C,\Theta_\C)$ be a conilpotent curved cooperad. A \textit{complete} curved $\C$-algebra $B$ amounts to the data of a curved $\C$-algebra $(B,\gamma_B,d_B)$ such that the canonical morphism of curved $\C$-algebras
\[
\varphi_B: B \longrightarrow \lim_{\omega} B/\mathrm{W}_\omega B
\]
is an isomorphism, where $\mathrm{W}_\omega B$ is the canonical filtration of the $\C$-algebra $B$.
\end{Definition}

\begin{Remark}
These notions where defined in Section \ref{Section: Recollections}.
\end{Remark}

\begin{Proposition}\label{bicomplete C alg}
Let $(\C,d_\C,\Theta_\C)$ be a conilpotent curved cooperad. The category of complete curved $\C$-algebras is a reflexive subcategory of pdg $\C$-algebras. It is thus presentable, and bicomplete.
\end{Proposition}

\begin{proof}
Its reflector is given by the composition of the previous two reflectors: the functor $\mathsf{Curv}$ and the completion functor with respect to the canonical filtration.
\end{proof}

\section{The groupoid-colored level}
In this section, we develop the formalism of groupoid-colored (co)operads introduced in \cite{ward19} in order to encode the different (co)operadic structures that we have encountered so far. More precisely, we construct a unital groupoid-colored operad, denoted by $u\mathcal{O}$, that encodes unital partial operads as its algebras and counital partial cooperads as its coalgebras. The goal of this section is to compute its Koszul dual conilpotent curved groupoid-colored cooperad.

\medskip

In order to do this, we need to extend to groupoid-colored operads the inhomogeneous Koszul duality for operads introduced in \cite{HirshMilles12} (by inhomogeneous we mean relations that involve constant-linear-quadratic terms in this case). Extending the quadratic Koszul duality for groupoid-colored operads of \cite{ward19} to the inhomogeneous case is conceptually quite straightforward. Since this theory will only be applied to the case that interests us, we give an overview of the main results, without entering in full generality.
\[
\begin{tikzpicture} 
\tikzstyle{block} = [draw, rectangle, text width = 12em, 
  text centered, minimum height = 7mm, node distance = 7em];
\tikzstyle{line} = [draw, -stealth, thick]
\node [block] at (0, 0) (one) {\scriptsize Koszul duality \\ homogeneous operads \cite{GinzburgKapranov95},\cite{GetzlerJones94}};
\node [block] at (0, 2) (two) {\scriptsize Koszul duality \\ inhomogeneous operads};
\node [block] at (9, 0) (three) {\scriptsize Koszul duality \\ groupoid-colored homogeneous operads};
\node [block] at (9, 2) (four) {\scriptsize Koszul duality \\ groupoid-colored inhomogeneous operads};
\path [line] (one) -- (two);
\path [line] (two) -- (four);
\path [line] (three) -- (four);
\path [line] (one) -- (three);
\node at (-0.75, 1) [align = left]{\scriptsize \cite{HirshMilles12}};
\node at (4.5, 0.3) [align = center]{\scriptsize \cite{ward19}};
\end{tikzpicture}
\]
This new Koszul duality gives us a conilpotent curved groupoid-colored cooperad $(u\mathcal{O})^{\ac}$, which encodes conilpotent curved partial cooperads as its curved coalgebras. Its category of curved algebras provides us with a new notion of operads which we call \textit{curved absolute partial operads}. See the Appendix \ref{Appendix B} for a detailed description. Using the Koszul curved twisting morphism given by this theory, we will construct two Bar-Cobar adjunctions that interrelate these objects in the next section.

\medskip

\subsection{$\mathbb{S}$-colored (co)operads}
We begin by applying the formalism of \cite{ward19} to the specific case that interests us. We consider the groupoid 
\[
\mathbb{S} \coloneqq \coprod_{n \geq 0} \mathrm{B}\mathbb{S}_n
\]
where $\mathrm{B}\mathbb{S}_n$ is the classifying space of $\mathbb{S}_n$ (the category with one object $\{*\}$ and $\mathrm{Aut}(*) = \mathbb{S}_n$). We denote by $n$ the single object in $\mathrm{B}\mathbb{S}_n~.$ In this context, an $\mathbb{S}$-color scheme as defined in \cite[Definition 2.2.1]{ward19} amounts to the following definition:

\begin{Definition}[pdg $\mathbb{S}$-color scheme]
An \textit{pdg} $\mathbb{S}$-\textit{color scheme} $E$ amounts to the data of a family $\{E(n_1,\cdots,n_r;n)\}$ of pdg modules for all $r$-tuples of natural numbers $(n_1,\cdots,n_r)$ in $\mathbb{N}^{r}$ and all $n$ in $\mathbb{N}$. Each pdg module $E(n_1,\cdots,n_r;n)$ comes equipped with an action of $\mathbb{S}_{n_1} \times \cdots \times \mathbb{S}_{n_r}$ on the right and an action of $\mathbb{S}_n$ on the left as part of the structure. There is a supplementary action of the symmetric groups on this family given by permuting the entries. For any $r$-tuples of natural numbers $(n_1,\cdots,n_r)$, there is an isomorphism
\[
\varphi_\sigma : E(n_1,\cdots,n_r;n) \longrightarrow  E(n_{\sigma^{-1}(1)},\cdots, n_{\sigma^{-1}(r)} ; n) 
\]
of right pdg $\kk[\mathbb{S}_{n_1} \times \cdots \times \mathbb{S}_{n_r}] \otimes \kk[\mathbb{S}_n]^{\text{op}}$-module for all $\sigma$ in $\mathbb{S}_r$ which defines a group action. This endows each $E(n_1,\cdots,n_r;n)$ with an left action of $\left(\mathbb{S}_{n_1} \times \cdots \times \mathbb{S}_{n_r}\right)~\wr~ \mathbb{S}_r$, where $\wr$ denotes the wreath product of groups.

\medskip

A \textit{morphism} $f: E \longrightarrow F$ is the data of a morphism of pdg $\kk[\mathbb{S}_{n_1} \times \cdots \times \mathbb{S}_{n_r}] \otimes \kk[\mathbb{S}_n]^{\text{op}}$-modules 
\[
f(n_1,\cdots,n_r;n): E(n_1,\cdots,n_r;n) \longrightarrow F(n_1,\cdots,n_r;n)~,
\]
for all $r$-tuples $(n_1,\cdots,n_r)$ in $\mathbb{N}^{r}$ which commutes with the permutation maps $\varphi_\sigma$. Pdg $\mathbb{S}$-color schemes form a category denoted by $\mathsf{pdg}~\mathsf{Col}_\mathbb{S}~.$
\end{Definition}

Let $E$ and $F$ be two pdg $\mathbb{S}$-color schemes, their \textit{composition product} $E \circ F$ is given by: 
\begin{align*}
&E \circ F (n_1, \cdots, n_r; n) \coloneqq \bigoplus_{j \geq 1}~ \bigoplus_{(a_1,\cdots ,a_j) \in \mathbb{N}^{j}} \Bigg( E(a_1, \cdots, a_j ;n) \otimes_{(\mathbb{S}_{a_1} \times \cdots \times \mathbb{S}_{a_j})~\wr ~\mathbb{S}_j } \\
&  \Bigg( \bigoplus_{i_1 + \cdots + i_j = r} \mathsf{Ind}_{\mathbb{S}_{i_1} \times \cdots \times \mathbb{S}_{i_j}}^{\mathbb{S}_r} \Big( F(n_1, \cdots ,n_{i_1} ; a_1) \otimes \cdots \otimes F(n_{r-i_j}, \cdots , n_r ; a_j) \Big) \Bigg)\Bigg)_{\mathbb{S}_j} ~.
\end{align*}
where the second sum runs over all $j$-tuples $(a_1,\cdots ,a_j)$ and the third sum runs over all $j$-tuples $(i_1,\cdots ,i_j)$ such that $i_1 + \cdots + i_j = r$. It is the analogue of the composition product $\circ$ of $\mathbb{S}$-modules. In this case, leaves are colored by elements of the symmetric groups, hence the extra requirement that the colors have to match. The unit of this product is given by the $\mathbb{S}$-color scheme $\I_{\mathbb{S}}$, defined by $\I_{\mathbb{S}}(n;n) \coloneqq \mathbb{K}[\mathbb{S}_n]$ as a bimodule over itself and $0$ elsewhere, endowed with the trivial pre-differential. 

\begin{lemma}
The data $(\mathsf{pdg}~\mathsf{Col}_\mathbb{S},\circ, \I_{\mathbb{S}})$ forms a monoidal category. 
\end{lemma}

\begin{proof}
A straightforward computation analogous to the standard case.
\end{proof}

\begin{Definition}[$\mathbb{S}$-colored operad]
A \textit{pdg} $\mathbb{S}$-\textit{colored operad} $\mathcal{G}$ is the data of a monoid $(\mathcal{G}, \gamma_\mathcal{G}, \eta, d_\mathcal{G})$ in the monoidal category $(\mathsf{pdg}~\mathsf{Col}_\mathbb{S},\circ, \I_{\mathbb{S}})$. 
\end{Definition}

\begin{Example}[$\mathbb{S}$-colored endomorphism operad] 
Let $M$ be a pdg $\mathbb{S}$-module. The collection
\[
\mathrm{End}_M(n_1,\cdots,n_r;n) \coloneqq \mathrm{Hom}(M(n_1) \otimes \cdots \otimes M(n_r), M(n)) \quad \text{for} \quad (n_1,\cdots,n_r) \in \mathbb{N}^{r}
\]
has a natural structure of pdg $\mathbb{S}$-color scheme where the right pdg $\kk[\mathbb{S}_{n_1} \times \cdots \times \mathbb{S}_{n_r}] \otimes \kk[\mathbb{S}_n]^{\text{op}}$-module structure comes from the $\mathbb{S}$-module structure of $M$ and where the permutation morphism $\varphi_\sigma$ are given by the natural action of the symmetric group $\mathbb{S}_r$ on the tensor product $M(n_1) \otimes \cdots \otimes M(n_r)$. The composition of morphisms and the identity $\mathrm{id}_n$ in $\mathrm{End}_M(n;n)$ endow this pdg $\mathbb{S}$-color scheme with an $\mathbb{S}$-colored operad structure, called the $\mathbb{S}$-colored \textit{endomorphism operad} of $M$.
\end{Example}

\medskip

\begin{Example}[$\mathbb{S}$-colored coendomorphism operad] 
Let $M$ be a pdg $\mathbb{S}$-module. The collection
\[
\mathrm{Coend}_M(n_1,\cdots,n_r;n) \coloneqq \mathrm{Hom}(M(n), M(n_1) \otimes \cdots \otimes M(n_r)) \quad \text{for} \quad (n_1,\cdots,n_r) \in \mathbb{N}^{r}
\]
has a natural structure of pdg $\mathbb{S}$-color scheme as in the above example. The composition of morphisms and the identity $\mathrm{id}_n$ in $\mathrm{End}_M(n;n)$ endow this pdg $\mathbb{S}$-color scheme with an $\mathbb{S}$-colored operad structure, called the $\mathbb{S}$-colored \textit{coendomorphism operad} of $M$.
\end{Example}

\begin{Definition}[$\mathbb{S}$-colored cooperad]
A \textit{pdg} $\mathbb{S}$-\textit{colored cooperad} $\mathcal{V}$ is the data of a comonoid $(\mathcal{V}, \Delta_\mathcal{V}, \epsilon, d_\mathcal{V})$ in the monoidal category $(\mathsf{Col}_\mathbb{S},\circ, \I_{\mathbb{S}})$. 
\end{Definition}

To any pdg $\mathbb{S}$-color scheme $E$ one can associate an endofunctor in the category of pdg $\mathbb{S}$-modules via the \textit{Schur realization functor}:
\[
\begin{tikzcd}[column sep=4pc,row sep=0.5pc]
\mathscr{S}_{\mathbb{S}} : \mathsf{pdg}~\mathsf{Col}_\mathbb{S} \arrow[r]
&\mathsf{End}(\mathsf{pdg}~\smod) \\
E \arrow[r,mapsto]
&\mathscr{S}_{\mathbb{S}}(E)(-)~.
\end{tikzcd}
\]
The Schur endofunctor $\mathscr{S}_{\mathbb{S}}(E)$ is given, for $M$ a pdg $\mathbb{S}$-module, by:
\[
\mathscr{S}_{\mathbb{S}}(E)(M)(n) \coloneqq \bigoplus_{(n_1,\cdots,n_r) \in \mathbb{N}^{r}~,~r\geq1} E(n_1,\cdots,n_r;n) \otimes_{\left(\mathbb{S}_{n_1} \times \cdots \times \mathbb{S}_{n_r}\right)~ \wr ~ \mathbb{S}_r} \left(M(n_1) \otimes \cdots \otimes M(n_r) \right)~.
\]

\begin{lemma}
The colored Schur realization functor $\mathscr{S}_{\mathbb{S}}$ is a strong monoidal functor between the monoidal categories $(\mathsf{pdg}~\mathsf{Col}_\mathbb{S},\circ, \I_{\mathbb{S}})$ and $(\mathsf{End}(\mathsf{pdg}~\smod),\circ, \mathrm{Id})~.$
\end{lemma}

\begin{proof}
The computation is analogous to the classical case.
\end{proof}

\begin{Corollary}
Let $E$ be an $\mathbb{S}$-color scheme. Any $\mathbb{S}$-colored operad structure on $E$ induces a monad structure on $\mathscr{S}_{\mathbb{S}}(E)$ and any $\mathbb{S}$-colored cooperad structure on $E$ induces a comonad structure on $\mathscr{S}_{\mathbb{S}}(E)$.
\end{Corollary}

\begin{Definition}[Algebra over an $\mathbb{S}$-colored operad]
Let $\mathcal{G}$ be a pdg $\mathbb{S}$-colored operad. A $\mathcal{G}$-algebra $M$ amounts to the data $(M,\gamma_M,d_M)$ of an algebra over the monad $\mathscr{S}_{\mathbb{S}}(\mathcal{G})$.
\end{Definition}

\begin{lemma}
Let $\mathcal{G}$ be a pdg $\mathbb{S}$-colored operad. The data of a pdg $\mathcal{G}$-algebra structure $\gamma_M$ on a pdg $\mathbb{S}$-module $(M,d_M)$ is equivalent to a morphism of pdg $\mathbb{S}$-colored operads $\Gamma_M: \mathcal{G} \longrightarrow \mathrm{End}_M$. 
\end{lemma}

\begin{proof}
Straightforward generalization of the standard case.
\end{proof}

\begin{Definition}[Coalgebra over an $\mathbb{S}$-colored cooperad]
Let $\mathcal{V}$ be a pdg $\mathbb{S}$-colored cooperad. A $\mathcal{V}$-coalgebra $K$ amounts to the data $(K,\Delta_K,d_K)$ of a coalgebra over the comonad $\mathscr{S}_{\mathbb{S}}(\mathcal{V})$.
\end{Definition}

The \textit{dual Schur realization functor} $\widehat{\mathscr{S}}_{\mathbb{S}}^c$ also associates to any pdg $\mathbb{S}$-color scheme $E$ an endofunctor in the category of pdg $\mathbb{S}$-modules 
\[
\begin{tikzcd}[column sep=4pc,row sep=0.5pc]
\widehat{\mathscr{S}}_{\mathbb{S}}^c : \mathsf{pdg}~\mathsf{Col}_\mathbb{S}^{\text{op}} \arrow[r]
&\mathsf{End}(\mathsf{pdg}~\smod) \\
E \arrow[r,mapsto]
&\widehat{\mathscr{S}}_{\mathbb{S}}^c(E)(-)~.
\end{tikzcd}
\]
The dual Schur endofunctor $\widehat{\mathscr{S}}_{\mathbb{S}}^c(E)$ is given, for $M$ a pdg $\mathbb{S}$-module, by:
\[
\widehat{\mathscr{S}}_{\mathbb{S}}^c (E)(M)(n) \coloneqq \prod_{(n_1,\cdots,n_r) \in \mathbb{N}^{r}~,~r\geq1} \mathrm{Hom}_{\left(\mathbb{S}_{n_1} \times \cdots \times \mathbb{S}_{n_r}\right)~ \wr ~ \mathbb{S}_r} \left( E(n_1,\cdots,n_r;n), M(n_1) \otimes \cdots \otimes M(n_r) \right)~.
\]

\begin{lemma}
The functor $\widehat{\mathscr{S}}_{\mathbb{S}}^c: (\mathsf{pdg}~\mathsf{Col}_\mathbb{S},\circ, \I_{\mathbb{S}})^\mathsf{op} \longrightarrow (\mathsf{End}(\mathsf{pdg}~\smod),\circ, \mathrm{Id})$ can be endowed with a lax monoidal structure. That is, there exists a natural transformation
\[
\varphi_{E,F}: \widehat{\mathscr{S}}_{\mathbb{S}}^c(E) \circ \widehat{\mathscr{S}}_{\mathbb{S}}^c(F) \longrightarrow \widehat{\mathscr{S}}_{\mathbb{S}}^c(E \circ F)~,
\]
which satisfies associativity and unitality compatibility conditions with respect to the monoidal structures. Furthermore, this natural transformation is a monomorphism for all pdg $\mathbb{S}$-color schemes $E,F$. 
\end{lemma}

\begin{proof}
The construction of $\varphi_{E,F}$ and the proof are completely analogous to \cite[Corollary 3.4]{grignoulejay18}.
\end{proof}

\begin{Corollary}
Let $E$ be an $\mathbb{S}$-color scheme. Any $\mathbb{S}$-colored cooperad structure on $E$ induces a monad structure on $\widehat{\mathscr{S}}_{\mathbb{S}}^c(E)$.
\end{Corollary}

\begin{Definition}[Algebra over an $\mathbb{S}$-colored cooperad]
Let $\mathcal{V}$ be an $\mathbb{S}$-colored cooperad. A $\mathcal{V}$-algebra $L$ amounts to the data $(L,\gamma_L,d_L)$ of an algebra over the monad $\widehat{\mathscr{S}}_{\mathbb{S}}^c(\mathcal{V})$.
\end{Definition}

\begin{Definition}[Coalgebra over an $\mathbb{S}$-colored operad]
Let $\mathcal{G}$ be a pdg $\mathbb{S}$-colored operad. A $\mathcal{G}$-coalgebra $N$ amount to the $(N,\Delta_N,d_N)$ of a morphism of pdg $\mathbb{S}$-modules
\[
\Delta_N: N \longrightarrow \displaystyle \prod_{(n_1,\cdots,n_r) \in \mathbb{N}^{r}~,~r\geq1} \mathrm{Hom}_{\left(\mathbb{S}_{n_1} \times \cdots \times \mathbb{S}_{n_r}\right)~\wr~ \mathbb{S}_r } \left( \mathcal{G}(n_1,\cdots,n_r;n), N(n_1) \otimes \cdots \otimes N(n_r) \right)~,
\]
such that the following diagram commutes 
\[
\begin{tikzcd}[column sep=4.5pc,row sep=3pc]
N \arrow[r,"\Delta_N"] \arrow[d,"\Delta_N",swap] 
&\widehat{\mathscr{S}}_{\mathbb{S}}^c(\mathcal{G})(N) \arrow[r,"\widehat{\mathscr{S}}_{\mathbb{S}}^c(\mathrm{id})(\Delta_N)"]
&\widehat{\mathscr{S}}_{\mathbb{S}}^c(\mathcal{G})(N) \circ \widehat{\mathscr{S}}_{\mathbb{S}}^c(\PP)(D) \arrow[d,"\varphi_{\mathcal{G},\mathcal{G}}(N)"] \\
\widehat{\mathscr{S}}_{\mathbb{S}}^c(\mathcal{G})(N) \arrow[rr,"\widehat{\mathscr{S}}_{\mathbb{S}}^c(\gamma_\mathcal{G})(\mathrm{id})"]
&
&\widehat{\mathscr{S}}_{\mathbb{S}}^c(\mathcal{G} \circ \mathcal{G})(N)~.
\end{tikzcd}
\]
\end{Definition}

\begin{Remark}
Let $\mathcal{G}$ be a pdg $\mathbb{S}$-colored operad. The data of a pdg $\mathcal{G}$-coalgebra structure $\Delta_N$ on a pdg $\mathbb{S}$-module $(N,d_N)$ is equivalent to a morphism of pdg $\mathbb{S}$-colored operads $\delta_N: \mathcal{G} \longrightarrow \mathrm{Coend}_N$. 
\end{Remark}

\medskip

\subsection{Partial $\mathbb{S}$-colored (co)operads}\label{subsection: comparison for partial colored (co)operads}
As the careful reader might expect by now, all the other definitions of Section \ref{Section: Recollections} generalize to the $\mathbb{S}$-colored setting \textit{mutatis mutandis}.

\begin{Definition}[Partial $\mathbb{S}$-colored operad]
An \textit{pdg partial} $\mathbb{S}$-\textit{colored operad} amounts to the data of $(\mathcal{G},\{\circ_i\},d_\mathcal{G})$ a pdg $\mathbb{S}$-color scheme $\mathcal{G}$ endowed with partial composition operations 
\[
\circ_i : \mathcal{G}(n_1,\cdots, n_i, \cdots, n_r;n) \otimes_{\mathbb{S}_{n_i}} \mathcal{G}(p_1,\cdots,p_l;n_i) \longrightarrow \mathcal{G}(n_1,\cdots,n_{i-1},p_1,\cdots,p_l,n_{i+1},\cdots,n_r;n)~.
\]
This family of partial compositions maps $\{\circ_i\}$ satisfies sequential and parallel axioms analogous to those of a partial operad. They satisfy an equivariance condition with respect to the permutations of the entries $\varphi_\sigma$ which is also completely analogous to Definition \ref{def: partialoperad}. 
\end{Definition}

\begin{Definition}[Unital partial $\mathbb{S}$-colored operad]
A \textit{pdg unital partial} $\mathbb{S}$-colored operad amounts to the data of $(\mathcal{G},\{\circ_i\},\eta,d_\mathcal{G})$ a partial pdg $\mathbb{S}$-colored operad $(\mathcal{G},\{\circ_i\},d_\mathcal{G})$ together with a morphism of pdg $\mathbb{S}$-color schemes $\eta: \I_{\mathbb{S}} \longrightarrow \mathcal{G}$ which acts as a unit for the partial compositions maps. 
\end{Definition}

The data of a unit $\eta: \I_{\mathbb{S}} \longrightarrow \mathcal{G}$ amounts to a family of elements $\{id_n \}$ in $\mathcal{G}(n;n)$. 

\begin{lemma}\label{lemma: add a unit colored case}
The category of pdg unital partial $\mathbb{S}$-colored operads is equivalent to the category of pdg $\mathbb{S}$-colored operads defined as monoids.
\end{lemma}

\begin{proof}
Straightforward generalization of Proposition \ref{prop: unital partial operads are operads}.
\end{proof}

Thus, given a pdg partial $\mathbb{S}$-colored operad $(\mathcal{G},\{\circ_i\},d_\mathcal{G})$ one can obtain a pdg $\mathbb{S}$-colored operad defined as a monoid by freely adding a unit to it $\mathcal{G}^{u} \coloneqq \mathcal{G} \oplus \I_\mathbb{S}$.

\begin{theorem}[{\cite[Theorems 2.10, 2.11 and 2.24]{ward19}}]
Let $\mathbb{V}$ be a groupoid. 
\begin{enumerate}
\item The exists a monad $\mathscr{T}_\mathbb{V}$ in the category of pdg $\mathbb{V}$-color schemes, called the $\mathbb{V}$-colored tree monad, such that the category of $\mathscr{T}_\mathbb{V}$-algebras is equivalent to the category of pdg unital partial $\mathbb{V}$-colored operads.

\item  The exists a monad $\overline{\mathscr{T}}_\mathbb{V}$ in the category of pdg $\mathbb{V}$-color schemes, called the reduced $\mathbb{V}$-colored tree monad, such that the category of $\overline{\mathscr{T}}_\mathbb{V}$-algebras is equivalent to the category of pdg partial $\mathbb{V}$-colored operads.
\end{enumerate}
\end{theorem}

\begin{Definition}[Partial $\mathbb{S}$-colored cooperad]
A \textit{pdg partial} $\mathbb{S}$-\textit{colored cooperad} amounts to the data of $(\mathcal{V},\{\Delta_i\},d_\mathcal{V})$ a pdg $\mathbb{S}$-color scheme $\mathcal{V}$ endowed with partial decomposition operations 
\[
\Delta_i :\mathcal{V}(n_1,\cdots,n_{i-1},p_1,\cdots,p_l,n_{i+1},\cdots,n_r;n) \longrightarrow \mathcal{V}(n_1,\cdots, n_i, \cdots, n_r;n) \otimes_{\mathbb{S}_{n_i}} \mathcal{V}(p_1,\cdots,p_l;n_i)~.
\]
This family of partial decompositions maps $\{\Delta_i\}$ satisfies cosequential and coparallel axioms analogous to those of a partial cooperad. They satisfy an equivariance condition with respect to the permutations of the entries $\varphi_\sigma$ which is also completely analogous to Definition \ref{def: partialcoop}. 
\end{Definition}

\begin{Definition}[Counital partial $\mathbb{S}$-colored cooperad]
A \textit{pdg counital partial} $\mathbb{S}$-colored cooperad  amounts to the data of $(\mathcal{V},\{\Delta_i\},\epsilon, d_\mathcal{V})$ a pdg partial $\mathbb{S}$-colored cooperad $(\mathcal{V},\{\Delta_i\},d_\mathcal{V})$ together with a morphism of pdg $\mathbb{S}$-color schemes $\epsilon: \mathcal{V} \longrightarrow \I_{\mathbb{S}}$ which acts as a counit for the partial decompositions maps. 
\end{Definition}

Any pdg partial $\mathbb{S}$-colored cooperad structure induces a morphism of pdg $\mathbb{S}$-color schemes
\[
\Delta_\mathcal{V}: \mathcal{V} \longrightarrow \overline{\mathscr{T}}^\wedge_\mathbb{S}(\mathcal{V})~,
\]
where $\overline{\mathscr{T}}^\wedge_\mathbb{S}$ is the endofunctor given by the completion of the reduced $\mathbb{V}$-colored tree monad with respect to its canonical weight filtration given by the number of internal edges of the rooted trees. Using this morphism one defines an analogous version of the coradical filtration of a partial cooperad and an analogous definition of a conilpotent pdg partial $\mathbb{S}$-colored cooperad. \textit{Mutatis mutandis} the same characterization of this type of partial cooperads still holds. See Section \ref{subsection: filtrations on partial (co)operads}.

\begin{theorem}[{\cite[Section 2.4.1]{ward19}}]
Let $\mathbb{V}$ be a groupoid such that for any $v$ in $\mathbb{V}$, the set $\mathrm{Aut}(v)$ is finite.
\begin{enumerate}
\item The exists a comonad structure on the underlying endofunctor of the reduced tree monad $\overline{\mathscr{T}}_\mathbb{V}$ on the category of $\mathbb{V}$-colored schemes.

\item The category of $\overline{\mathscr{T}}_\mathbb{V}$-coalgebras is equivalent to the category of conilpotent pdg partial $\mathbb{S}$-colored cooperads.
\end{enumerate}
\end{theorem}

Let $(\mathcal{V},\{\Delta_i\},d_\mathcal{V})$ be a conilpotent pdg partial cooperad. Then one can cofreely add a counit to it by setting $\mathcal{V}^{u} \coloneqq \mathcal{V} \oplus \I_{\mathbb{V}}$. The pdg $\mathbb{S}$-color scheme has an unique cooperad structure induced by the partial decompositions maps $\{\Delta_i\}$. It defines a functor 
\[
\mathsf{Conil}: \mathsf{pdg}~\mathsf{pCoop}_\mathbb{S}^\mathsf{conil} \longrightarrow \mathsf{pdg}~\mathsf{Coop}_\mathbb{S}~.
\]
\begin{Definition}[Conilpotent $\mathbb{S}$-colored cooperad]
A pdg $\mathbb{S}$-colored cooperad $\mathcal{V}$ is said to be \textit{conilpotent} if its in the essential image of the functor defined above.
\end{Definition}

\begin{Remark}
For any conilpotent pdg $\mathbb{S}$-colored cooperad $\mathcal{V}$, there is a canonical filtration on $\mathcal{V}$-algebras and a notion of \textit{complete} $\mathcal{V}$-algebra analogous to the case studied in Section \ref{Section: Recollections}.
\end{Remark}

\subsection{Koszul duality for quadratic $\mathbb{S}$-colored operads}\label{subsection: Koszul duality for quadratic S-colored operads}
We now state the quadratic Koszul duality of \cite{ward19} in the context of $\mathbb{S}$-colored operads. In order to do this, we briefly restrict from the underlying category of pdg modules to the underlying category of dg modules. 

\begin{Proposition}\label{prop S colored totalization}
Let $(\mathcal{G},\{\circ_i\},d_\mathcal{G})$ be a dg partial $\mathbb{S}$-colored operad. The \textit{totalization} of $\mathcal{G}$ is given by 
\[
\prod_{(n_1,\cdots,n_r;n) \in \mathbb{N}^{r+1}} \left(\mathcal{G}(n_1,\cdots,n_r;n)^{\left(\mathbb{S}_{n_1} \times \cdots \times \mathbb{S}_{n_r}\right)~\wr~ \mathbb{S}_r}\right)^{\mathbb{S}_n}~.
\]
It can be endowed with a dg pre-Lie algebra structure. Let $\mu$ be in $\mathcal{G}(n_1,\cdots,n_r;n)$ and $\nu$ be in $\mathcal{G}(p_1,\cdots,p_l;p)$, where $\mu$ is colored by $(\sigma_1, \cdots, \sigma_r;\sigma)$ in $\mathbb{S}_{n_1} \times \cdots \times \mathbb{S}_{n_r} \times \mathbb{S}_n$ and where $\nu$ is colored by $(\tau_1, \cdots, \tau_l;\tau)$ in $\mathbb{S}_{p_1} \times \cdots \times \mathbb{S}_{p_l} \times \mathbb{S}_p$. The pre-Lie bracket of $\mu$ and $\nu$ is given by
\[
\mu \star \nu \coloneqq \sum_{\{(n_i,\sigma_i) = (p,\tau)\}} \sum_{\sigma_P} (\mu \circ_i \nu)^{\sigma_P}~,
\]
where the first sum runs over all pairs $(n_i,\sigma_i)$ which are equal to $(p,\tau)$ (if there are none, then $\mu \star \nu \coloneqq 0$), and where the second sum anges over unshuffles $\sigma_P$ associated to all ordered partitions $P$ in $\mathrm{OrPar}(1,\cdots, 1, \underbrace{l-i+1}_{i\text{-th position}}, 1, \cdots, 1)$.
\end{Proposition}

\begin{proof}
The associator of the product $\star$ is right symmetric like in the case of partial operads, and is compatible with the differential by definition.
\end{proof}

\begin{Definition}[Convolution partial $\mathbb{S}$-colored operad]
Let $(\mathcal{V},\{\Delta_i\}, d_\mathcal{V})$ be a conilpotent dg partial $\mathbb{S}$-colored cooperad and let $(\mathcal{G},\{\circ_i\},d_\mathcal{G})$ be a dg partial $\mathbb{S}$-colored operad. The dg $\mathbb{S}$-color scheme
\[
\mathcal{H}om(\mathcal{V},\mathcal{G})(n_1,\cdots,n_r;n) \coloneqq \mathrm{Hom}_{\mathsf{dg}}(\mathcal{V}(n_1,\cdots,n_r;n), \mathcal{G}(n_1,\cdots,n_r;n))
\]
endowed with the differential 
\[
\partial(\alpha) \coloneqq d_{\mathcal{G}} \hspace{2pt} \circ \hspace{2pt} \alpha - (-1)^{|\alpha|} \alpha \hspace{2pt} \circ \hspace{2pt} d_{\mathcal{V}}~.
\]
forms a dg partial $\mathbb{S}$-colored operad structure where the partial compositions maps are given by 
\[
\alpha \circ_i \beta \coloneqq \circ_i \cdot(\alpha \otimes \beta) \cdot \Delta_i~.
\]
\end{Definition}

\begin{Definition}[Twisting morphism]
Let 
\[
\mathfrak{g}_{\mathcal{V},\mathcal{G}} \coloneqq \prod_{(n_1,\cdots,n_r;n) \in \mathbb{N}^{r+1}} \mathrm{Hom}_{\mathbb{S}\text{-}\mathsf{col}}(\mathcal{V}(n_1,\cdots,n_r;n), \mathcal{G}(n_1,\cdots,n_r;n))
\]
be the dg pre-Lie algebra given by the totalization of the convolution operad of $\mathcal{V}$ and $\mathcal{G}$. A \textit{twisting morphism} is a Maurer-Cartan element $\alpha$ of $\mathfrak{g}_{\mathcal{V},\mathcal{G}}$, that is, a morphism $\alpha: \mathcal{V} \longrightarrow \mathcal{G}$ of dg $\mathbb{S}$-color schemes of degree $-1$  satisfying :
\[
\partial(\alpha) + \alpha \star \alpha = 0~.
\]
The set of twisting morphism between $\mathcal{V}$ and $\mathcal{G}$ is denoted by $\mathrm{Tw}(\mathcal{V},\mathcal{G})~.$
\end{Definition}

Twisting morphisms induce Bar-Cobar adjunctions relative to them.

\begin{Proposition}[Bar-Cobar adjunction relative to $\alpha$]\label{Bar-Cobar colored adjunction, classical case}
Let $\alpha: \mathcal{V} \longrightarrow \mathcal{G}$ be a twisting morphism between a conilpotent partial $\mathbb{S}$-colored cooperad and a partial $\mathbb{S}$-colored operad. It induces a Bar-Cobar adjunction
\[
\begin{tikzcd}[column sep=5pc,row sep=3pc]
          \mathsf{dg}~\mathcal{G}^{u}\text{-}\mathsf{alg} \arrow[r, shift left=1.1ex, "\Omega_\alpha"{name=F}] &\mathsf{dg}~\mathcal{V}^{u}\text{-}\mathsf{coalg}^{\mathsf{conil}} \arrow[l, shift left=.75ex, "\text{B}_\alpha"{name=U}]
            \arrow[phantom, from=F, to=U, , "\dashv" rotate=-90]
\end{tikzcd}
\]
relative to the twisting morphism $\alpha$. 
\end{Proposition}

\begin{proof}
This adjunction is constructed using the free $\mathcal{G}^{u}$-algebra given by $\mathscr{S}_{\mathbb{S}}(\mathcal{G}^{u})$ and the cofree conilpotent $\mathcal{V}^{u}$-coalgebra given by $\mathscr{S}_{\mathbb{S}}(\mathcal{V}^{u})$. They are endowed with differentials which are analogous to those constructed in the standard case \cite[Section 11.2]{LodayVallette12}. See \cite[Section 2.7.1]{ward19} for a detailed exposition of this construction.
\end{proof}

\begin{Definition}[Koszul twisting morphism]
Let $\alpha: \mathcal{V} \longrightarrow \mathcal{G}$ be a twisting morphism between a conilpotent partial $\mathbb{S}$-colored cooperad and a   partial $\mathbb{S}$-colored operad. It is a \textit{Koszul twisting morphism} if the twisted complex $\mathcal{V}^{u} \circ_\alpha \mathcal{G}^{u}$ is acyclic or if the twisted complex $\mathcal{G}^{u} \circ_\alpha \mathcal{V}^{u}$ is acyclic. 
\end{Definition}

See \cite[Section 2.7]{ward19} or \cite[Section 6.4.5]{LodayVallette12} for a precise definition of the twisted complexes $\mathcal{G}^{u} \circ_\alpha \mathcal{V}^{u}$ and $\mathcal{V}^{u} \circ_\alpha \mathcal{G}^{u}$.

\begin{Remark}
There is a natural isomorphism $\Omega_\alpha \circ \mathrm{B}_\alpha \cong \mathscr{S}_\mathbb{S}(\mathcal{G}^{u} \circ_\alpha \mathcal{V}^{u})$. This implies that $\alpha$ is Koszul, then the counit of adjunction 
\[
\Omega_\alpha \circ \mathrm{B}_\alpha \longrightarrow \mathrm{Id}_{\mathcal{G}\text{-}\mathsf{alg}}
\]
is a natural quasi-isomorphism. The converse is true under some assumptions, see \cite[Theorem 11.3.3]{LodayVallette12}.
\end{Remark}

\begin{lemma}[{\cite[Section 2.48]{ward19}}]\label{Koszul duality Ward}
Let $\alpha: \mathcal{V} \longrightarrow \mathcal{G}$ be a twisting morphism between a conilpotent partial $\mathbb{S}$-colored cooperad and a partial $\mathbb{S}$-colored operad. The following are equivalent:
\begin{enumerate}
\item The twisting morphism $\alpha$ is a Koszul twisting morphism.
\item The morphism $f_\alpha: \Omega_{\mathbb{S}}\mathcal{V} \longrightarrow \mathcal{G}$ induced by $\alpha$ is a quasi-isomorphism.
\item The morphism $g_\alpha: \mathcal{V} \longrightarrow \mathrm{B}_{\mathbb{S}}\mathcal{G}$ induced by $\alpha$ is a quasi-isomorphism.
\end{enumerate}
Here $\mathrm{B}_{\mathbb{S}}$ and $\Omega_{\mathbb{S}}$ denote respectively the Bar and Cobar constructions for partial groupoid-colored (co)operads constructed in \cite{ward19}.
\end{lemma}

\begin{Definition}[{\cite[Section 2.5.3]{ward19}}]
Let $(\mathcal{G}, \{\circ_i\}, d_\mathcal{G})$ be a partial $\mathbb{S}$-colored operad. A \textit{quadratic presentation} of $\mathcal{G}$ is a pair $(E,R)$ such that there is an isomorphism of partial $\mathbb{S}$-colored operads
\[
\mathcal{G} \cong \overline{\mathscr{T}}_{\mathbb{S}}(E)/(R)~,
\]
where $\overline{\mathscr{T}}_{\mathbb{S}}(E)$ denotes the free partial $\mathbb{S}$-colored operad on the $\mathbb{S}$-color scheme $E$ modulo the quadratic relations $R \subset \overline{\mathscr{T}}_{\mathbb{S}}(E)^{(2)}~.$
\end{Definition}

\begin{Definition}[{\cite[Section 2.5.4]{ward19}}]
Let $(\mathcal{G}, \{\circ_i\}, d_\mathcal{G})$ be a partial $\mathbb{S}$-colored operad and $(E,R)$ be a quadratic presentation of $\mathcal{G}$. The \textit{Koszul dual conilpotent partial} $\mathbb{S}$-\textit{colored cooperad} $\mathcal{G}^{\ac}$ of $\mathcal{G}$ is given by the conilpotent partial $\mathbb{S}$-colored cooperad 
\[
\mathcal{G}^{\ac} \coloneqq \overline{\mathscr{T}}_{\mathbb{S}}^c(sE,s^2R)~,
\]
where $\overline{\mathscr{T}}_{\mathbb{S}}^c(sE,s^2R)$ denotes the cofree conilpotent partial $\mathbb{S}$-colored cooperad  cogenerated by the $\mathbb{S}$-color scheme $sE$, with $s^2R$ as the corelations. (Smallest sub-cooperad of $\overline{\mathscr{T}}_{\mathbb{S}}^c(sE)$ containing $s^2R$). Here $s$ denotes the suspension of graded $\mathbb{S}$-color schemes.
\end{Definition}

It comes equipped with a canonical twisting morphism $\kappa: \mathcal{G}^{\ac} \longrightarrow \mathcal{G}$ given by
\[
\mathcal{G}^{\ac} \twoheadrightarrow sE \cong E \hookrightarrow \mathcal{G}~. 
\]

\begin{Definition}[Koszul partial $\mathbb{S}$-colored operad]\label{def: Koszul quadratic S-colored operad}
Let $(\mathcal{G}, \{\circ_i\}, d_\mathcal{G})$ be a partial $\mathbb{S}$-colored operad. The partial $\mathbb{S}$-colored operad $\mathcal{G}$ is said to be a Koszul quadratic $\mathbb{S}$-colored operad if there exists a quadratic presentation $(E,R)$ of $\mathcal{G}$ such that the canonical twisting morphism 
\[
\kappa: \mathcal{G}^{\ac} \twoheadrightarrow sE \cong E \hookrightarrow \mathcal{G}~. 
\]
is a Koszul twisting morphism.
\end{Definition}

\subsection{The partial $\mathbb{S}$-colored operad encoding partial (co)operads and its Koszul dual}
We now define the partial $\mathbb{S}$-colored operad $\mathcal{O}$ that encodes partial (co)operads as its (co)algebras. This is a direct generalization of the set-colored operad encoding non-symmetric partial (co)operads defined in \cite[Definition 4.1]{VanderLaan03}. The ground category is still dg modules in this section. 

\begin{Definition}[The partial $\mathbb{S}$-colored operad encoding partial (co)operads]\label{The colored operad O}
The partial $\mathbb{S}$-colored operad $\mathcal{O}$ is given by the following quadratic presentation: 
\[
\mathcal{O} \coloneqq \overline{\mathscr{T}}_{\mathbb{S}}(E)/(R)~,
\]
where $\overline{\mathscr{T}}_{\mathbb{S}}(E)$ denotes the free partial $\mathbb{S}$-colored operad on the $\mathbb{S}$-color scheme $E$ modulo the quadratic relations $R \subset \overline{\mathscr{T}}_{\mathbb{S}}(E)^{(2)}$. The $\mathbb{S}$-color scheme $E$ is the generated by:
\[
\gamma_i^{n,k} \in E(n,k;n+k-1) \hspace{1pc} \textrm{for} \hspace{1pc} 1 \leq i \leq n~,
\]
as a right $\mathbb{K}[\mathbb{S}_n \times \mathbb{S}_k] \otimes_\mathbb{K} \mathbb{K}[\mathbb{S}_{n+k-1}]^{\text{op}}$-module, endowed with the zero differential.

\medskip

The right action of the symmetric groups $\mathbb{S}_n$ and $\mathbb{S}_k$ can be written as in both cases as the left action of an element $\mathbb{S}_{n+k-1}$. For $\sigma$ in $\mathbb{S}_k$, its action is equal to 
\[
\sigma \cdot (\gamma_i^{n,k}) = (\gamma_i^{n,k}) \cdot \sigma' \hspace{1pc} \textrm{for} \hspace{1pc} 1 \leq i \leq n~,
\]
where $\sigma' \in \mathbb{S}_{n+k-1}$ is the unique permutation that acts as the identity everywhere except for $\{i,\cdots,i+k-1\}$, where it acts as $\sigma$. For $\tau$ in $\mathbb{S}_n$, its action is equal to
\[
\tau \cdot (\gamma_i^{n,k}) = (\gamma_{\tau(i)}^{n,k}) \cdot \tau' \hspace{1pc} \textrm{for} \hspace{1pc} 1 \leq i \leq n~,
\]
where $\tau'$ is the unique permutation of $\mathbb{S}_{n+k-1}$ that acts as $\tau$ on the block $\{1,\cdots,n+n-1\} - \{i,\cdots, i+k-1\}$ with values in $\{1,\cdots,n+k-1\} - \{ \tau(i),\cdots, \tau(i)+k-1\}$ and as the identity on $\{i,\cdots,i+k-1\}$ with values in $\{\tau(i), \cdots, \tau(i) + k -1 \}~.$ The $\mathbb{S}_2$ action that permutes the entries is simply given by $\varphi_{(12)} (\gamma_i^{n,k}) \coloneqq \gamma_i^{k,n}~.$

\medskip

The relations $R$ are the given by: 
\[
\begin{tikzcd}
&\left\{ 
    \arraycolsep=1.4pt\def\arraystretch{1.8}\begin{array}{lllll}
         \gamma_{i+j-1}^{n+k-1,l} \circ_1 \gamma_i^{n,k}= \gamma_i^{l+n-1,k} \circ_1 \gamma_j^{n,l} \hspace{1pc}\mathrm{if} \hspace{1pc} 1 \leq i \leq n, 1 \leq j \leq k \hspace{1pc}~, (1)\\
         \gamma_{j+k-1}^{n+k-1,l} \circ_1 \gamma_i^{n,k} = \gamma_j^{n+k-1,l}\circ_1 \gamma_i^{n,k} \hspace{1pc}\mathrm{if} \hspace{1pc} 1\leq i < j \leq n \hspace{1pc} (2)~.
    \end{array}
\right.
\end{tikzcd}
\]
The relation $(1)$ is the \textit{parallel} axiom and $(2)$ is the \textit{sequential} axiom. 
\end{Definition}

\begin{Remark}
The equivariance relations of the partial (de)compositions of partial (co)operads are not coded as relations in the partial $\mathbb{S}$-colored operad $\mathcal{O}$. They are encoded by the action of the symmetric groups on the generators $\{\gamma_i^{n,k}\}$.

\medskip

Although similar in spirit, it does not coincide with the approach of \cite{BrunoMalte}, where the authors encode partial operads using a \textit{set}-colored operad, and where the action of the symmetric groups is given by operations in their set-colored operad.
\end{Remark}

By adding freely a unit to $\mathcal{O}$ one obtains an $\mathbb{S}$-colored operad $\mathcal{O}^u$ defined as a monoid.

\begin{Proposition}\label{Prop: O-algebres}
The category of dg $\mathcal{O}^u$-algebras is equivalent to the category of dg partial operads.
\end{Proposition}

\begin{proof}
Let $(P,d_P)$ be a dg-$\mathbb{S}$-module and let $\Gamma_P: \mathcal{O}^{u} \longrightarrow \mathrm{End}_P$ be a morphism of dg $\mathbb{S}$-colored operads. It is determined by the image of the generators $\gamma_i^{n,k}$ in $\mathcal{O}(n,k;n+k-1)$. Let's denote $\circ_i^{n,k} \coloneqq \Gamma_P(\gamma_i^{n,k}): P(n) \otimes P(k) \longrightarrow P(n+k-1)$. The family $\{\circ_i^{n,k}\}$ satisfies the equivariance axiom of a partial operad since $\Gamma_P$ is $\kk[\mathbb{S}_n \times \mathbb{S}_k] \otimes \kk[\mathbb{S}_{n+k-1}]$-equivariant; it satisfies the parallel and sequential axioms by definition of $\mathcal{O}$. Since $\Gamma_P$ commutes with the pre-differentials, $d_P$ is a derivation with respect to the partial composition maps. Hence it endows $(P,d_P)$ with a dg partial operad structure. 
\end{proof}

\begin{Proposition}\label{Prop: O-coalgebras}
The category of dg $\mathcal{O}^u$-coalgebras is equivalent to the category of dg partial cooperads.
\end{Proposition}

\begin{proof}
The proof is analogous to the proof of Proposition \ref{Prop: O-algebres}. One replaces the endomorphism operad with the coendomorphism operad.
\end{proof}

\begin{lemma}
The Koszul dual conilpotent partial $\mathbb{S}$-colored cooperad $\mathcal{O}^{\ac}$ is isomorphic to the operadic suspension of the linear dual of $\mathcal{O}$. The operad $\mathcal{O}$ is Koszul autodual. 
\end{lemma}

\begin{proof}
First, since each $\mathcal{O}(n_1,\cdots,n_r;n)$ is finite dimensional over $\kk$ and since it is \textit{reduced} in the sense of \cite[Definition 2.35]{ward19}, the arity-wise linear dual of $\mathcal{O}$ is a conilpotent partial $\mathbb{S}$-colored operad, and we have that $(\mathcal{O}^*)^* \cong \mathcal{O}$. Therefore it is only necessary to compute its Koszul dual partial $\mathbb{S}$-colored operad, which is given by the operadic suspension of $(\mathcal{O}^{\ac})^*$.

\medskip

We adapt the proof of \cite[Theorem 4.3]{VanderLaan03} to the groupoid-colored case. The Koszul dual operad $\mathcal{O}^!$ is given by $\overline{\mathscr{T}}_{\mathbb{S}}(E^*)/(R^{\bot})$, where $\overline{\mathscr{T}}_{\mathbb{S}}(E)$ is the free partial $\mathbb{S}$-colored partial operad generated by $E^*$, and $(R^{\bot})$ denotes the operadic ideal generated by the orthogonal of $R$ inside $\overline{\mathscr{T}}_{\mathbb{S}}(E^*)^{(2)}$. Using the same arguments as in the set-colored case, one can show that the dimension of the relations $R$ as a $\mathbb{S}$-color scheme is exactly half of the dimension over $\kk$ of $\overline{\mathscr{T}}_{\mathbb{S}}(E)^{(2)}$. Since, the relations $R$ must be contained in the ideal generated by the orthogonal, concludes by a dimension argument that the relations in $R$ form a basis of $R^\bot$. 
\end{proof}

\begin{lemma}\label{lemma with the trees bijection}
The following statements hold: 
\begin{enumerate}
\item There is an isomorphism of monads in the category of $\mathbb{S}$-modules between the monad $\mathscr{S}_{\mathbb{S}}(\mathcal{O}^{u})$ and the reduced tree monad $\overline{\mathscr{T}}$ encoding partial operads.

\item The is an isomorphism of comonads in the category of $\mathbb{S}$-modules between the comonad $\mathscr{S}_{\mathbb{S}}((\mathcal{O}^*)^{u})$ and the reduced tree comonad $\overline{\mathscr{T}}^c$ encoding conilpotent partial cooperads.
\end{enumerate}
\end{lemma}

\begin{proof}
The first results follows immediately from the fact that the category of $\mathcal{O}$-algebras is equivalent to the category of partial operads. Nevertheless, for any $\mathbb{S}$-module $M$, there is an explicit bijection between the $\mathbb{S}$-modules $\mathscr{S}_{\mathbb{S}}(\mathcal{O}^{u})(M)$ and $\overline{\mathscr{T}}(M)$ that identifies the partial operad structures. For any element 
\[
(\psi;m_1, \cdots m_r) \quad \text{in} \quad \mathcal{O}(n_1,\cdots,n_r;n) \allowbreak \otimes_{\left(\mathbb{S}_{n_1} \times \cdots \times \mathbb{S}_{n_r}\right)~\wr~ \mathbb{S}_r} M(n_1) \otimes \cdots \otimes M(n_r)~,
\]
one can represent $\psi$ as a equivalence class of binary trees with vertices labeled by the generators $\gamma_i^{n,k}$ and with edges labeled by elements of the symmetric groups that match the coloring of the edge. To any $(\psi;m_1, \cdots m_r)$, one associates the rooted tree $\tau_\psi$ obtained by composing the $n_i$-corollas labeled by $m_i$ in the way the binary tree $\psi$ indicates, applying at each step the permutations that label the edges of $\psi$. Pictorially, the bijection is given by:

\begin{center}
\includegraphics[width=130mm,scale=1.7]{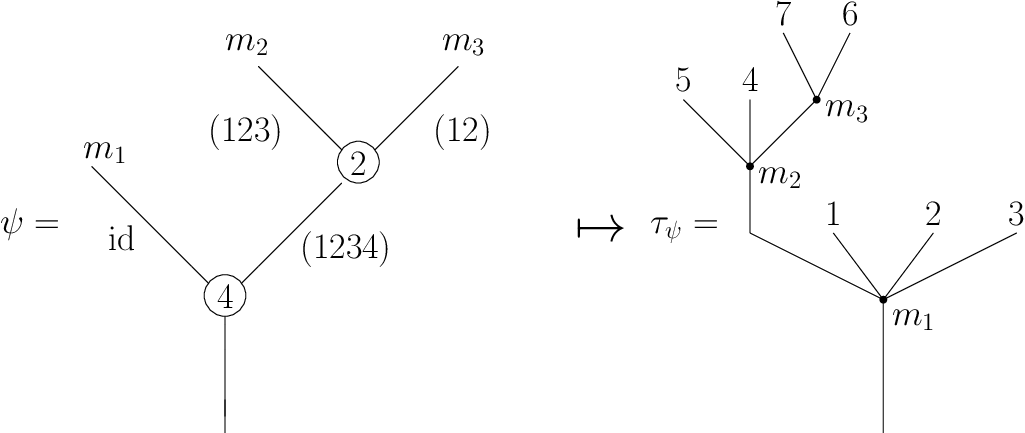}
\end{center}

This type of bijections can be found in \cite[Section 1.3]{dehlingvallette15}). It is straightforward to check that is bijection identifies the partial operad structures of $\mathscr{S}_{\mathbb{S}}(\mathcal{O}^{u})(M)$ and $\overline{\mathscr{T}}(M)$. Moreover, this bijection on the underlying $\mathbb{S}$-modules of $\mathscr{S}_{\mathbb{S}}((\mathcal{O}^*)^{u})(M)$ and $\overline{\mathscr{T}}^c(M)$ identifies the two conilpotent partial cooperad structures as well.
\end{proof}

\begin{Proposition}\label{Prop les alge/cog conil sur la coop O*}
Let $\mathcal{O}^{\ac}$ be the Koszul dual conilpotent partial $\mathbb{S}$-colored cooperad of $\mathcal{O}$.

\begin{enumerate}
\item The category of dg $(\mathcal{O}^{\ac})^{u}$-coalgebras is equivalent to the category of shifted conilpotent dg partial cooperads.

\item The category of complete dg $(\mathcal{O}^{\ac})^{u}$-algebras is equivalent to the category of shifted complete dg absolute partial operads.
\end{enumerate}
\end{Proposition}

\begin{proof}
The first statement follows from the previous Lemma. For a definition of \textit{absolute partial operads}, as well as their characterization, we refer to the Appendix \ref{Appendix B}.
\end{proof}

\begin{Proposition}
Let $\kappa: \mathcal{O}^{\ac} \longrightarrow \mathcal{O}$ be the canonical twisting morphism. The adjunction induced by the twisting morphism $\kappa$
\[
\begin{tikzcd}[column sep=5pc,row sep=3pc]
          \mathsf{dg}~\partialop \arrow[r, shift left=1.1ex, "\Omega_\kappa"{name=F}] &\mathsf{dg}~\partialcoop^{\mathsf{conil}} \arrow[l, shift left=.75ex, "\mathrm{B}_\kappa"{name=U}]
            \arrow[phantom, from=F, to=U, , "\dashv" rotate=-90]
\end{tikzcd}
\]
between dg partial operads and conilpotent dg partial cooperads is naturally isomorphic to the classical Bar-Cobar adjunction of \cite[Section 6.5]{LodayVallette12}.
\end{Proposition}

\begin{proof}
In order to check that two adjunctions are isomorphic, it is only necessary to check that there exists a natural isomorphism between left adjoints, since the mate of this natural isomorphism will also induce an natural isomorphism of right adjoints. The isomorphism between the reduced tree monad $\overline{\mathscr{T}}$ and the monad $\mathscr{S}_\mathbb{S}(\mathcal{O})$ induces a natural isomorphism of partial operads between $\Cobar$ and $\Cobar_\kappa$. It is straightforward to check that this bijection commutes with the respective differentials of the two Cobar constructions.
\end{proof}

\begin{theorem}\label{the Koszulity of O}
The canonical twisting morphism $\kappa: \mathcal{O}^{\ac} \longrightarrow \mathcal{O}$ is a Koszul twisting morphism. Hence the partial $\mathbb{S}$-colored operad $\mathcal{O}$ is Koszul.
\end{theorem}

\begin{proof}
We extend the arguments of \cite[Theorem 11.3.3]{LodayVallette12} to this specific case. Indeed, we know that for any dg partial operad $(\PP,\{\circ_i\},d_\PP)$, the counit $\epsilon_\kappa: \Cobar_\kappa \text{B}_\kappa \PP \longrightarrow \PP$ of the adjunction induced by $\kappa$ is a quasi-isomorphism of dg partial operads. In particular, any dg $\mathbb{S}$-module can be endowed with the trivial dg partial operad structure given by the family of partial composition maps $\{ 0 \}$. Therefore we conclude that $\mathcal{O} \circ_\kappa \mathcal{O}^{\ac}$ and $\mathcal{O}^{\ac} \circ_\kappa \mathcal{O}$ are both acyclic. 
\end{proof}

\begin{Remark}
Generalizing the formalism of \cite{grignoulejay18}, one can construct a complete Bar-Cobar adjunction relative to the twisting morphism $\kappa: \mathcal{O}^{\ac} \longrightarrow \mathcal{O}$ between complete dg absolute partial operads and non-necessarily conilpotent dg partial cooperads. 
\end{Remark}

\subsection{Inhomogeneous $\mathbb{S}$-colored Koszul duality}
We introduce the unital partial $\mathbb{S}$-colored operad $u\mathcal{O}$ which encodes (co)unital partial (co)operads as its (co)algebras. In order to compute its Koszul dual conilpotent curved $\mathbb{S}$-colored, we generalize the inhomogeneous Koszul duality of \cite{HirshMilles12} to the $\mathbb{S}$-colored case. This provides us with a Koszul dual conilpotent curved partial $\mathbb{S}$-colored cooperad $c\mathcal{O}^\vee$. It encodes conilpotent curved partial cooperads as its curved coalgebras and \textit{complete curved absolute partial operads} as its complete curved algebras. Absolute partial operads are a new type of operad-like structure, for which we provide useful descriptions in the Appendix \ref{Appendix B}. The ground category considered is once again the category of pdg modules. 

\medskip

Let $U$ be the $\mathbb{S}$-colored scheme given by $U(0;1) \coloneqq \kk.u$, where $u$ is an element of degree $0$, and zero elsewhere. 

\begin{Definition}[The unital partial $\mathbb{S}$-colored operad encoding (co)unital partial (co)operads]\label{def colored operad uO}
Let $(E,R)$ be the quadratic presentation of the partial $\mathbb{S}$-colored operad $\mathcal{O}$. The unital partial $\mathbb{S}$-colored operad $u\mathcal{O}$ is given by
\[
u\mathcal{O} \coloneqq \mathscr{T}_{\mathbb{S}}(E \oplus U)/(R')~,
\]
where $\mathscr{T}_{\mathbb{S}}(E \oplus U)$ is the free unital partial $\mathbb{S}$-colored operad generated by the $\mathbb{S}$-colored scheme $E \oplus U$, and where $(R')$ is the operadic ideal generated by $R'$. Here $R'$ is the sub-$\mathbb{S}$-color scheme of $\mathscr{T}_{\mathbb{S}}(E \oplus U)^{(\leq 2)}$ given by $R$, together with the additional relations 
\[
\begin{tikzcd}
\left\{ 
    \arraycolsep=1.4pt\def\arraystretch{1.8}\begin{array}{lllll}
         \gamma_{1}^{1,n} \circ_1 u = |_n \hspace{1pc}\mathrm{for} \hspace{1pc} n \in \mathbb{N}~,\\
         \gamma_{i}^{n,1} \circ_2 u = |_n \hspace{1pc}\mathrm{for}  \hspace{1pc} 1 \leq i \leq n \hspace{1pc} \mathrm{and} \hspace{1pc} n \in \mathbb{N}~,
    \end{array}
\right.
\end{tikzcd}
\]
where $|_n \in \mathscr{T}_{\mathbb{S}}(E \oplus U)^{(0)}(n;n)$ is the trivial tree. 
\end{Definition}

\begin{lemma}
We have that:
\begin{enumerate}
\item The category of dg $u\mathcal{O}$-algebras is equivalent to the category of dg unital partial operads.

\item The category of dg $u\mathcal{O}$-coalgebras is equivalent to the category of dg counital partial cooperads.
\end{enumerate}
\end{lemma}

\begin{proof}
It follows immediately from Proposition \ref{Prop: O-algebres} and Proposition \ref{Prop: O-coalgebras}.
\end{proof}

\begin{Definition}[Inhomogeneous quadratic presentation]
Let $(\mathcal{G},\{\circ_i\},\eta)$ be a unital partial $\mathcal{S}$-colored operad and let 
\[
\mathcal{G} \cong \mathscr{T}_{\mathbb{S}}(V)/(S)
\]
be a presentation of $\mathcal{G}$, where $S \subset \{ \I_\mathbb{S} \oplus V \oplus \mathscr{T}_{\mathbb{S}}(V)^{(2)} \}$. The presentation is a \textit{inhomogeneous quadratic presentation} if the $\mathbb{S}$-color scheme of relations $S$ satisfies the following conditions.

\begin{enumerate}
\item That the space of generators is \textit{minimal}, that is $S \cap \{\I_{\mathbb{S}} \oplus V\} = \{0\}$.

\item That the space of relations is \textit{maximal}, that is $(S) \cap \{\I_{\mathbb{S}} \oplus V \oplus \mathscr{T}_{\mathbb{S}}(V)^{(2)} \} = S$. 
\end{enumerate}
\end{Definition}

\begin{lemma}
The presentation of the unital partial $\mathbb{S}$-colored operad $u\mathcal{O}$ given in its definition is a \textit{inhomogeneous quadratic presentation}.
\end{lemma}

\begin{proof}
It is straightforward to check given the aforementioned presentation.
\end{proof}

\begin{Definition}[Curved partial $\mathbb{S}$-colored cooperad]\label{def: curved colored partial cooperad}
A \textit{curved partial} $\mathbb{S}$-\textit{colored cooperad} $(\mathcal{V},\{\Delta_i\}, \allowbreak d_\mathcal{V},\Theta_\mathcal{V})$ amounts to the data of a partial pdg $\mathbb{S}$-colored operad $(\mathcal{V},\{\Delta_i\},d_\mathcal{V})$ and a morphism of pdg $\mathbb{S}$-color schemes $\Theta_\mathcal{V}: (\mathcal{V},d_\mathcal{V}) \longrightarrow (\I_\mathbb{S},0)$ of degree $-2$, such that the following diagram commutes: 
\[
\begin{tikzcd}[column sep=7.5pc,row sep=3pc]
\mathcal{V} \arrow[r,"\Delta_{(1)}"] \arrow[rrd,"d_\mathcal{V}^2", bend right =10]
&\mathcal{V} \circ_{(1)} \mathcal{V} \arrow[r,"(\mathrm{id}~ \circ ~ \Theta_\mathcal{V})~-~(\Theta_\mathcal{V}~ \circ_{(1)}~ \mathrm{id})~"] 
&(\mathcal{V} \circ \I_\mathbb{S}) \oplus (\I_\mathbb{S} \circ \mathcal{V}) \cong  \mathcal{V} \oplus \mathcal{V}  \arrow[d,"\mathrm{proj}"]\\
&
&\mathcal{V}~,
\end{tikzcd}
\]
where $\mathrm{proj}$ is given by $\mathrm{proj}(\mu,\nu) \coloneqq \mu + \nu$.
\end{Definition}

\begin{Remark}
The definitions developed in Section \ref{Section: Curved cooperads} generalize \textit{mutatis mutandis} to the $\mathbb{S}$-colored case. 
\end{Remark}

We extend the formalism of semi-augmented operads of \cite{HirshMilles12} to the $\mathbb{S}$-colored case in order to compute the Koszul dual conilpotent curved partial $\mathbb{S}$-colored cooperad of $u\mathcal{O}$. 

\begin{Definition}[Semi-augmented unital partial $\mathbb{S}$-colored operad]
A \textit{semi-augmented} unital partial $\mathbb{S}$-colored operad $(\mathcal{G},\{\circ_i\},\eta,\iota)$ is the data of a unital partial $\mathbb{S}$-colored operad $(\mathcal{G},\{\circ_i\}, \allowbreak \eta,d_\mathcal{G})$ together with a morphism of $\mathbb{S}$-color schemes $\iota: \mathcal{G} \longrightarrow \I_{\mathbb{S}}$ of degree $0$ such that $\iota \cdot \eta = \mathrm{id}_{\I_\mathbb{S}}~.$
\end{Definition}

\begin{Remark}
The unital partial $\mathbb{S}$-colored operad $u\mathcal{O}$ is canonically semi-augmented by the identity morphism of $\I_\mathbb{S}$, we denote this semi-augmentation by $\iota_{u\mathcal{O}}$.
\end{Remark}

\begin{Proposition}\label{Proposition: S-colored inhomogeneous Koszul duality}
Let $(\mathcal{G},\{\circ_i\},\eta,\iota)$ be a semi-augmented unital partial $\mathcal{S}$-colored operad that admits a inhomogeneous quadratic presentation $(V,S)$. Let $qS \coloneqq S \cap \mathscr{T}_{\mathbb{S}}(V)^{(2)}$. Let $\varphi: qS \longrightarrow \I_\mathbb{S} \oplus V$ be the linear map that gives $S$ as its graph. The Koszul dual conilpotent partial $\mathbb{S}$-colored cooperad of $\mathcal{G}$ is given by 
\[
\mathcal{G}^{\ac} \coloneqq \overline{\mathscr{T}}_{\mathbb{S}}^c(sV, s^2qR)~.
\]
It is endowed with a coderivation of degree $-1$ $d_{\mathcal{G}^{\ac}}$ given by the unique extension of:
\[
\begin{tikzcd}[column sep=4pc,row sep=1pc]
\mathcal{G}^{\ac} \arrow[r,twoheadrightarrow]
&s^2qR \arrow[r,"s^{-1} \varphi_1 "]
&sE~,
\end{tikzcd}
\]
and with a curvature $\Theta_{\mathcal{G}^{\ac}}$ given by the degree $-2$ map: 
\[
\begin{tikzcd}[column sep=4pc,row sep=1pc]
\mathcal{G}^{\ac} \arrow[r,twoheadrightarrow]
&s^2qR \arrow[r,"s^{-2} \varphi_0 "]
&\I_{\mathbb{S}}~.
\end{tikzcd}
\]
The data of $(\mathcal{G}^{\ac}, d_{\mathcal{G}^{\ac}}, \Theta_{\mathcal{G}^{\ac}})$ forms a conilpotent curved partial $\mathbb{S}$-colored cooperad. 
\end{Proposition}

\begin{proof}
The proof is completely analogous to the non-$\mathbb{S}$-colored case developed in \cite[Section 4]{HirshMilles12}.
\end{proof}

\begin{Definition}[Koszul unital partial $\mathbb{S}$-colored operad]
Let $(\mathcal{G},\{\circ_i\},\eta,\iota)$ be a semi-augmented unital partial $\mathcal{S}$-colored operad that admits a inhomogeneous quadratic presentation $(V,S)$. Let $qS = S \cap \mathscr{T}_{\mathbb{S}}(V)^{(2)}$ and let 
\[
q\mathcal{G} \coloneqq \overline{\mathscr{T}}(V)/(qS)
\]
be the quadratic partial $\mathbb{S}$-colored operad associated to $\mathcal{G}$. The semi-augmented unital partial $\mathcal{S}$-colored operad $\mathcal{G}$ is said to be \textit{Koszul} if the quadratic operad $q\mathcal{G}$, endowed with the quadratic presentation $(V,qS)$, is a Koszul quadratic partial $\mathbb{S}$-colored in the sense of Definition \ref{def: Koszul quadratic S-colored operad}.
\end{Definition}

Let $J$ be the graded $\mathbb{S}$-color scheme given by $J(0;1) \coloneqq \kk.\theta$, where $\theta$ is an element of degree $-2$, and zero elsewhere.

\begin{Definition}[Curved partial $\mathbb{S}$-colored $c\mathcal{O}^\vee$]\label{def: cO vee}
Let $(E,R)$ be the quadratic presentation of the partial $\mathbb{S}$-colored operad $\mathcal{O}$. The \textit{curved partial} $\mathbb{S}$-\textit{colored cooperad} $c\mathcal{O}^\vee$ is given by the presentation 
\[
c\mathcal{O}^\vee \coloneqq \overline{\mathscr{T}}^c_{\mathbb{S}}(E \oplus J,R)~,
\]
where $E \oplus J$ are the cogenerators and $R$ are the corelations. It is endowed with the curvature $\Theta_{c\mathcal{O}^\vee}: c\mathcal{O}^\vee \longrightarrow \I_{\mathbb{S}}$ defined on $c\mathcal{O}^\vee(n;n)$ by the following map
\[
\left\{ \begin{tikzcd}[column sep=1.5pc,row sep=0pc]
\gamma_1^{1,n} \circ_1 \theta - \sum_{i=0}^n \gamma_i^{n,1} \circ_2 \theta \arrow[r,mapsto]
&\mathrm{id}_n ~,\\
\mu \arrow[r,mapsto]
&0~, \\
\end{tikzcd}
\right.
\]
if $\mu$ is not contained in the sub-$\mathbb{S}_n$-module generated by $\gamma_1^{1,n} \circ_1 \theta - \sum_{i=0}^n \gamma_i^{n,1} \circ_2 \theta$.
\end{Definition}

\begin{theorem}\label{Koszulity of uO}
Let $u\mathcal{O}$ be the unital partial $\mathbb{S}$-colored operad encoding (co)unital partial (co)operads. 

\begin{enumerate}
\item The Koszul dual conilpotent curved partial $\mathbb{S}$-colored cooperad $(u\mathcal{O})^{\ac}$ is isomorphic to the suspension of $c\mathcal{O}^\vee$.

\vspace{0.5pc}

\item The unital partial $\mathbb{S}$-colored operad $u\mathcal{O}$ is a Koszul, meaning that $q(u\mathcal{O})$ is a Koszul quadratic partial $\mathbb{S}$-colored operad.
\end{enumerate}
\end{theorem}

\begin{proof}
The Koszul dual cooperad $(u\mathcal{O})^{\ac}$ is given by $\overline{\mathscr{T}}^c_{\mathbb{S}}(s(E \oplus U), s^2 qR')$. It is cogenerated by a degree $1$ operation $u \in sU(0;1)$ and the suspension of the family $\{\gamma_i^{n,k}\}$. The corelations $s^2 qR'$ are given by $s^{2}R$. 

\medskip

Let us compute the rest of the structure. The projection of $R'$ into $E \oplus U$ is zero, hence $(u\mathcal{O})^{\ac}$ has a zero pre-differential. The projection of $R'$ into $\I_{\mathbb{S}}$ will be non zero only on the relations that involve the element $u \in U(0;1)$. It is straightforward to compute that the pre-image of $\mathrm{id}_n \in \I_{\mathbb{S}}$ by the projection $s^{-2} \varphi_0$ is given by 
\[
\sum_{i=0}^n s\gamma_{i}^{n,1} \circ_2 su - s\gamma_{1}^{1,n} \circ_1 su~.
\]
This completely characterizes the curvature of $(u\mathcal{O})^{\ac}$. A direct inspection identifies $(u\mathcal{O})^{\ac}$ with the operadic suspension of $c\mathcal{O}^\vee$.

\medskip

Let us prove the second assertion. The space $qR'$ is in fact $R$, and therefore $q(u\mathcal{O})$ is given by the coproduct $\mathcal{O} \coprod \{u\}$. By Theorem \ref{the Koszulity of O}, we know that $\mathcal{O}$ is a quadratic Koszul $\mathbb{S}$-colored operad. Therefore the Koszul complex of $q(u\mathcal{O})$ is also acyclic by an argument analogous to that of \cite[Proposition 6.16]{HirshMilles12}.
\end{proof}

The last point of this section is to understand what curved coalgebras and complete curved algebras over the conilpotent curved partial $\mathbb{S}$-colored cooperad $c\mathcal{O}^\vee$ are. 

\begin{Proposition}\label{Prop: cO-cogebres.}
The category of curved $(c\mathcal{O}^\vee)^u$-coalgebras is isomorphic to the category of conilpotent curved partial cooperads.
\end{Proposition}

\begin{proof}
First, notice that, for any pdg $\mathbb{S}$-module $M$, there is an isomorphism of pdg $\mathbb{S}$-modules 
\[
\mathscr{S}_{\mathbb{S}}(c\mathcal{O}^\vee)(M) \cong \overline{\mathscr{T}}^c(M \oplus \nu), 
\]
where $\nu$ is an arity $1$ and degree $-2$ generator. Indeed, this can by shown by extending the bijection given in the proof of Lemma \ref{lemma with the trees bijection}. This extension is defined by sending $\theta$ in $c\mathcal{O}^\vee(0;1)$ to the generator $\nu$. Pictorially it is given by
\begin{center}
\includegraphics[width=110mm,scale=1.3]{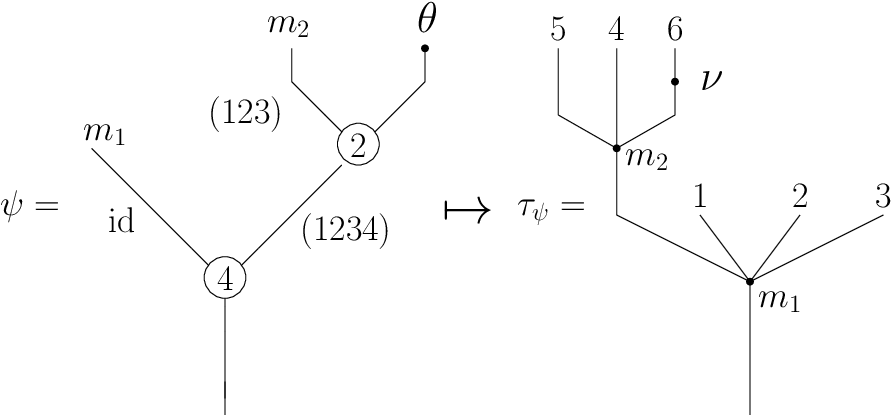}.
\end{center}
This isomorphism is natural in $M$. Therefore the endofunctor $\overline{\mathscr{T}}^c(- \oplus \nu)$ in the category of pdg $\mathbb{S}$-modules as a comonad structure. A direct computation shows that this comonad structure coincides with the reduced tree comonad structure. 

\medskip

Let $(\C,\{\Delta_i\},d_\C,\Theta_\C)$ be a conilpotent curved partial cooperad. Let 
\[
\Delta_\C: \C \longrightarrow \overline{\mathscr{T}}^c(\C)
\]
be its structure map as a coalgebra over the reduced tree comonad. We construct an extension
\[
\Delta_\C^+: \C \longrightarrow \overline{\mathscr{T}}^c(\C) \times \overline{\mathscr{T}}^c(\nu) \cong \overline{\mathscr{T}}^c(\C \oplus \nu)~,
\]
using $\Theta_\C: \C \longrightarrow \nu.\I$ and its unique extension to $\overline{\mathscr{T}}^c(\nu)$. 

\medskip

The data $(\C,\Delta_\C^+,d_\C)$ forms a pdg $(c\mathcal{O}^\vee)^u$-coalgebra. It is in fact a curved $(c\mathcal{O}^\vee)^u$-coalgebra. Indeed, the diagram 
\[
\begin{tikzcd}[column sep=3pc,row sep=3pc]
\C \arrow[r,"\Delta_\C^+ "] \arrow[rd,"d_\C^2",swap]
&\mathscr{S}_\mathbb{S}(c\mathcal{O}^\vee)(\C) \arrow[d,"\mathscr{S}_\mathbb{S}(\Theta_{c\mathcal{O}^\vee})(\mathrm{id})"]\\
&\C \cong \mathscr{S}_\mathbb{S}(\I_\mathbb{S})(\C)
\end{tikzcd}
\]
commutes since $(\C,\{\Delta_i\},d_\C,\Theta_\C)$ forms a \textit{curved} partial cooperad. The other way around, given a curved $(c\mathcal{O}^\vee)^u$-coalgebra $(\C, \Delta_C^+, d_\C)$, one can compose the structural map
\[
\Delta_\C^+: \C \longrightarrow \overline{\mathscr{T}}^c(\C \oplus \nu)
\]
with the projection $\overline{\mathscr{T}}^c(\C \oplus \nu) \twoheadrightarrow \overline{\mathscr{T}}^c(\C)$, which endows $\C$ with a conilpotent partial cooperad structure $\{\Delta_i\}$. By composing $\Delta_\C^+$ with the projection $\overline{\mathscr{T}}^c(\C \oplus \nu) \twoheadrightarrow \I.\nu$, one obtains a map 
\[
\Theta_\C: \C \longrightarrow \I
\]
of pdg $\mathbb{S}$-modules of degree $-2$. The data $(\C,\{\Delta_i\},d_\C,\Theta_\C)$ forms a conilpotent curved partial cooperad, since $(\C, \Delta_C^+, d_\C)$ is a \textit{curved} $(c\mathcal{O}^\vee)^u$-coalgebra.
\end{proof}

\begin{Proposition}\label{Prop: cO-algebres.}
The category of complete curved $(c\mathcal{O}^\vee)^u$-algebras is isomorphic to the category of complete curved absolute partial operads.
\end{Proposition}

\begin{proof}
For the proof of this statement, see Proposition \ref{Prop: cO algebres vrai}. For an explicit description of these objects, we refer to the Appendix \ref{Appendix B}.
\end{proof}

\medskip

\section{Curved twisting morphisms and Bar-Cobar adjunctions at the operadic level}\label{Section: Constructions Bar-Cobar operadiques}
The curved Koszul duality established in the previous section at the groupoid-colored level gives a \textit{curved twisting morphism} $\kappa$ between the groupoid-colored curved cooperad encoding curved partial (co)operads $c\mathcal{O}^\vee$ and the groupoid-colored operad encoding unital partial (co)operads $u\mathcal{O}$. This curved twisting morphism induces two different Bar-Cobar adjunctions.

\medskip

The first Bar-Cobar adjunction induced by $\kappa$ is between dg unital partial operads and conilpotent curved partial cooperads. It will be shown to be isomorphic to the Bar-Cobar adjunction defined in \cite[Section 4.1]{grignou2019}: 

\hspace{9.3pc}
\begin{tikzcd}[column sep=5pc,row sep=3pc]
            \mathsf{curv}~\mathsf{pCoop}^{\mathsf{conil}} \arrow[r, shift left=1.1ex, "\Omega"{name=F}] &\mathsf{dg}~\mathsf{upOp}~.  \arrow[l, shift left=.75ex, "\text{B}"{name=U}]
            \arrow[phantom, from=F, to=U, , "\dashv" rotate=-90]
\end{tikzcd}

On the other hand, using the generalization of \cite{grignoulejay18} to the groupoid-colored setting, we obtain a new complete Bar-Cobar adjunction:

\hspace{10.5pc}
\begin{tikzcd}[column sep=6pc,row sep=3pc]
            \mathsf{dg}~\mathsf{upCoop} \arrow[r, shift left=1.1ex, "\widehat{\Omega}"{name=F}] &\mathsf{curv}~\mathsf{abs}~\mathsf{pOp}^{\mathsf{comp}}~, \arrow[l, shift left=.75ex, "\widehat{\text{B}}"{name=U}]
            \arrow[phantom, from=F, to=U, , "\dashv" rotate=-90]
\end{tikzcd}

between complete curved \textit{absolute} partial operads and counital partial cooperads. Constructing this second adjunction is the main goal of this section.

\medskip

\subsection{Curved twisting morphisms and curved pre-Lie algebras}
In order to defined curved twisting morphisms in the first place, we introduce a new type of structure, called \textit{curved pre-Lie algebras}. In this type of algebras, the curvature has to satisfy a special condition, called the \textit{left-nucleus condition}. This condition comes from the deformation theory. More precisely, it appears in \cite[Proposition 1.1, Chapter 4]{DSV18} for the following reason: one can show that a dg pre-Lie algebra is \textit{twistable} by a Maurer-Cartan element if and only if this element satisfies the extra condition of being left-nucleus.   

\begin{Definition}[Curved pre-Lie algebra]\label{curvedprelie}
A \textit{curved pre-Lie algebra} $(\mathfrak{g}, \{-,-\},d_{\mathfrak{g}},\vartheta)$ amounts to the data of a pre-Lie algebra $(\mathfrak{g},\{-,-\})$, a derivation $d_{\mathfrak{g}}$ with respect to $\{-,-\}$ of degree $-1$, and a morphism of pdg modules of degree $-2$ $\vartheta: \mathbb{K} \longrightarrow \mathfrak{g}$. The data of this morphism is equivalent to the data of an element $\vartheta(1) \coloneqq \vartheta$ in $\mathfrak{g}_{-2}$. They are subject to the following conditions.
\begin{enumerate}
    \item The element $\vartheta$ has to be \textit{left-nucleus}, that is, for all $\mu,\nu$ in $\mathfrak{g}$: 
    \[
    \{\vartheta,\{\mu,\nu\}\} = \{\{\vartheta,\mu\},\nu\}~.
    \]
    \item Moreover, for all $\mu$ in $\mathfrak{g}$:
    \[
    d_{\mathfrak{g}}(\mu) = \{\vartheta,\mu\} - \{\mu,\vartheta\}~.
    \]
    \item And finally, $d_{\mathfrak{g}}(\vartheta)=0$.
\end{enumerate}
\end{Definition}

In the same spirit as for curved Lie algebras and curved associative algebras, one can define a curved partial operad, $cp\mathcal{L}ie$, that encodes curved pre-Lie algebras. Let $H$ be the pdg $\mathbb{S}$-module $(\mathbb{K}.\vartheta, 0, \mathbb{K}[\mathbb{S}_2].\nu, 0, \cdots)$ endowed with the zero pre-differential.

\begin{Definition}[Curved operad encoding curved pre-Lie algebras]
The curved partial operad $cp\mathcal{L}ie$ is given by the presentation: 
\[
cp\mathcal{L}ie \coloneqq \overline{\mathscr{T}}(H)/(D)~,
\]
where $(D)$ is the operadic ideal generated by the following relations: 
\begin{enumerate}
\item The right pre-Lie relation, already present in the classical pre-Lie operad, given by:
\[
\includegraphics[width=120mm,scale=1]{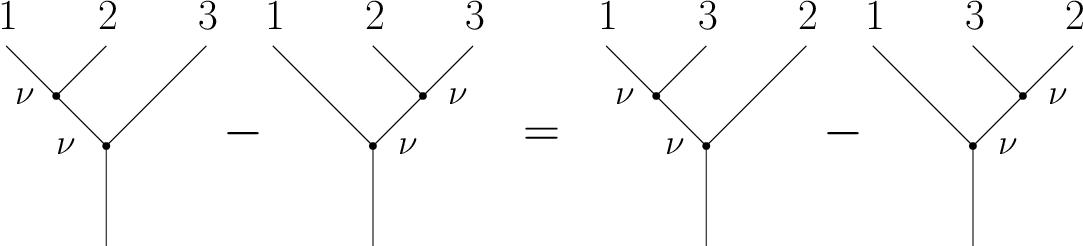}~.
\]
\item The left-nucleus relation stating for the curvature $\vartheta$, given by:
\[
\includegraphics[width=60mm,scale=1]{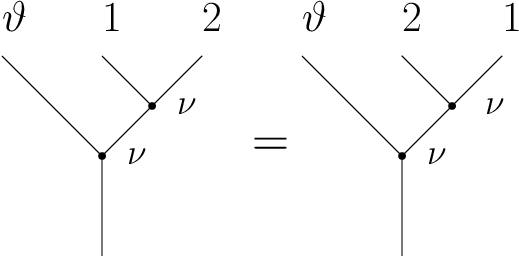}~.
\]
\end{enumerate}

It is endowed with the curvature $\Theta_{cp\mathcal{L}ie}(\mathrm{id}) \coloneqq \nu \circ_{1} \vartheta - \nu \circ_2 \vartheta~.$
\end{Definition}

\begin{Proposition}
The data $(cp\mathcal{L}ie, 0, \Theta_{cp\mathcal{L}ie})$ forms a curved partial operad. The category of curved $cp\mathcal{L}ie$-algebras is equivalent to the category of curved pre-Lie algebras. 

\medskip

Furthermore, the morphism of curved partial operads $ \Liec \longrightarrow \Assc$ given by the skew-symmetrization of the Lie bracket factors through the curved partial operad $cp\mathcal{L}ie$.
\end{Proposition}

\begin{proof}
Let us show that the data forms a curved partial operad. The proof is quite similar to that of Lemma \ref{lemmalie}. In order to check that $[\nu \circ_{1} \vartheta - \nu \circ_2 \vartheta,-] = 0$, since it is a derivation with respect partial compositions, it is enough to test this equality on the generators. First, $[\nu \circ_{1} \vartheta - \nu \circ_2 \vartheta, \vartheta] = 0$ is evident. Secondly, expanding $[\nu \circ_{1} \vartheta - \nu \circ_2 \vartheta, \nu]$ makes six terms that appear: it is straightforward to check that four of them cancel because of the right pre-Lie relation and the last two because of the left-nucleus relation. Proving that this curved partial operad encodes curved pre-Lie algebras is completely analogous to Lemma \ref{lemmalie}. One checks that the map $\Liec \longrightarrow cp\mathcal{L}ie$ given on generators by $\beta \mapsto \nu - \nu^{(12)}$ and $\zeta \mapsto \vartheta$ is indeed a morphism of curved partial operads. Furthermore, there is a morphism of curved partial operads $cp\mathcal{L}ie \longrightarrow \Assc$ simply given by $\nu \mapsto \mu$ and $\vartheta \mapsto \phi$. Their composition gives back the morphism constructed in Lemma \ref{assliem}.
\end{proof}

\begin{Remark}
Contrary to the case of Lie algebras and associative algebras, we had to introduce a new relation into the curved partial operad that encodes curved pre-Lie algebras. In the previous cases, this meant that the curved homotopy version of them was given simply by adding a curvature to the classical homotopy version. For instance, a curved $\mathcal{A}_{\infty}$-algebra is just an $\mathcal{A}_{\infty}$-algebra with an added structure of a curvature; the relations of a curved $\mathcal{A}_{\infty}$-algebra are clearly analogous to those of an $\mathcal{A}_{\infty}$-algebra. In the curved pre-Lie case, more structure will appear when one resolves the left-nucleus relation up to homotopy. Nevertheless, since the left-nucleus relation is not even quadratic, one would first need an expanded Koszul duality in order to treat this case.
\end{Remark}

We leave it to the reader to generalize the definition of curved partial operads to the $\mathbb{S}$-colored setting. (See Definition \ref{def: curved colored partial cooperad} for a similar definition).

\begin{lemma}[Totalization of a curved partial $\mathbb{S}$-colored operad]\label{lemma S colored totalization}
Let $(\mathcal{G}, \{\circ_i\}, d_{\mathcal{G}}, \Theta_\mathcal{G})$ be a curved partial $\mathbb{S}$-colored operad. The \textit{totalization} of $\mathcal{G}$ given by 
\[
\prod_{(n_1,\cdots,n_r;n) \in \mathbb{N}^{r+1}} \left(\mathcal{G}(n_1,\cdots,n_r;n)^{\left(\mathbb{S}_{n_1} \times \cdots \times \mathbb{S}_{n_r}\right)~\wr~ \mathbb{S}_r}\right)^{\mathbb{S}_n}
\]
forms a curved pre-Lie algebra, where the bracket is defined like in Proposition \ref{prop S colored totalization} and where the curvature is given by 
\[
\vartheta(1) = \sum_{n\geq 0} \theta_{n}~.
\]
\end{lemma}

\begin{proof}
Let $g$ be in $\mathcal{G}(n_1,\cdot,n_r;n)$, we have that $d_\mathcal{G}^2(g) = \theta_n \circ_1 g - \sum_{i=0}^r g \circ_i \theta_{n_i}$ since $\mathcal{G}$ is a curved partial $\mathbb{S}$-colored operad. The bracket $\vartheta(1) \star g = \theta_n \circ_1 g$ since only $\theta_n$ matches the appropriate color. Similarly, $g \star \vartheta(1) = \sum_{i=0}^r g \circ_i \theta_{n_i}$. Since each $\theta_{n}$ is an arity one operation, by the sequential and parallel axioms of a partial $\mathbb{S}$-colored operad, the curvature satisfies the \textit{left-nucleus} relation of a curved pre-Lie algebra.
\end{proof}

\begin{lemma}[Curved convolution partial $\mathbb{S}$-colored operad]\label{lemma curved S colored convolution}
Let $(\mathcal{G},\{\circ_i\},\eta,d_\mathcal{G})$ be a dg unital partial $\mathbb{S}$-colored operad and let $(\mathcal{V},\{\Delta_i\},d_\mathcal{V},\Theta_\mathcal{V})$ be a curved partial $\mathbb{S}$-colored cooperad. The convolution partial $\mathbb{S}$-colored operad $\mathcal{H}om(\mathcal{V},\mathcal{G})$ forms a curved partial $\mathbb{S}$-colored operad endowed with the curvature:
\[
\begin{tikzcd}
\Theta_{\mathcal{H}om}: \mathcal{V} \arrow[r,"\Theta_\mathcal{V}"]
&\I_{\mathbb{S}} \arrow[r,"\eta"]
&\mathcal{G}~.
\end{tikzcd}
\]
\end{lemma}

\begin{proof}
Let $\alpha$ be in $\mathcal{H}om(\mathcal{V},\mathcal{G})(n_1,\cdots,n_r;n)$, we have that:
\[
\partial^2(\alpha) = - (-1)^{2.|\alpha|} \alpha \circ d_\mathcal{V}^2 = \Theta_{\mathcal{H}om} \circ_1 \alpha - \sum_{i=0}^r \alpha \circ_i \Theta_{\mathcal{H}om}~,
\]
since $d_\mathcal{G}^2 = 0$ and $d_\mathcal{V}^2 = (\mathrm{id} \circ_{(1)} \Theta_\mathcal{V} - \Theta_\mathcal{V} \circ \mathrm{id} ) \cdot \Delta_{(1)}~.$
\end{proof}

\begin{Definition}[Maurer-Cartan of a curved pre-Lie]
Let $(\mathfrak{g}, \{-,-\},d_{\mathfrak{g}},\vartheta)$ be a curved pre-Lie algebra. A \textit{Maurer-Cartan element} $\alpha$ is a degree $-1$ element that of $\mathfrak{g}$ that satisfies the following equation:
\[
d_{\mathfrak{g}}(\alpha) + \{\alpha,\alpha\} = \vartheta~ .
\]
\end{Definition}

\begin{Remark}
The set of Maurer-Cartan elements of a curved pre-Lie (or similarly a curved Lie) algebra can be empty. A curvature term $\vartheta \neq 0$ stops $0$ in $\mathfrak{g}_{-1}$ from being a canonical Maurer-Cartan element in $\mathfrak{g}$.
\end{Remark}

\begin{Definition}[Curved $\mathbb{S}$-colored twisting morphism]
Let $(\mathcal{G},\{\circ_i\},\eta,d_\mathcal{G})$ be a unital partial dg $\mathbb{S}$-colored operad and let $(\mathcal{V},\{\Delta_i\},d_\mathcal{V},\Theta)$ be a curved partial $\mathcal{S}$-colored cooperad. A \textit{curved twisting morphism} $\alpha$ between $\mathcal{V}$ and $\mathcal{G}$ is a morphism of pdg $\mathbb{S}$-color schemes $\alpha: \mathcal{V} \longrightarrow \mathcal{G}$ of degree $-1$ that satisfies the Maurer-Cartan equation in the totalization of the curved convolution $\mathbb{S}$-colored operad $\mathcal{H}om(\mathcal{V},\mathcal{G})$. Otherwise stated, $\alpha$ satisfies:
\[
\partial(\alpha) + \alpha \star \alpha = \Theta_{\mathcal{H}om}~.
\]
\end{Definition}

\begin{lemma}\label{lemma: kappa is a curved twisting morphism}
The $\mathbb{S}$-color scheme morphism $\kappa: \mathcal{S}\otimes(c\mathcal{O})^\vee \longrightarrow u\mathcal{O}$ given by
\[
\kappa: \mathcal{S}\otimes(c\mathcal{O}^\vee) \twoheadrightarrow sE \oplus sJ \cong E \oplus U \hookrightarrow u\mathcal{O}
\]
is a curved twisting morphism. Here $\mathcal{S}$ denotes the operadic suspension of $c\mathcal{O}^\vee$.
\end{lemma}

\begin{proof}
Notice that $\partial(\kappa) = 0$ since the pre-differentials are null. Let us show that
\[
\kappa \star \kappa = \Theta_{\mathcal{H}om}~.
\]
The morphism of $\mathbb{S}$-color schemes $\kappa \star \kappa$ is non-zero only on elements of weight two in $c\mathcal{O}^\vee$. These are elements which can informally be written as  $\gamma_i^{n+m-1,k} \circ_1 \gamma_j^{n,m}$ and $\gamma_i^{n,k+m-1} \circ_2 \gamma_j^{k,m}$, or as $\gamma_i^{1,n} \circ_1 \theta$ and $\gamma_i^{n,1} \circ_2 \theta$. One computes that $\kappa \star \kappa$ of the first kind of weight two elements is zero because of the Koszul signs, by using the sequential and parallel relations. On the second kind of weight two elements, $\kappa \star \kappa$ is equal to $\gamma_i^{1,n} \circ_1 u$ and $\gamma_i^{n,1} \circ_2 u$. Using the unital relation in $u\mathcal{O}$, these are both equal to $\mathrm{id}_n$, the operadic unit of the $\mathbb{S}$-colored operad $u\mathcal{O}$. Hence $\kappa \star \kappa$ is equal to $\Theta_{\mathcal{H}om}$. 
\end{proof}

\begin{Notation}\label{not: le dirac S-colore}
Let $f: M \longrightarrow N$ be a morphism of graded $\mathbb{S}$-modules of degree $0$ and $g: M \longrightarrow N$ be a morphism of graded $\mathbb{S}$-modules degree $p$. We denote $\diracComb_{\mathbb{S}}(n_1,\cdots,n_r)(f,g)$ the map
\[
\sum_{i=1}^r f(n_1) \otimes \cdots \otimes f(n_{i-1}) \otimes g(n_i) \otimes f(n_{i+1}) \otimes \cdots \otimes f(n_r) : M(n_1) \otimes \cdots \otimes M(n_r) \longrightarrow N(n_1) \otimes \cdots \otimes N(n_r)
\]
which is an $\mathbb{S}_{n_1} \times \cdots \times \mathbb{S}_{n_r} \wr \mathbb{S}_r$-equivariant morphism of degree $p$. Let $E$ be a graded $\mathbb{S}$-color scheme. The family of maps $\{\diracComb_{\mathbb{S}}(n_1,\cdots,n_r)(f,g)\}$ induces a morphism of graded $\mathbb{S}$-modules of degree $p$:
\[
\mathscr{S}_{\mathbb{S}}(E)(M) \longrightarrow \mathscr{S}_{\mathbb{S}}(E)(N)
\]
by applying $\mathrm{id}_E \otimes \diracComb_{\mathbb{S}}(n_1,\cdots,n_r)(f,g)$ to each component. By a slight abuse of notation, this morphism will be denoted by $\mathscr{S}_{\mathbb{S}}(\mathrm{id}_E)(\diracComb_{\mathbb{S}}(f,g))~.$ Likewise, it induces a morphism of graded $\mathbb{S}$-modules of degree $p$:

\[
\widehat{\mathscr{S}}_{\mathbb{S}}^c (E)(M) \longrightarrow \widehat{\mathscr{S}}_{\mathbb{S}}^c (E)(N)
\]
by applying $\mathrm{Hom}(id_E,\diracComb_{\mathbb{S}}(n_1,\cdots,n_r)(f,g))$ to each component. By a slight abuse of notation, this morphism will be denoted by $\widehat{\mathscr{S}}_{\mathbb{S}}^c(\mathrm{id}_E)(\diracComb_{\mathbb{S}}(f,g))~.$
\end{Notation}

\subsection{Classical Bar-Cobar adjunction relative to $\kappa$}\label{subsection: classical bar-cobar relative to kappa}
The first adjunction induced by $\kappa$ will be an adjunction between dg $u\mathcal{O}$-algebras and curved $\mathcal{S}\otimes(c\mathcal{O}^\vee)^{u}$-coalgebras. That is, between dg unital partial operads and shifted conilpotent curved partial cooperads.

\begin{Definition}[Bar-Cobar constructions relative to $\kappa$]\label{Bar Cobar constructions relative to kappa}
Using $\kappa$, one can define two functor:
\begin{enumerate}
\item Let $(\PP,\gamma_\PP,d_\PP)$ a dg $u\mathcal{O}$-algebra. Its \textit{Bar construction relative to} $\kappa$, denoted by $\mathrm{B}_\kappa\PP$, is given by the cofree $\mathcal{S}\otimes (c\mathcal{O}^\vee)^u$-coalgebra $\mathscr{S}_{\mathbb{S}}(\mathcal{S}\otimes (c\mathcal{O}^\vee)^u)(\PP)$. Its pre-differential $d_{\mathrm{bar}}$ is given by the sum of two terms $d_1$ and $d_2$. The first term is given by 
\[
d_1 \coloneqq \mathscr{S}_{\mathbb{S}}(\mathrm{id})(\diracComb_{\mathbb{S}}(\mathrm{id},d_\PP))~. 
\]
The second term $d_2$ is the unique coderivation extending:
\[
\begin{tikzcd}[column sep=4pc,row sep=3pc]
\mathscr{S}_{\mathbb{S}}(\mathcal{S}\otimes c\mathcal{O}^\vee)(\PP) \arrow[r,"\mathscr{S}_{\mathbb{S}}(\kappa)(\mathrm{id})"]
&\mathscr{S}_{\mathbb{S}}(u\mathcal{O})(\PP) \arrow[r,"\gamma_\PP"]
&\PP~.
\end{tikzcd}
\]
\item Let $(\C,\Delta_\C,d_\C)$ be a curved $\mathcal{S}\otimes c\mathcal{O}^\vee$-coalgebra. Its \textit{Cobar construction relative to} $\kappa$, denoted by $\Omega_\kappa \C$, is given by the free $u\mathcal{O}$-algebra $\mathscr{S}_{\mathbb{S}}(u\mathcal{O})(\C)$. Its differential $d_{\mathrm{cobar}}$ is the sum of two terms $d_1$ and $d_1$. The first term is given by 
\[
d_1 \coloneqq -\mathscr{S}_{\mathbb{S}}(\mathrm{id})(\diracComb_{\mathbb{S}}(\mathrm{id},d_\C))~.
\]
The second term $d_2$ is the unique derivation extending:
\[
\begin{tikzcd}[column sep=4pc,row sep=3pc]
\C \arrow[r,"\Delta_\C"]
&\mathscr{S}_{\mathbb{S}}(\mathcal{S}\otimes c\mathcal{O}^\vee)(\C) \arrow[r,"\mathscr{S}_{\mathbb{S}}(\kappa)(\mathrm{id})"]
&\mathscr{S}_{\mathbb{S}}(u\mathcal{O})(\C)~.
\end{tikzcd}
\]
\end{enumerate}
\end{Definition}

\begin{lemma}
There is an adjunction 

\hspace{7pc}
\begin{tikzcd}[column sep=5pc,row sep=3pc]
            \mathsf{curv}~\mathcal{S}\otimes (c\mathcal{O}^\vee)^u\text{-}\mathsf{coalg} \arrow[r, shift left=1.1ex, "\Omega_\kappa"{name=F}] &\mathsf{dg}~u\mathcal{O}\text{-}\mathsf{alg}~.  \arrow[l, shift left=.75ex, "\text{B}_\kappa"{name=U}]
            \arrow[phantom, from=F, to=U, , "\dashv" rotate=-90]
\end{tikzcd}

\end{lemma}

\begin{proof}
The proof is a minor generalization of a the adjunction induced by a curved twisting morphism at the level of algebras in \cite[Section 5]{HirshMilles12}.
\end{proof}

This adjunction relative to $\kappa$ is in fact isomorphic to the adjunction between conilpotent curved coaugmented cooperads and operads introduced in \cite[Section 4.1]{grignou2019}. Before proving this result, we briefly recall the definition of this adjunction. 

\begin{Definition}[{\cite[Definition 58]{grignou2019}}]
Let $(\PP,\{\circ_i\},\eta, d_\PP)$ be a dg unital partial operad. The \textit{Bar construction} $\mathrm{B}\PP$ of $\PP$ is given by:
\[
\mathrm{B}\PP \coloneqq \left( \overline{\mathscr{T}}^c(s\PP \oplus \nu), d_{\mathrm{bar}} = d_1 + d_2, \Theta_{\mathrm{bar}} \right)~,
\]
where $\overline{\mathscr{T}}^c(s\PP \oplus \nu)$ is the cofree conilpotent partial pdg cooperad generated by the dg $\mathbb{S}$-module $s\PP \oplus \vartheta$. Here $\nu$ is an arity $1$ and degree $-2$ generator. It is endowed the pre-differential $d_{\mathrm{bar}}$, given by the sum of $d_1$ and $d_2$. The term $d_1$ is the unique coderivation extending
\[
\begin{tikzcd}[column sep=4pc,row sep=3pc]
\overline{\mathscr{T}}^c(s\PP \oplus \nu)  \arrow[r, twoheadrightarrow]
&s\PP \arrow[r,"s d_{\PP}"]
&s\PP~.
\end{tikzcd} 
\]
The term $d_2$ comes from the \textit{structure} of a unital partial operad on $\PP$, it is given by the unique coderivation extending 
\[
\begin{tikzcd}[column sep=4pc,row sep=3pc]
\overline{\mathscr{T}}^c(s\PP \oplus \nu) \arrow[r, twoheadrightarrow]
&\I.\nu \oplus (s\PP \circ_{(1)} s\PP) \arrow[r,"s \eta + s^2\gamma_{(1)}"]
&s\PP~.
\end{tikzcd}
\]
It is also endowed with the following curvature:
\[
\begin{tikzcd}[column sep=4pc,row sep=3pc]
\Theta_{\mathrm{bar}}: \overline{\mathscr{T}}^c(s\PP \oplus \nu) \arrow[r,twoheadrightarrow] 
&\I.\nu \arrow[r,"s^{-2}"]
&\I~.
\end{tikzcd}
\]
The resulting Bar construction of $\PP$ forms a conilpotent curved partial cooperad. 
\end{Definition}

\begin{Definition}[{\cite[Definition 60]{grignou2019}}]\label{Def: Cobar de Brice}
Let $(\C,\{\Delta_i\},d_\C,\Theta_\C)$ be a curved partial cooperad. The \textit{Cobar construction} $\Omega \C$ of $\C$ is given by:
\[
\Omega \C \coloneqq \left(\mathscr{T}(s^{-1}\C), d_{\mathrm{cobar}} = d_1 - d_2 \right)~,
\]
where $\mathscr{T}(s^{-1}\C)$ is the free unital partial operad generated by the pdg $\mathbb{S}$-module $s^{-1}\C$. It is endowed with the differential $d_{\mathrm{cobar}}$ given by the difference of $d_1$ and $d_2$. The term $d_1$ is the unique derivation extending
\[
\begin{tikzcd}[column sep=5pc,row sep=3pc]
s^{-1}\C \arrow[r,"s^{-1}d_\C"]
&s^{-1}\C \arrow[r,rightarrowtail]
&\mathscr{T}(s^{-1}\C)~.
\end{tikzcd}
\]
The term $d_2$ comes from the \textit{structure} of a conilpotent curved partial cooperad on $\C$, it is given by the unique derivation extending
\[
\begin{tikzcd}[column sep=5pc,row sep=3pc]
s^{-1}\C \arrow[r,"s^{-2}\Delta_{(1)} + s^{-1}\Theta_\C"]
&(s^{-1}\C \circ_{(1)} s^{-1}\C) \oplus \I \arrow[r,rightarrowtail]
&\mathscr{T}(s^{-1}\C)~.
\end{tikzcd}
\]
The resulting Cobar construction of $\C$ forms a dg unital partial operad.
\end{Definition}

\begin{Remark}
The Cobar constructions $\Omega \C$ is not augmented because the canonical morphism $\mathscr{T}(s^{-1}\C) \twoheadrightarrow \I$ does not commute with the differentials in general. Indeed, $d_{\mathrm{cobar}}(\Theta_\C(\mathrm{id}))$ is the trivial tree $|$. It is therefore augmented if and only if the curvature $\Theta_\C (\mathrm{id})$ is zero.
\end{Remark}

The Bar-Cobar constructions described above also form an adjunction 
\[
\begin{tikzcd}[column sep=5pc,row sep=3pc]
            \mathsf{curv}~\mathsf{pCoop}^{\mathsf{conil}} \arrow[r, shift left=1.1ex, "\Omega"{name=F}] &\mathsf{dg}~\mathsf{upOp}~.  \arrow[l, shift left=.75ex, "\mathrm{B}"{name=U}]
            \arrow[phantom, from=F, to=U, , "\dashv" rotate=-90]
\end{tikzcd}
\]
\begin{Proposition}[Mise en abîme]\label{Mise en abime}
The Bar-Cobar adjunction relative to $\kappa$ given by $\Omega_\kappa \dashv \mathrm{B}_\kappa$ is naturally isomorphic to the Bar-Cobar adjunction $\Omega \dashv \mathrm{B}$ constructed in \cite{grignou2019}.
\end{Proposition}

\begin{proof}
Let $(\PP,\{\circ_i\},\eta, d_\PP)$ be a dg unital partial operad. One uses the bijection given in the proof of Proposition \ref{Prop: cO-cogebres.} to construct a natural isomorphism $\mathrm{B} \PP \cong \mathrm{B}_\kappa \PP~.$
\end{proof}

\begin{Corollary}\label{corollary: koszulness implies Quillen equivalence}
This gives another proof that the Bar-Cobar adjunction $\Omega \dashv \mathrm{B}$ constructed in \cite{grignou2019} is a Quillen equivalence.
\end{Corollary}

\begin{proof}
Since we consider a transferred structure on curved conilpotent partial cooperads, proving that this adjunction is a Quillen equivalence simply amounts to show that for any unital partial operad $\PP$, the unit 
\[
\PP \longrightarrow \Omega \mathrm{B} \PP~,
\]
is a quasi-isomorphism. This follows from the fact that the Koszul complex of the curved twisting morphism $\kappa: \mathcal{S}\otimes(c\mathcal{O}^\vee) \longrightarrow u\mathcal{O}$ that induces this Bar-Cobar adjunction is acyclic, using analogous arguments to \cite[Theorem 2.6]{Vallette}.
\end{proof}

The notion of curved twisting morphism between conilpotent curved partial cooperad and dg unital partial operad is again encoded by a curved pre-Lie algebra.

\begin{lemma}
Let $(\PP,\{\circ_i\},d_\PP,\Theta_\PP)$ be a curved partial operad. The totalization of $\PP$ given by 
\[
\prod_{n \geq 0} \PP(n)^{\mathbb{S}_n}~,
\]
together with its pre-Lie bracket and endowed with the curvature $\vartheta(1) \coloneqq \Theta_\PP(\mathrm{id})$ forms a curved pre-Lie algebra.
\end{lemma}

\begin{proof}
The proof is completely analogous to Lemma \ref{lemma S colored totalization}.
\end{proof}

\begin{lemma}[Curved convolution operad]
Let $(\PP,\{\circ_i\},\eta,d_\PP)$ be a dg unital partial operad and let $(\C,\{\Delta_i\},d_\C,\Theta)$ be a curved partial cooperad. The convolution partial pdg operad $(\mathcal{H}om(\C,\PP),\{\circ_i\},\partial)$ forms a curved partial operad endowed with the curvature given by
\[
\begin{tikzcd}
\Theta_{\mathcal{H}om}(\mathrm{id}): \mathcal{C} \arrow[r,"\Theta"]
&\I \arrow[r,"\eta"]
&\mathcal{P}~.
\end{tikzcd}
\]
\end{lemma}

\begin{proof}
The proof is completely analogous to Lemma \ref{lemma curved S colored convolution}.
\end{proof}

\begin{Definition}[Curved twisting morphism]
Let $(\PP,\{\circ_i\},\eta,d_\PP)$ be a dg unital partial operad and let $(\C,\{\Delta_i\},d_\C,\Theta)$ be a curved partial cooperad. A \textit{curved twisting morphism} $\alpha$ is a Maurer-Cartan element curved pre-Lie algebra given by the curved convolution operad:
\[
\mathfrak{g}_{\C,\PP} \coloneqq \prod_{n\geq 0}\mathrm{Hom}_{\mathbb{S}}(\C(n),\PP(n))~.
\]
This is the data of a  morphism of $\mathbb{S}$-modules $\alpha: \C \longrightarrow \PP$ of a degree $-1$ such that:
\[
\partial(\alpha) + \alpha \star \alpha = \Theta_{\mathcal{H}om}(\mathrm{id})~.
\]
The set of twisting morphism between $\C$ and $\PP$ will be denoted $\mathrm{Tw}(\C,\PP)$. 
\end{Definition}

The set of curved twisting morphism between conilpotent curved partial cooperads and unital partial dg operads defines a bifunctor
\[
\begin{tikzcd}
\mathrm{Tw}(- , -) : \left(\mathsf{curv}~\mathsf{pCoop}^{\mathsf{conil}}\right)^{\text{op}} \times \mathsf{dg}~\mathsf{upOp}  \arrow[r]
&\mathsf{Set}~,
\end{tikzcd}
\]
which is represented on both sides by the Bar-Cobar construction defined before.

\begin{Proposition}[{\cite[Proposition 63]{grignou2019}}]
Let $(\PP,\{\circ_i\},\eta,d_\PP)$ be a dg unital partial operad and let $(\C,\{\Delta_i\},d_\C,\Theta)$ be a conilpotent curved partial cooperad. There are isomorphism:
\[
\mathrm{Hom}_{\mathsf{dg}~\mathsf{upOp}}(\Omega\C,\PP) \cong \mathrm{Tw}(\C,\PP) \cong \mathrm{Hom}_{\mathsf{curv}~\mathsf{pCoop}^{\mathsf{conil}}}(\C,\mathrm{B}\PP)~,
\]
which are natural in $\C$ and $\PP$.
\end{Proposition}

\subsection{Complete Bar-Cobar adjunction relative to $\kappa$}
We introduce a new adjunction, which we call the \textit{complete Bar-Cobar adjunction}, between complete curved absolute partial operads and dg counital partial cooperads. This adjunction is again induced by the curved twisting morphism $\kappa$, using the techniques developed in \cite{grignoulejay18}. In order to construct it, we need a mild generalization of the results in \cite{anelcofree2014}, in which the author only considers set-colored operads. 

\begin{theorem}\label{thm: existence de la cogèbre colibre}
Let $(\mathcal{G},\gamma,\eta,d_\G)$ be a dg $\mathbb{S}$-colored operad. The category of dg $\mathcal{G}$-coalgebras is comonadic, that is, there exists a comonad $(\mathscr{L}(\G),\omega,\xi)$ in the category of dg $\mathbb{S}$-modules such that the category of dg $\mathscr{L}(\G)$-coalgebras is equivalent to the category of dg $\mathcal{G}$-coalgebras.

\medskip

The endofunctor $\mathscr{L}(\G)$ is given by the following pullback:
\[
\begin{tikzcd}[column sep=3pc,row sep=3pc]
\mathscr{L}(\G) \arrow[r,"p_2"] \arrow[d,"p_1",swap] \arrow[dr, phantom, "\ulcorner", very near start]
&\widehat{\mathscr{S}}_\mathbb{S}^c(\G) \circ \widehat{\mathscr{S}}_\mathbb{S}^c(\G) \arrow[d,"\varphi_{\G,\G}"] \\
\widehat{\mathscr{S}}_\mathbb{S}^c(\G) \arrow[r,"\widehat{\mathscr{S}}_\mathbb{S}^c(\gamma)"]
&\widehat{\mathscr{S}}_\mathbb{S}^c(\G \circ \G)
\end{tikzcd}
\]

in the category $\mathsf{End}(\mathsf{dg}~\smod)$. It gives a pullback 
\[
\begin{tikzcd}[column sep=3pc,row sep=3pc]
\mathscr{L}(\G)(M) \arrow[r,"p_2(M)"] \arrow[d,"p_1(M)",swap] \arrow[dr, phantom, "\ulcorner", very near start]
&\widehat{\mathscr{S}}_\mathbb{S}^c(\G)(M) \circ \widehat{\mathscr{S}}_\mathbb{S}^c(\G)(M) \arrow[d,"\varphi_{\G,\G}(M)"] \\
\widehat{\mathscr{S}}_\mathbb{S}^c(\G)(M) \arrow[r,"\widehat{\mathscr{S}}_\mathbb{S}^c(\gamma)(M)"]
&\widehat{\mathscr{S}}_\mathbb{S}^c(\G \circ \G)(M)
\end{tikzcd}
\]
for any dg $\mathbb{S}$-module $M$. Here $p_1$ is a monomorphism since $\varphi_{\G,\G}(M)$ is also a monomorphism. The structural map of the comonad $\omega(M): \mathscr{L}(\G)(M) \longrightarrow \mathscr{L}(\G) \circ \mathscr{L}(\G)(M)$ is given by the map $p_2$ in the previous pullback, as its image lies in $\mathscr{L}(\G) \circ \mathscr{L}(\G)(M)$ instead of $\widehat{\mathscr{S}}_\mathbb{S}^c(\G) \circ \widehat{\mathscr{S}}_\mathbb{S}^c(\G)(M)$. The counit is given of the comonad $\mathscr{L}(\G)$ is given by
\[
\begin{tikzcd}[column sep=4pc,row sep=0.5pc]
\xi(M): \mathscr{L}(\G)(M) \arrow[r,"p_1(M)"] 
&\widehat{\mathscr{S}}_\mathbb{S}^c(\G)(M) \arrow[r,"\widehat{\mathscr{S}}_\mathbb{S}^c(\eta)(id_M)"]
&M~.
\end{tikzcd}
\]
\end{theorem}

\begin{proof}
In order to prove the first result, one has to generalize the framework of \cite{anelcofree2014} from set-colored operads to groupoid-colored operads. This is purely formal: the underlying theory upon which the result is obtained is the formalism of analytic functors and operads developed in \cite{Gambino_2017}. In \textit{loc.cit}, the formalism is already developed in the case where the colors form a category. \textit{Mutatis mutandis}, one can perform the same proofs using this general formalism. Thus the category of dg $\mathcal{G}$-coalgebras is indeed comonadic, by \cite[Theorem 2.7.11]{anelcofree2014}.

\medskip

The second statement is equivalent to claiming that the infinite recursion that defines the comonad $\mathscr{L}(\G)$ stops at the first step. We know this is the case for operads in the category of dg modules. We have to check that the lemmas proven for the category of dg modules also hold for the category of dg $\mathbb{S}$-modules. This comes from the fact that the base field $\kk$ is of characteristic $0$, and hence for any $n$, the $\kk$-algebra $\kk[\mathbb{S}_n]$ is a semi-simple finite dimensional $\kk$-algebra where every module is injective and projective. One can check that the different lemmas of \cite[Section 3]{anelcofree2014} that make the proof work are also true in this context.
\end{proof}

Before defining the complete Bar construction relative to $\kappa$ using the cofree dg-$u\mathcal{O}$-coalgebra $\mathscr{L}(u\mathcal{O})$ construction, one needs to understand coderivations on this object. We generalize the results of \cite[Section 6]{grignoulejay18} to our framework. Let $M$ be a graded $\mathbb{S}$-module, we denote by $\pi_M$ the map given by
\[
\begin{tikzcd}[column sep=5pc,row sep=3pc]
\pi_M: \mathscr{L}(u\mathcal{O})(M)  \arrow[r,"p_1"]
&\widehat{\mathscr{S}}_\mathbb{S}^c(u\mathcal{O})(M) \arrow[r,"\widehat{\mathscr{S}}_\mathbb{S}^c(\eta_{u\mathcal{O}})(\mathrm{id})"]
&M~,
\end{tikzcd} 
\]
where $\eta_{u\mathcal{O}}: \I_\mathbb{S} \longrightarrow u\mathcal{O}$ is the unit of the unital partial $\mathbb{S}$-colored operad $u\mathcal{O}$.

\begin{lemma}\label{lemma: coderivations on the cofree cooperad}
Let $f: M \longrightarrow N$ be a morphism graded $\mathbb{S}$-modules of degree $0$ and $g:M \longrightarrow N$ be a morphism of graded $\mathbb{S}$-modules of degree $p$, then the degree $p$ map: 
\[
\begin{tikzcd}[column sep=4pc,row sep=1pc]
&\widehat{\mathscr{S}}_{\mathbb{S}}^c(\mathrm{id}_{u\mathcal{O}})(\diracComb_{\mathbb{S}}(f,g)): \widehat{\mathscr{S}}_{\mathbb{S}}^c(u\mathcal{O})(M) \arrow[r]
&\widehat{\mathscr{S}}_{\mathbb{S}}^c(u\mathcal{O})(N)
\end{tikzcd}
\]
restricts to a degree $p$ morphism $\mathscr{L}(\mathrm{id})(\diracComb_{\mathbb{S}}(f,g)): \mathscr{L}(u\mathcal{O})(M) \longrightarrow \mathscr{L}(u\mathcal{O})(N)$.
\end{lemma}

\begin{proof}
The proof is essentially the same as \cite[Lemma 6.20]{grignoulejay18}. It is a consequence of the universal property of pullbacks: it suffices to construct a morphism of graded $\mathbb{S}$-modules 
\[
\mathscr{L}(u\mathcal{O})(M) \longrightarrow \widehat{\mathscr{S}}_{\mathbb{S}}^c(u\mathcal{O}) \circ \widehat{\mathscr{S}}_{\mathbb{S}}^c(u\mathcal{O})(N)
\]
that coincides with the morphism $\widehat{\mathscr{S}}_{\mathbb{S}}^c(u\mathcal{O})(\diracComb_{\mathbb{S}}(f,g))$ on $\widehat{\mathscr{S}}_{\mathbb{S}}^c(u\mathcal{O} \circ u\mathcal{O})(N)$. One can check that this morphism is given by 
\[
\widehat{\mathscr{S}}_{\mathbb{S}}^c(\mathrm{id}) \left(\diracComb\left(\widehat{\mathscr{S}}_{\mathbb{S}}^c(\mathrm{id})(f), \allowbreak \widehat{\mathscr{S}}_{\mathbb{S}}^c(\mathrm{id})(\diracComb(f,g))\right)\right) \circ p_2~.\]
\end{proof}

\begin{Proposition}[Coderivations on the cofree construction]
Let $M$ be a graded $\mathbb{S}$-module and let $\mathscr{L}(u\mathcal{O})(M)$ be the cofree graded $u\mathcal{O}$-coalgebra generated by $M$. There is a natural bijection between maps 
\[
\varphi: \mathscr{L}(u\mathcal{O})(M) \longrightarrow M
\]
of degree $p$ and coderivations 
\[
d_\varphi: \mathscr{L}(u\mathcal{O})(M) \longrightarrow \mathscr{L}(u\mathcal{O})(M)
\]
of degree $p$. This bijection sends $\varphi$ to the degree $p$ coderivation given by 
\[
d_\varphi \coloneqq \mathscr{L}(\mathrm{id})(\diracComb(\pi_M,\varphi)) \cdot \omega(M)~,
\]
where $\omega(M)$ is the comonad structure map of $\mathscr{L}(u\mathcal{O})$.
\end{Proposition}

\begin{proof}
For any $\varphi$, the morphism $d_\varphi$ is an endomorphism of $\mathscr{L}(u\mathcal{O})(M)$ by Lemma \ref{lemma: coderivations on the cofree cooperad}. One has to check that it is a coderivation. The main point is to use the fact that $p_1$ is a monomorphism to be able to check it directly on the dual Schur functors: the same diagrams as in \cite[Proposition 6.23]{grignoulejay18} commute, thus $d_\varphi$ is a coderivation.
\end{proof}

We first present the formal constructions before giving a more explicit description of this adjunction.

\begin{Definition}[Complete Bar construction relative to $\kappa$]
Let $(\PP,\gamma_\PP, d_\PP)$ be a complete curved $\mathcal{S}\otimes c\mathcal{O}^\vee$-algebra. The \textit{complete Bar construction relative to} $\kappa$, denoted by $\widehat{\text{B}}_\kappa\PP$, is given by:
\[
\widehat{\mathrm{B}}_\kappa\PP \coloneqq \left(\mathscr{L}(u\mathcal{O})(\PP), d_{\mathrm{bar}} = d_1 + d_2\right)
\]
where $\mathscr{L}(u\mathcal{O})$ is the cofree dg $u\mathcal{O}$-coalgebra generated by the pdg $\mathbb{S}$-module $\PP$. It is endowed with the pre-differential $d_{\mathrm{bar}}$ given by the sum of two terms $d_1$ and $d_2$. The term $d_1$ is of the pre-differential is the unique coderivation extending the following map
\[
\begin{tikzcd}[column sep=5pc,row sep=3pc]
\phi_1: \mathscr{L}(u\mathcal{O})(\PP) \arrow[r,"\mathscr{L}(u\mathcal{O})(d_\PP)"]
&\mathscr{L}(u\mathcal{O})(\PP) \arrow[r,"\pi_\PP"]
&\PP~.
\end{tikzcd}
\]
It is given by 
\[
d_1 = \mathscr{L}(\mathrm{id})(\diracComb_{\mathbb{S}}(\mathrm{id},d_\mathcal{P}))~.
\]
The term $d_2$ is given by the unique coderivation that extends the map: 
\[
\begin{tikzcd}[column sep=4pc,row sep=3pc]
\phi_2: \mathscr{L}(u\mathcal{O})(\PP)  \arrow[r,"p_1"]
&\widehat{\mathscr{S}}_\mathbb{S}^c(u\mathcal{O})(\PP) \arrow[r,"\widehat{\mathscr{S}}_\mathbb{S}^c(\kappa)(\mathrm{id})"]
&\widehat{\mathscr{S}}_\mathbb{S}^c(\mathcal{S} \otimes c\mathcal{O}^\vee)(\PP) \arrow[r,"\gamma_\PP "]
&\PP~.
\end{tikzcd}
\]
\end{Definition}

\begin{Proposition}
Let $(\PP,\gamma_\PP, d_\PP)$ be a complete curved $\mathcal{S}\otimes c\mathcal{O}^\vee$-algebra. The complete Bar construction $\widehat{\mathrm{B}}_\kappa\PP$ forms a dg counital partial cooperad. Meaning that
\[
d_{\mathrm{bar}}^2=0~.
\]
\end{Proposition}

\begin{proof}
The proof is very similar to the proof of \cite[Proposition 9.2]{grignoulejay18}. In order to check that $d_{\mathrm{bar}}^2=0$, by a straightforward generalization of \cite[Proposition 6.25]{grignoulejay18} to the $\mathbb{S}$-colored case, it is enough to check that $(\phi_1 + \phi_2) \cdot d_{\mathrm{bar}} = 0$. We have that:
\[
d_{\mathrm{bar}} = \widehat{\mathscr{S}}_\mathbb{S}^c(\mathrm{id})(\diracComb(\mathrm{id},d_\PP)) + \widehat{\mathscr{S}}_\mathbb{S}^c(\mathrm{id}_{u\mathcal{O}}) \circ \widehat{\mathscr{S}}_\mathbb{S}^c(\diracComb(\pi_{\PP},\gamma_\PP \cdot \widehat{\mathscr{S}}_\mathbb{S}^c(\kappa)(\mathrm{id}_\PP))(\mathrm{id}_\PP) \cdot p_2(\PP)~,
\]
where the first term is $d_1$ and the second is $d_2$. One computes that: 
\[
(\phi_1 + \phi_2) \cdot d_{\mathrm{bar}} = \gamma_\PP \cdot \widehat{\mathscr{S}}_\mathbb{S}^c(\kappa \star \kappa - \Theta_{\mathcal{S} \otimes c\mathcal{O}^\vee} \cdot \eta_{u\mathcal{O}})(\mathrm{id}_\PP) = 0~. 
\]
since $\kappa$ is a curved twisting morphism. Therefore $d_{\mathrm{bar}}^2=0$ and $\widehat{\mathrm{B}}_\kappa \PP$ is a dg counital partial cooperad.
\end{proof}

Derivations on $\widehat{\mathscr{S}}_\mathbb{S}^c(\mathcal{S}\otimes c\mathcal{O}^\vee)(M)$, the free pdg $\mathcal{S}\otimes c\mathcal{O}^\vee$-algebra on an pdg $\mathbb{S}$-module $M$, are completely characterized by their restrictions to the generators $M$. See \cite[Proposition 6.15]{grignoulejay18} for an analogous statement, which holds \textit{mutatis mutandis} in our case. Let $M$ be an pdg $\mathbb{S}$-module, we denote $\iota_M$ be the following map:
\[
\begin{tikzcd}[column sep=5pc,row sep=3pc]
\iota_M: M \arrow[r,"\mathscr{S}_\mathbb{S}^c(\epsilon)(id_M)"]
&\mathscr{S}_\mathbb{S}^c(\mathcal{S}\otimes c\mathcal{O}^\vee)(M)~.
\end{tikzcd}
\]
\begin{Definition}[Complete Cobar construction relative to $\kappa$]
Let $(\C, \{\Delta_C\}, d_\C)$ be a dg $u\mathcal{O}$-coalgebra. The \textit{complete Cobar construction relative to} $\kappa$, denoted by $\widehat{\Omega}_\kappa \C$, is given by:
\[
\widehat{\Omega}\C \coloneqq \left(\widehat{\mathscr{S}}_\mathbb{S}^c(\mathcal{S}\otimes c\mathcal{O}^\vee)(\C), d_{\mathrm{cobar}} = d_1 - d_2\right)
\]

where $\widehat{\mathscr{S}}_\mathbb{S}^c(\mathcal{S}\otimes c\mathcal{O}^\vee)(\C)$ is the free pdg $\mathcal{S}\otimes c\mathcal{O}^\vee$-algebra. It is endowed with the pre-differential $d_{\mathrm{cobar}}$ which is the difference of two terms $d_1$ and $d_2$. The term $d_1$ is the unique derivation that extends the map:
\[
\begin{tikzcd}[column sep=5pc,row sep=3pc]
\psi_1: \C \arrow[r,"\iota_\C"]
&\mathscr{S}_\mathbb{S}^c(\mathcal{S}\otimes c\mathcal{O}^\vee)(\C) \arrow[r,"\mathscr{S}_\mathbb{S}^c(\mathrm{id})(d_\C)" ]
&\mathscr{S}_\mathbb{S}^c(\mathcal{S}\otimes c\mathcal{O}^\vee)(\C)~. 
\end{tikzcd}
\]
It is given by
\[
d_1 = \widehat{\mathscr{S}}_{\mathbb{S}}^c(\mathrm{id})(\diracComb(\mathrm{id},d_\C))~.
\]
The term $d_2$ is given by the unique derivation that extends the map: 
\[
\begin{tikzcd}[column sep=5pc,row sep=3pc]
\psi_2: \C \arrow[r,"\Delta_\C"]
&\mathscr{S}_\mathbb{S}^c(u\mathcal{O})(\C) \arrow[r,"\mathscr{S}_\mathbb{S}^c(\kappa)(\mathrm{id}_\C)" ]
&\mathscr{S}_\mathbb{S}^c(\mathcal{S}\otimes c\mathcal{O}^\vee)(\C)~. 
\end{tikzcd}
\]
\end{Definition}

\begin{Proposition}
Let $(\C, \{\Delta_C\}, d_\C)$ be a dg $u\mathcal{O}$-coalgebra. The complete Cobar construction relative to $\kappa$ $\widehat{\Omega}_\kappa\C$ forms a complete curved $\mathcal{S}\otimes c\mathcal{O}^\vee$-algebra. In other words, the following diagram commutes: 
\[
\begin{tikzcd}[column sep=9pc,row sep=3pc]
\widehat{\mathscr{S}}_\mathbb{S}^c(\I_\mathbb{S}) \circ \widehat{\mathscr{S}}_\mathbb{S}^c(\mathcal{S}\otimes c\mathcal{O}^\vee)(\C) \arrow[r,"\widehat{\mathscr{S}}_\mathbb{S}^c(\Theta_{\mathcal{S}\otimes c\mathcal{O}^\vee}) \circ \widehat{\mathscr{S}}_\mathbb{S}^c(\mathrm{id})(\mathrm{id}) "] \arrow[rd,"- d_{\mathrm{cobar}}^2",swap]
&\widehat{\mathscr{S}}_\mathbb{S}^c(\mathcal{S}\otimes c\mathcal{O}^\vee) \circ \widehat{\mathscr{S}}_\mathbb{S}^c(\mathcal{S}\otimes c\mathcal{O}^\vee)(\C)\arrow[d,"\widehat{\mathscr{S}}_\mathbb{S}^c(\Delta_{\mathcal{S}\otimes c\mathcal{O}^\vee}) \cdot \varphi_{\mathcal{S}\otimes c\mathcal{O}^\vee}"]\\
&\widehat{\mathscr{S}}_\mathbb{S}^c(\mathcal{S}\otimes c\mathcal{O}^\vee)(\C)~.
\end{tikzcd}
\]
\end{Proposition}

\begin{proof}
The proof is again very similar to the proof of \cite[Proposition 9.1]{grignoulejay18}. In order to check that 
\[
- d_{\mathrm{cobar}}^2 = \widehat{\mathscr{S}}_\mathbb{S}^c(\Delta_{\mathcal{S}\otimes c\mathcal{O}^\vee}) \cdot \varphi_{\mathcal{S}\otimes c\mathcal{O}^\vee} \cdot \widehat{\mathscr{S}}_\mathbb{S}^c(\Theta_{\mathcal{S}\otimes c\mathcal{O}^\vee}) \circ \widehat{\mathscr{S}}_\mathbb{S}^c(\mathrm{id})(\mathrm{id}) ~,
\]
it is in fact enough to check that 
\[
- d_{\mathrm{cobar}} \cdot (\psi_1 - \psi_2) = \widehat{\mathscr{S}}_\mathbb{S}^c(\Theta_{\mathcal{S}\otimes c\mathcal{O}^\vee})(\mathrm{id}_\C)~.
\]
This comes from a straightforward generalization of \cite[Proposition 7.4]{grignoulejay18}. The pre-differential $d_{\mathrm{cobar}}$ is given by
\[
d_{\mathrm{cobar}} = \widehat{\mathscr{S}}_\mathbb{S}^c(\mathrm{id})(\diracComb(\iota_{\C},d_\C)) - \widehat{\mathscr{S}}_\mathbb{S}^c(\Delta_{\mathcal{S}\otimes c\mathcal{O}^\vee})(\mathrm{id}_\C) \cdot \widehat{\mathscr{S}}_\mathbb{S}^c(\diracComb(\iota_{\C},\widehat{\mathscr{S}}_\mathbb{S}^c(\kappa)(\mathrm{id}_\C) \cdot \Delta_\C))(\mathrm{id}_\C)~,
\]
where the first term is $d_1$ and the second is $d_2$. One computes that:
\[
- d_{\mathrm{cobar}} \cdot (\psi_1 - \psi_2) = \widehat{\mathscr{S}}_\mathbb{S}^c(\kappa \star \kappa)(\mathrm{id}_\C)~.
\]
Since $\kappa$ is a curved twisting morphism, this concludes the proof. 
\end{proof}

Let us give explicit descriptions of these functors.

\begin{Definition}[Complete Cobar construction]
Let $(\mathcal{C},\{\Delta_i\},\epsilon_\C,d_\C)$ be a dg counital partial cooperad. The \textit{complete Cobar construction} $\widehat{\Omega}\C$ of $\mathcal{C}$ is given by:
\[
\widehat{\Omega}\C \coloneqq \left( \overline{\mathscr{T}}^\wedge(s^{-1}\C \oplus \nu), d_{\mathrm{cobar}} = d_1 - d_2, \Theta_{\mathrm{cobar}} \right)~,
\]
where $\overline{\mathscr{T}}^\wedge(s^{-1}\C \oplus \nu)$ is the completed reduced tree monad applied to the dg $\mathbb{S}$-module $s^{-1}\C \oplus \nu$. Here $\nu$ is an arity $1$ degree $-2$ generator. It is endowed with the differential $d_{\mathrm{cobar}}$ given by the difference of $d_1$ and $d_2$. The term $d_1$ is the unique derivation extending
\[
\begin{tikzcd}[column sep=5pc,row sep=3pc]
s^{-1}\C \arrow[r,"s^{-1}d_\C"]
&s^{-1}\C \arrow[r,rightarrowtail]
&\overline{\mathscr{T}}^\wedge(s^{-1}\C \oplus \nu)~.
\end{tikzcd}
\]
The term $d_2$ comes from the \textit{structure} of a dg counital partial cooperad on $\C$, it is given by the unique derivation extending
\[
\begin{tikzcd}[column sep=5pc,row sep=3pc]
s^{-1}\C \arrow[r,"s^{-2}\Delta_{(1)} + s^{-1}\epsilon_\C"]
&(s^{-1}\C \circ_{(1)} s^{-1}\C) \oplus \I.\nu \arrow[r,rightarrowtail]
&\overline{\mathscr{T}}^\wedge(s^{-1}\C \oplus \I.\nu)~.
\end{tikzcd}
\]
Its curvature $\Theta_{\mathrm{cobar}}$ is given by the following map
\[
\begin{tikzcd}[column sep=4pc,row sep=3pc]
\Theta_{\mathrm{cobar}}: \I \arrow[r,"s^{-2}"]
&\I.\nu \arrow[r,hookrightarrow] 
&\overline{\mathscr{T}}^\wedge(s^{-1}\C \oplus \I.\nu)~.
\end{tikzcd}
\]
The resulting complete Cobar construction of $\C$ forms a complete curved absolute partial operad.
\end{Definition}

\begin{Notation}
We denote by $\mathscr{T}^\vee(-)$ the cofree counital partial cooperad endofunctor in the category of pdg $\mathbb{S}$-modules. It is given by the cofree graded $u\mathcal{O}$-coalgebra $\mathscr{L}(u\mathcal{O})(-)$.
\end{Notation}

\begin{Definition}[Complete Bar construction]
Let $(\PP, \gamma_\PP, d_\PP,\Theta_\PP)$ a complete curved absolute partial operad. The \textit{complete Bar construction} $\widehat{\mathrm{B}}\PP$ of $\PP$ is given by:
\[
\widehat{\mathrm{B}\PP} \coloneqq \left(\mathscr{T}^\vee(s\PP), d_{\mathrm{bar}} = d_1 + d_2 \right)~,
\]
where $\mathscr{T}^\vee(s\PP)$ is the cofree counital partial cooperad generated by the pdg $\mathbb{S}$-module $s\PP$. It is endowed the pre-differential $d_{\mathrm{bar}}$, given by the sum of $d_1$ and $d_2$. The term $d_1$ is the unique coderivation extending
\[
\begin{tikzcd}[column sep=4pc,row sep=3pc]
\mathscr{T}^\vee(s\PP)  \arrow[r, twoheadrightarrow]
&s\PP \arrow[r,"sd_{\PP}"]
&s\PP~.
\end{tikzcd} 
\]
The term $d_2$ comes from the \textit{structure} of $\PP$, it is given by the unique coderivation extending 
\[
\begin{tikzcd}[column sep=5pc,row sep=3pc]
\mathscr{T}^\vee(s\PP) \arrow[r, twoheadrightarrow]
&\I \oplus (s\PP \circ_{(1)} s\PP) \arrow[r,"s\Theta_\PP + s^2\gamma_{(1)}"]
&s\PP~.
\end{tikzcd}
\]
The resulting complete Bar construction of $\PP$ forms a dg counital partial cooperad. 
\end{Definition}

\begin{Remark}
The complete Bar construction of a complete curved absolute partial operad $\PP$ is, in general, not coaugmented. Indeed, one has that
\[
d_{\mathrm{bar}}(|) = d_2(|) = \Theta_\PP(\mathrm{id})~,
\]
where $|$ denotes the coaugmented counit of $\mathscr{T}^\vee(s\PP)$.
\end{Remark}

\begin{Proposition}
The following holds:
\begin{enumerate}
\item Let $(\mathcal{C},\{\Delta_i\},\epsilon_\C,d_\C)$ be a dg counital partial cooperad. There is a natural isomorphism
\[
\widehat{\Omega}_\kappa \C \cong \widehat{\Omega}\C~.
\]
\item Let $(\PP, \gamma_\PP, \Theta_\PP, d_\PP)$ a complete curved absolute partial operad. There is a natural isomorphism 
\[
\widehat{\mathrm{B}}_\kappa \PP \cong \widehat{\mathrm{B}} \PP~.
\]
\end{enumerate}
\end{Proposition}

\begin{proof}
This can be shown by direct inspection, extending the isomorphism of Lemma \ref{lemma: dual Schur of cOvee}.
\end{proof}

\begin{lemma}[Curved convolution operad]
Let $(\PP,\gamma_\PP,d_\PP,\Theta_\PP)$ be a complete curved absolute partial operad and let $(\C,\{\Delta_i\},\epsilon_\C,d_\C)$ be a dg counital partial cooperad. The convolution pdg partial operad $(\mathcal{H}om(\C,\PP),\{\circ_i\},\partial)$ forms a curved partial operad endowed with the curvature given by
\[
\begin{tikzcd}
\Theta_{\mathcal{H}om}(\mathrm{id}): \mathcal{C} \arrow[r,"\epsilon_\C"]
&\I \arrow[r,"\Theta_\PP"]
&\mathcal{P}~.
\end{tikzcd}
\]
\end{lemma}

\begin{proof}
The proof is completely analogous to Lemma \ref{lemma curved S colored convolution}.
\end{proof}

\begin{Remark}
We only used the operations $\{\circ_i\}$ on $\PP$ in order to define this convolution operad. These operations are obtained by restricting the structural morphism $\gamma_\PP$ to finite sums inside the completed reduced tree monad. See Appendix B \ref{Appendix B} for more details. In fact, using the morphism $\gamma_\PP$, one can show that this convolution curved partial operad is in fact an curved \textit{absolute} partial operad. The idea is that infinite sums of convolution operations have a well-defined image in $\PP$ because $\PP$ is an absolute operad. 

\medskip

Moreover, the convolution operad between a conilpotent cooperad $\C$ and an operad is \textit{also} an \textit{absolute operad}. Indeed, the decompositions in $\C$ always produce finite sums, therefore infinite sums of convolution operations are still well-defined. For a more detailed discussion about how \textit{convolution structures} always produce \textit{absolute} types of algebraic structures, see \cite[Section 4.1]{integration}.
\end{Remark}

Once we have the curved convolution partial operad, we can define the notion of curved twisting morphism between complete curved absolute partial operads and dg counital partial cooperads.

\begin{Definition}[Curved twisting morphism]
Let $(\PP,\gamma_\PP,d_\PP,\Theta_\PP)$ be a complete curved absolute partial operad and let $(\C,\{\Delta_i\},\epsilon,d_\C)$ be a dg counital partial cooperad. A \textit{curved twisting morphism} $\alpha$ is a Maurer-Cartan element curved pre-Lie algebra given by the curved convolution operad:
\[
\mathfrak{g}_{\C,\PP} \coloneqq \prod_{n\geq 0}\mathrm{Hom}_{\mathbb{S}}(\C(n),\PP(n))~.
\]
This is the data of a morphism of graded $\mathbb{S}$-modules $\alpha: \C \longrightarrow \PP$ of a degree $-1$ such that:
\[
\partial(\alpha) + \alpha \star \alpha = \Theta_{\mathcal{H}om}(\mathrm{id})~.
\]
The set of twisting morphism between $\C$ and $\PP$ will be denoted $\mathrm{Tw}(\C,\PP)$. 
\end{Definition}

The set of curved twisting morphism between between complete curved partial operads and counital partial dg cooperads
\[
\begin{tikzcd}
\mathrm{Tw}(- , -) : (\mathsf{dg}~\mathsf{upCoop})^{\text{op}} \times \mathsf{curv}~\mathsf{abs}~\mathsf{pOp}^{\mathsf{comp}}  \arrow[r]
&\mathsf{Set}~,
\end{tikzcd}
\]
which is represented on both sides by the complete Bar-Cobar constructions.

\begin{Proposition}[Complete Bar-Cobar adjunction]
Let $(\PP,\gamma_\PP,\Theta_\PP,d_\PP)$ be a complete curved partial operad and let $(\C,\{\Delta_i\},\epsilon,d_\C)$ be a dg counital partial  cooperad. There are isomorphisms:
\[
\mathrm{Hom}_{\mathsf{curv}~\mathsf{abs}~\mathsf{pOp}^{\mathsf{comp}}}(\widehat{\Omega}\C,\PP) \cong \mathrm{Tw}(\C,\PP) \cong \mathrm{Hom}_{\mathsf{dg}~\mathsf{upCoop}}(\C,\widehat{\mathrm{B}}\PP)~,
\]
which are natural in $\C$ and $\PP$. 
\end{Proposition}

\begin{proof}
This can be shown by a straightforward computation.
\end{proof}

\begin{Corollary}\label{Complete Bar-Cobar adjunction}
There is a pair of adjoint functors:
\[
\begin{tikzcd}[column sep=7pc,row sep=3pc]
            \mathsf{dg}~\mathsf{upCoop} \arrow[r, shift left=1.1ex, "\widehat{\Omega}"{name=F}] &\mathsf{curv}~\mathsf{abs}~\mathsf{pOp}^{\mathsf{comp}}~. \arrow[l, shift left=.75ex, "\widehat{\mathrm{B}}"{name=U}]
            \arrow[phantom, from=F, to=U, , "\dashv" rotate=-90]
\end{tikzcd}
\]
\end{Corollary}

\section{Counital partial cooperads up to homotopy and transfer of model structures}
In this section, we introduce the notion of a counital partial cooperads up to homotopy. This notion is dual to that of unital partial operads up to homotopy developed in \cite{grignou2021}. In fact, there is a $\mathbb{S}$-colored operad which encodes unital partial operads up to homotopy as its algebras and counital partial cooperads up to homotopy as its coalgebras. The advantage of counital partial cooperads up to homotopy with respect to counital partial cooperads is that the admit a canonical cylinder object. Thus, using a left transfer theorem, we endow them with a model structure where weak equivalences are arity-wise quasi-isomorphisms. Notice that here, the category of counital partial cooperads up to homotopy with \textit{strict morphisms} is the one being considered. Afterwards, we transfer this model structure to the category of complete curved absolute partial operads via a complete Bar-Cobar adjunction.

\begin{Notation}
We denote by $\mathrm{RT}_\omega^n$ the set of rooted trees of arity $n$ with $\omega$ internal edges. There is no restriction on the number of incoming edges of the vertices considered.
\end{Notation}

\begin{Definition}[Counital partial cooperad up to homotopy]
Let $(\C,d_\C)$ be a dg $\mathbb{S}$-module. A \textit{counital partial cooperad up to homotopy structure} on $\C$ amounts to the data of a derivation of degree $-1$
\[
d: \overline{\mathscr{T}}^\wedge(s^{-1}\C \oplus \nu) \longrightarrow \overline{\mathscr{T}}^\wedge(s^{-1}\C \oplus \nu)
\]
such that the restriction of $d$ to $s^{-1}\C$ is given by $s^{-1}d_\C$ and such that for any series of rooted trees of arity $n$ labeled by elements of $s^{-1}\C$:
\[
d^2\left(\sum_{\omega \geq 1} \sum_{\tau \in \mathrm{RT}_\omega^n} \tau \right) = \sum_{\omega \geq 1} \sum_{\tau \in \mathrm{RT}_\omega^n} \left( v \circ_1 \tau - \sum_{i=0}^n \tau \circ_i v \right)~.
\]
In order words, the data of $\left(\overline{\mathscr{T}}^\wedge(s^{-1}\C \oplus \nu),d \right)$ forms a complete curved absolute partial operad.
\end{Definition}

\begin{Remark}
The morphism 
\[
\begin{tikzcd}[column sep=3pc,row sep=3pc]
\C  \arrow[r,"\cong"]
&s^{-1}\C \arrow[r,"d"]
&\overline{\mathscr{T}}^\wedge(s^{-1}\C \oplus \nu) \arrow[r,twoheadrightarrow]
&\I.\nu
\end{tikzcd} 
\]
endows $\C$ with a counit $\epsilon_\C: \C \longrightarrow \I$ which satisfies the counital axiom of counital partial cooperad only up to higher homotopies. The morphism 
\[
\begin{tikzcd}[column sep=3pc,row sep=3pc]
\C  \arrow[r,"\cong"]
&s^{-1}\C \arrow[r,"d"]
&\overline{\mathscr{T}}^\wedge(s^{-1}\C \oplus \nu) \arrow[r,twoheadrightarrow]
&\overline{\mathscr{T}}^{\wedge~(2)}(s^{-1}\C \oplus \nu)
\end{tikzcd} 
\]
endows $\C$ with a family of partial decompositions maps $\{\Delta_i\}$ which satisfy the coparallel and cosequential axioms of a counital partial cooperad only up to higher homotopies. All the higher homotopies are given by the morphism
\[
\begin{tikzcd}[column sep=3pc,row sep=3pc]
\C  \arrow[r,"\cong"]
&s^{-1}\C \arrow[r,"d"]
&\overline{\mathscr{T}}^\wedge(s^{-1}\C \oplus \nu) \arrow[r,twoheadrightarrow]
&\overline{\mathscr{T}}^{\wedge~(\geq 3)}(s^{-1}\C \oplus \nu)~.
\end{tikzcd} 
\]
\end{Remark}

\begin{Example}
Any dg counital partial cooperad $(\C,\{\Delta_i\},\epsilon,d_\C)$ is an example of counital partial cooperad up to homotopy via its complete Cobar construction $\widehat{\Omega}\C$. 
\end{Example}

\begin{Definition}[Strict morphisms]
Let $(\C,d_1,d_\C)$ and $(\D,d_2,d_\D)$ be two counital partial cooperads up to homotopy. A \textit{morphism} $f: \C \longrightarrow \D$ amounts to the data of a morphism of graded $\mathbb{S}$-modules $f: \C \longrightarrow \D$ such that the following diagram commute
\[
\begin{tikzcd}[column sep=3pc,row sep=2pc]
\C  \arrow[r,"\cong"] \arrow[d,"f",swap]
&s^{-1}\C \arrow[r,"d_1 "]
&\overline{\mathscr{T}}^\wedge(s^{-1}\C \oplus \nu) \arrow[d,"\overline{\mathscr{T}}^\wedge(s^{-1}f \oplus \nu)"]\\
\D  \arrow[r,"\cong"]
&s^{-1}\D \arrow[r,"d_2 "]
&\overline{\mathscr{T}}^\wedge(s^{-1}\D \oplus \nu)~. 
\end{tikzcd} 
\]
\end{Definition}

\begin{Remark}[$\infty$-morphisms]
Let $(\C,d_1,d_\C)$ and $(\D,d_2,d_\D)$ be two counital partial cooperads up to homotopy. An $\infty$\textit{-morphism} $f: \C \rightsquigarrow \D$ amounts to the data of a morphism of complete curved absolute partial operads 
\[
f: \left(\overline{\mathscr{T}}^\wedge(s^{-1}\C \oplus \nu),d_1 \right) \longrightarrow \left(\overline{\mathscr{T}}^\wedge(s^{-1}\D \oplus \nu),d_2 \right)~.
\]
These $\infty$-morphisms of counital partial cooperads up to homotopy should have all the expected properties of $\infty$-morphism between unital partial operads described in \cite{grignou2021}. In particular, $\infty$-quasi-isomorphisms of counital partial cooperads admit inverses in homology, and their set describes the hom-sets of the homotopy category. For the sake of brevity, we do not explore this path here.
\end{Remark}

Counital partial cooperads up to homotopy are coalgebras over a dg $\mathbb{S}$-colored operad, which is given by the Cobar construction on the conilpotent curved $\mathbb{S}$-colored partial cooperad $\mathcal{S} \otimes c\mathcal{O}^\vee$. This Cobar construction is the $\mathbb{S}$-colored analogous of the Cobar construction of Definition \ref{Def: Cobar de Brice}. (We swear, dear reader, this is the last Cobar construction of the article !)

\begin{Definition}[The dg $\mathbb{S}$-colored operad encoding (co)unital partial (co)operads up to homotopy]
The dg $\mathbb{S}$-colored operad $\Omega_\mathbb{S}(\mathcal{S} \otimes c\mathcal{O}^\vee)$ is given by
\[
\Omega_\mathbb{S}\left(\mathcal{S} \otimes c\mathcal{O}^\vee \right) \coloneqq \left(\mathscr{T}_\mathbb{S}\left(s^{-1}(\mathcal{S} \otimes c\mathcal{O}^\vee)\right), d_{\mathrm{cobar}} = -d_2 \right)~,
\]
where $\mathscr{T}_\mathbb{S}(s^{-1}(\mathcal{S} \otimes c\mathcal{O}^\vee))$ is the free unital partial $\mathbb{S}$-colored operad generated by the desuspension of the pdg $\mathbb{S}$-colored scheme $\mathcal{S} \otimes c\mathcal{O}^\vee$. It is endowed with a differential $d_2$, which is given by the unique derivation extending the map:
\[
\begin{tikzcd}[column sep=6.5pc,row sep=3pc]
s^{-1}(\mathcal{S} \otimes c\mathcal{O}^\vee) \arrow[r, "s^{-2}\Delta_{(1)} \oplus s^{-1}\Theta_{\mathcal{S} \otimes c\mathcal{O}^\vee}"]
&\left(s^{-1}(\mathcal{S} \otimes c\mathcal{O}^\vee) \circ_{(1)} s^{-1}(\mathcal{S} \otimes c\mathcal{O}^\vee)\right) \oplus \I_\mathbb{S} \hookrightarrow \mathscr{T}_\mathbb{S}\left(s^{-1}(\mathcal{S} \otimes c\mathcal{O}^\vee)\right)~.
\end{tikzcd} 
\]
\end{Definition}

\begin{Proposition}
The following holds:
\begin{enumerate}
\item The category of unital partial operads up to homotopy with strict morphisms is equivalent to the category of dg algebras over $\Omega_\mathbb{S}\left(\mathcal{S} \otimes c\mathcal{O}^\vee \right)$.

\item The category of counital partial cooperads up to homotopy with strict morphism is equivalent to the category of dg coalgebras over $\Omega_\mathbb{S}\left(\mathcal{S} \otimes c\mathcal{O}^\vee \right)$.
\end{enumerate}
\end{Proposition}

\begin{proof}
For the first statement, the proof is \textit{mutatis mutandis} the same as the proof of \cite[Proposition 36]{grignou2021}. For the second statement, the proof is \textit{mutatis mutandis} the same as the proof of \cite[Theorem 12.1]{grignoulejay18}.
\end{proof}

\begin{Corollary}
There forgetful functor from counital partial cooperads up to homotopy to the category of dg $\mathbb{S}$-modules admits a right adjoint
\[
\begin{tikzcd}[column sep=7pc,row sep=3pc]
             \mathsf{upCoop}_\infty \arrow[r, shift left=1.3ex, "\mathrm{U}"{name=F}] &\mathsf{dg}~\mathbb{S}\text{-}\mathsf{mod}~,  \arrow[l, shift left=1.2ex, "\mathscr{L}\left(\Omega_\mathbb{S}\left(\mathcal{S} \otimes c\mathcal{O}^\vee \right)\right)"{name=U}]
            \arrow[phantom, from=F, to=U, , "\dashv" rotate=-90]
\end{tikzcd}
\]
which is the cofree counital partial cooperad up to homotopy functor.
\end{Corollary}

\begin{proof}
This is an immediate consequence of Theorem \ref{thm: existence de la cogèbre colibre}.
\end{proof}

Using this adjunction, we want to transfer the standard model structure on dg $\mathbb{S}$-modules via a \textit{left-transfer theorem}. The key ingredient to apply the left-transfer theorem in this situation is the construction of a \textit{functorial cylinder object} in the category of counital partial cooperads up to homotopy.

\medskip

Let $\mathrm{I}$ be the commutative Hopf monoid given by the cellular chains on the interval $[0,1]$, see \cite[Example 3.3.3]{bergermoerdieck03}. It is a commutative algebra with a compatible coassociative coalgebra structure is given by the choice of a cellular diagonal on the topological interval. It is an interval object in the category of dg modules, meaning that there is a factorization
\[
\kk \oplus \kk \rightarrowtail \mathrm{I} \qi \kk
\]
of the codiagonal map $\mathrm{id} \oplus \mathrm{id}: \kk \oplus \kk \longrightarrow \kk$.

\begin{Definition}[Interval dg $\mathbb{S}$-modules]
The \textit{interval} dg $\mathbb{S}$-module $\mathrm{Int}$ is the dg $\mathbb{S}$-module given by 
\[
\mathrm{Int}(n) \coloneqq \mathrm{I}
\]
for all $n \geq 0$. 
\end{Definition}

\begin{Remark}
If we endow $\mathrm{Int}$ with the partial decomposition maps $\Delta_i: \mathrm{Int}(n+k-1) \longrightarrow \mathrm{Int}(n) \otimes \mathrm{Int}(k)$ given by the coalgebra structure of $\mathrm{I}$
\[
\Delta: \mathrm{I} \longrightarrow \mathrm{I} \otimes \mathrm{I}~,
\]
then $\mathrm{Int}$ \textit{almost} forms a dg counital partial cooperad. It satisfies the cosequential axiom but, since $\Delta$ is not \textit{cocommutative}, it does not satisfy the coparallel axiom. If there existed a \textit{cocommutative} interval object in the category of dg module, we could directly transfer the standard model structure of dg $\mathbb{S}$-module to the category of counital partial cooperads, in a dual version of \cite[Theorem 3.2]{bergermoerdieck03}. This is one reason to work with counital partial cooperads up to homotopy.
\end{Remark}

\begin{Proposition}
Let $(\C,d_1,d_\C)$ be a counital partial cooperad up to homotopy. The Hadamard product $\C \otimes \mathrm{Int}$ has a canonical structure of counital partial cooperad up to homotopy.
\end{Proposition}

\begin{proof}
The proof is completely dual to the proof of \cite[Proposition 31]{grignou2021}, where given a unital partial operad up to homotopy, the author builds a unital partial operad up to homotopy structure on $\PP \otimes \mathrm{Int}$. 

\medskip

Let us sketch this construction: the idea to build the map that induces the coderivation
\[
\Delta: s^{-1}(\C \otimes \mathrm{Int}) \longrightarrow \overline{\mathscr{T}}^\wedge(s^{-1}(\C \otimes \mathrm{Int}) \oplus \nu)
\]
is to do so by induction on the weight of rooted trees in $\overline{\mathscr{T}}^\wedge$. Suppose there is a map
\[
\Delta_{m-1}: s^{-1}(\C \otimes \mathrm{Int}) \longrightarrow \overline{\mathscr{T}}^{\leq (m-1)}(s^{-1}(\C \otimes \mathrm{Int}) \oplus \nu)
\]
such that 
\[
d^2_{m-1}\left(\sum_{\omega \geq 1}^{m-1} \sum_{\tau \in \mathrm{RT}_\omega^n} \tau \right) = \sum_{\omega \geq 1}^{m-1} \sum_{\tau \in \mathrm{RT}_\omega^n} \left( v \circ_1 \tau - \sum_{i=0}^n \tau \circ_i v \right)~,
\]
where $d_{m-1}$ is the induced coderivation of $\overline{\mathscr{T}}^{\leq (m-1)}$ by $\Delta_{m-1}$. Furthermore, we suppose that the following equalities hold 

\[
\Delta_\C~(s\pi \otimes \mathrm{Id}_\C) = (s\pi \otimes \mathrm{Id}_\C \oplus \mathrm{Id}_\nu)^{\otimes (m-1)}~\Delta_{m-1}~,
\]

\[
(s\iota_i \otimes \mathrm{Id}_\C \oplus \mathrm{Id}_\nu)^{\otimes (m-1)}~\Delta_\C = \Delta_{m-1}~(s\iota_i  \otimes \mathrm{Id}_\C)~,
\]
\vspace{0.1pc}

where $p$ is the projection $\C \otimes \mathrm{Int} \longrightarrow \C$ and $\iota_i: \C \longrightarrow \C \otimes \mathrm{Int}$, for $i=1,2$ are the two inclusions. Notice that these two conditions imply the functoriality of this construction. Extending $\Delta_{m-1}$ to $\Delta_m$ can be shown to be equivalent to finding a lift in a commutative square 
in the underlying category of dg modules. This lift because in the square, the left hand side is a cofibration and the right hand side is an acyclic fibration.
\end{proof}

\begin{Corollary}[Functorial cylinder object]\label{lemma: functorial cylinder}
Let $(\C,d_1,d_\C)$ be a counital partial cooperad up to homotopy. Then the codiagonal morphism $\mathrm{id} \oplus \mathrm{id}: \mathcal{C} \oplus \mathcal{C} \longrightarrow \mathcal{C}$ factors as follows
\[
\begin{tikzcd}[column sep=4pc,row sep=0pc]
\mathcal{C} \oplus \mathcal{C} \arrow[r,rightarrowtail]
&\mathcal{C} \otimes \mathrm{Int} \arrow[r,"\sim"]
&\mathcal{C} ~,
\end{tikzcd}
\]
where the first arrow is a degree-wise monomorphism and the second arrow is an arity-wise quasi-isomorphism. Thus $\C \otimes \mathrm{Int}$ is a cylinder object which is functorial in $\mathcal{C}$. 
\end{Corollary}

\begin{theorem}[Transferred model structure]\label{thm: model structure on counital cooperads}
There is a cofibrantly generated model structure on the category of counital partial cooperads up to homotopy defined by the following classes of morphisms:
\begin{enumerate}
\item The class of weak-equivalences $\mathrm{W}$ is given by strict morphisms of counital partial cooperads up to homotopy $f: \mathcal{C} \longrightarrow \mathcal{D}$ such that $f(n): \mathcal{C}(n) \longrightarrow \mathcal{D}(n)$ is a quasi-isomorphism for all $n \geq 0$. 
\vspace{0.5pc}
\item The class of cofibrations is $\mathrm{Cof}$ is given by strict morphisms of counital partial cooperads up to homotopy $f: \mathcal{C} \longrightarrow \mathcal{D}$ such that $f(n): \mathcal{C}(n) \rightarrowtail \mathcal{D}(n)$ is a degree-wise monomorphism for all $n \geq 0$. 
\vspace{0.5pc}
\item The class of fibrations $\mathrm{Fib}$ is given by strict morphisms of counital partial cooperads up to homotopy which have the right lifting property against all morphism in $\mathrm{W} \cap \mathrm{Cof}$. 
\end{enumerate}
\end{theorem}

\begin{proof}
In order to apply the left-transfer theorem such as \cite{lefttransfer}, essentially need a functorial cofibrant replacement functor and a functorial cylinder object. Since every object is cofibrant in this model structure, we only need a functorial cylinder object. This follows from Corollary \ref{lemma: functorial cylinder}.
\end{proof}

Now our aim is to transfer this structure to category of complete curved absolute partial operads, replicating the methods of Section 10 in \cite{grignoulejay18}. For this purpose, we induce a complete Bar-Cobar adjunction using a curved twisting morphism between the $\mathbb{S}$-colored objects encoding these categories. (Yes, dear reader, we lied). For that, we consider the universal curved twisting morphism $\iota$:

\begin{lemma}
The $\mathbb{S}$-color scheme map given by the inclusion 
\[
\iota: \mathcal{S} \otimes c\mathcal{O}^\vee \hookrightarrow \Omega_\mathbb{S}\left(\mathcal{S} \otimes c\mathcal{O}^\vee \right)
\]
is a curved twisting morphism. 
\end{lemma}

\begin{proof}
This is immediate to check from the definition. 
\end{proof}

\begin{Proposition}[Complete Bar-Cobar adjunction relative to $\iota$]\label{Prop: complete Bar-Cobar adjunction relative to iota}
There is a pair of adjoint functors 
\[
\begin{tikzcd}[column sep=7pc,row sep=3pc]
            \mathsf{upCoop}_\infty \arrow[r, shift left=1.1ex, "\widehat{\Omega}_\iota"{name=F}] &\mathsf{curv}~\mathsf{abs}~\mathsf{pOp}^{\mathsf{comp}}~. \arrow[l, shift left=.75ex, "\widehat{\mathrm{B}}_\iota"{name=U}]
            \arrow[phantom, from=F, to=U, , "\dashv" rotate=-90]
\end{tikzcd}
\]
\end{Proposition}

\begin{proof}
The constructions of the complete Bar and the complete Cobar functors are \textit{mutatis mutandis} the same as those described in the previous section, where we constructed the complete Bar-Cobar relative to $\kappa$. The proof that these functors form an adjunction is also \textit{mutatis mutandis} the same.
\end{proof}

This allows us in turn to transfer the above model structure from counital partial cooperads up to homotopy to complete curve absolute partial operads using a \textit{right transfer theorem}.

\begin{theorem}\label{thm: model structure on curved operads}
There is a cofibrantly generated model structure on the category of complete curved absolute partial operads defined by the following classes of morphisms:
\begin{enumerate}
\item The class of weak-equivalences $\mathrm{W}$ is given by morphisms of complete curved absolute partial operads $f: \mathcal{P} \longrightarrow \mathcal{Q}$ such that $\widehat{\mathrm{B}}_\iota(f)(n): \widehat{\mathrm{B}}_\iota(\mathcal{P})(n) \longrightarrow \widehat{\mathrm{B}}_\iota(\mathcal{Q})(n)$ is a quasi-isomorphism for all $n \geq 0$. 
\vspace{0.5pc}
\item The class of fibrations is $\mathrm{Fib}$ is given by morphisms of complete curved absolute partial operads $f: \mathcal{P} \longrightarrow \mathcal{Q}$ such that $\widehat{\mathrm{B}}_\iota (f): \widehat{\mathrm{B}}_\iota(\mathcal{P}) \longrightarrow \widehat{\mathrm{B}}_\iota(\mathcal{Q})$ is a fibration of counital partial cooperads up to homotopy. 
\vspace{0.5pc}
\item The class of cofibrations $\mathrm{Cof}$ is given by morphisms of complete curved absolute partial operads which have the left lifting property against all morphism in $\mathrm{W} \cap \mathrm{Fib}$. 
\end{enumerate}
\end{theorem}

\begin{proof}
In order to apply the right transfer theorem, we need to check that the acyclicity conditions of \cite[Section 2.5]{bergermoerdieck03} are satisfied. Thus these conditions need to be checked by direct computation, as it is done in \cite[Section 10.5]{grignoulejay18}. The key point is to show that acyclic cofibrations of complete curved partial cooperads are precisely given by graded quasi-isomorphisms of Definition \ref{def: graded qi}. Then the result follows, as it is clear that the complete Cobar functor preserve small objects, and since these trivial cofibrations are stable under pushouts and sequential colimits. In is a straightforward but tedious exercise to check that the arguments used in \textit{loc.cit} generalize to our framework.
\end{proof}

\begin{Proposition}
Let $f: (\mathcal{P},\gamma_\PP,d_\PP,\Theta_\PP) \longrightarrow (\mathcal{Q},\gamma_\Q,d_\Q,\Theta_\Q)$ be a morphism of complete curved absolute partial operads. If $f(n): \mathcal{P}(n) \longrightarrow \mathcal{Q}(n)$ is a degree-wise epimorphism, then it is a fibration. Thus all complete curved absolute partial operads are fibrant. 
\end{Proposition}

\begin{proof}
The arguments of Section 10.2 in \cite{grignoulejay18} generalize to the operadic setting \textit{mutatis mutandis}. 
\end{proof}

\begin{Definition}[Graded quasi-isomorphism]\label{def: graded qi}
Let $f: (\mathcal{P},\gamma_\PP,d_\PP,\Theta_\PP) \longrightarrow (\mathcal{Q},\gamma_\Q,d_\Q,\Theta_\Q)$ be a morphism of complete curved absolute partial operads. It is a \textit{graded quasi-isomorphism} if the induced morphism of dg $\mathbb{S}$-modules
\[
\mathrm{gr}(f): \mathrm{gr}(\mathcal{P}) \cong \bigoplus_{\omega \geq 1} \overline{\mathscr{F}}_{\omega} \PP/ \overline{\mathscr{F}}_{\omega +1} \PP \longrightarrow  \mathrm{gr}(\mathcal{Q}) \cong \bigoplus_{\omega \geq 1} \overline{\mathscr{F}}_{\omega} \mathcal{Q}/ \overline{\mathscr{F}}_{\omega +1} \Q
\]
is an arity-wise quasi-isomorphism. Here $\overline{\mathscr{F}}_{\omega} \PP$ denotes the $\omega$-term of the canonical filtration of an absolute partial operad defined in Appendix \ref{Appendix B}.
\end{Definition}

\begin{Remark}
One checks easily that since $d_\mathcal{P}^2$ raises the weight of an operation in $\mathcal{P}$ by one, it is equal to zero in the associated graded. Therefore $\mathrm{gr}(\mathcal{P})$ forms a dg $\mathbb{S}$-module.
\end{Remark}

\begin{Proposition}
Let $f: (\mathcal{P},\gamma_\PP,d_\PP,\Theta_\PP) \longrightarrow (\mathcal{Q},\gamma_\Q,d_\Q,\Theta_\Q)$  be a graded quasi-isomorphism between two complete curved absolute partial operads. Then it is a weak-equivalence, meaning that $\widehat{\mathrm{B}}_\iota(f): \widehat{\mathrm{B}}(\mathcal{P}) \longrightarrow \widehat{\mathrm{B}}_\iota(\mathcal{Q})$ is an arity-wise quasi-isomorphism.
\end{Proposition}

\begin{proof}
The proof is completely analogous to that of \cite[Theorem 10.23]{grignoulejay18}. It is essentially done by induction. The cofree counital cooperad functor preserves quasi-isomorphisms. Thus it sends graded quasi-isomorphisms between curved partial operads endowed with the trivial operad structure to quasi-isomorphisms. Then one shows a weak version of the five-lemma \cite[Lemma 10.20]{grignoulejay18} which allows us to go from weight $\omega$ to weight $\omega +1$. See Section 10.3 for more details on that matter. 
\end{proof}

\begin{Remark}
The above proposition implies that the weak-equivalences of complete curved operads as defined in \cite{JoanCurved} by graded quasi-isomorphisms should be strictly included inside the weak-equivalences of our model structure. Nevertheless, the objects and the underlying categories considered are quite different in nature, so we do not attempt a precise comparison result.
\end{Remark}

There is a morphism of dg unital partial $\mathbb{S}$-colored operads $f_\kappa: \Omega_\mathbb{S}\left(\mathcal{S} \otimes c\mathcal{O}^\vee \right) \longrightarrow u\mathcal{O}$, which in turn induces a morphism between the associated comonads in the category of dg $\mathbb{S}$-modules. Thus one has an adjunction
\[
\begin{tikzcd}[column sep=7pc,row sep=3pc]
            \mathsf{dg}~\mathsf{upCoop} \arrow[r, shift left=1.1ex, "\mathrm{Res}_{f_\kappa}"{name=F}] &\mathsf{upCoop}_\infty~, \arrow[l, shift left=.75ex, "\mathrm{Coind}_{f_\kappa}"{name=U}]
            \arrow[phantom, from=F, to=U, , "\dashv" rotate=-90]
\end{tikzcd}
\]
where the forgetful functor $\mathrm{Res}_{f_\kappa}$ is fully faithful and preserves quasi-isomorphisms.

\begin{Proposition}\label{prop: commuting triangle of adjunctions}
There following triangle of adjunctions 
\[
\begin{tikzcd}[column sep=5pc,row sep=2.5pc]
&\hspace{1pc}\mathsf{upCoop}_\infty \arrow[dd, shift left=1.1ex, "\widehat{\Omega}_{\iota}"{name=F}] \arrow[ld, shift left=.75ex, "\mathrm{Coind}_{f_\kappa}"{name=C}]\\
\mathsf{dg}~\mathsf{upCoop}  \arrow[ru, shift left=1.5ex, "\mathrm{Res}_{f_\kappa}"{name=A}]  \arrow[rd, shift left=1ex, "\widehat{\Omega}_{\kappa}"{name=B}] \arrow[phantom, from=A, to=C, , "\dashv" rotate=-70]
& \\
&\hspace{2.5pc}\mathsf{curv}~\mathsf{abs}~\mathsf{pOp}^{\mathsf{comp}}~, \arrow[uu, shift left=.75ex, "\widehat{\text{B}}_{\iota}"{name=U}] \arrow[lu, shift left=.75ex, "\widehat{\text{B}}_{\kappa}"{name=D}] \arrow[phantom, from=B, to=D, , "\dashv" rotate=-110] \arrow[phantom, from=F, to=U, , "\dashv" rotate=-180]
\end{tikzcd}
\]
commutes up to natural isomorphism.
\end{Proposition}

\begin{proof}
It is immediate to check that the left-adjoints are naturally isomorphic. Thus by mates of adjunction, the right adjoints are also naturally isomorphic.
\end{proof}

\begin{Remark}
We believe that the complete Bar-Cobar adjunction relative to $\iota$ is a Quillen equivalence. One could generalize the arguments in Section 11 of \cite{grignoulejay18} to the $\mathbb{S}$-colored framework in order to prove this. 

\medskip

On the other hand, one could try to induce a model structure on counital partial cooperads using the forgetful-cofree adjunction and transfer it to the category of complete curved absolute partial operads. It is not clear to us that these model structures should coincide. Indeed, this would amount to proving that the adjunction
\[
\begin{tikzcd}[column sep=7pc,row sep=3pc]
            \mathsf{dg}~\mathsf{upCoop} \arrow[r, shift left=1.1ex, "\mathrm{Res}_{f_\kappa}"{name=F}] &\mathsf{upCoop}_\infty~, \arrow[l, shift left=.75ex, "\mathrm{Coind}_{f_\kappa}"{name=U}]
            \arrow[phantom, from=F, to=U, , "\dashv" rotate=-90]
\end{tikzcd}
\]
is a Quillen equivalence. But this adjunction restricts to the adjunction
\[
\begin{tikzcd}[column sep=7pc,row sep=3pc]
            \mathsf{dg}~u\mathcal{A}ss\text{-}\mathsf{coalg} \arrow[r, shift left=1.1ex, "\mathrm{Res}_{f_\kappa}"{name=F}] &u\mathcal{A}_\infty\text{-}\mathsf{coalg}~, \arrow[l, shift left=.75ex, "\mathrm{Coind}_{f_\kappa}"{name=U}]
            \arrow[phantom, from=F, to=U, , "\dashv" rotate=-90]
\end{tikzcd}
\]
when one restricts to dg $\mathbb{S}$-modules concentrated in arity one. It is not clear at all that this adjunction is a Quillen equivalence, see \cite[Conjecture 8.10]{grignoulejay18}.
\end{Remark}

\section{Duality functors and Koszul duality}
In this section, we build duality functors that intertwine the two Bar-Cobar adjunctions defined in Section 6, forming an algebraic duality square of commuting adjunctions. On the homotopical side of things, we build a second duality square of commuting Quillen adjunctions when (co)unital partial (co)operads are replaced by their "up to homotopy" counterparts. Using these squares, we compute minimal cofibrant resolutions for complete curved absolute partial operads in certain cases of interest. 

\subsection{Algebraic duality square square}
Our goal is to construct left adjoint functors to the linear dual functor that sends "coalgebraic objects" into "algebraic objects".

\begin{lemma}\label{lemma: adjoint à droite}
The linear duality functor 
\[
\begin{tikzcd}[column sep=4pc,row sep=0pc]
\mathsf{dg}~\mathsf{upCoop}^{\mathsf{op}} \arrow[r,"(-)^*"] 
&\mathsf{dg}~\mathsf{upOp}
\end{tikzcd}
\]
admits a left adjoint.
\end{lemma}

\begin{proof}
Consider the following square of functors
\[
\begin{tikzcd}[column sep=4pc,row sep=4pc]
\mathsf{dg}~\mathsf{upCoop}^{\mathsf{op}} \arrow[r,"(-)^*"]  \arrow[d,"\mathrm{U}^\mathsf{op}"{name=SD},shift left=1.1ex ]
&\mathsf{dg}~\mathsf{upOp} \arrow[d,"\mathrm{U}"{name=LDC},shift left=1.1ex ] \\
\mathsf{dg}~\mathbb{S}\text{-}\mathsf{mod}^{\mathsf{op}} \arrow[r,"(-)^*"{name=CC},,shift left=1.1ex] \arrow[u,"\left(\mathscr{T}^\vee(-)\right)^\mathsf{op}"{name=LD},shift left=1.1ex ] \arrow[phantom, from=SD, to=LD, , "\dashv" rotate=0]
&\mathsf{dg}~\mathbb{S}\text{-}\mathsf{mod}~.  \arrow[l,"(-)^*"{name=CB},shift left=1.1ex] \arrow[u,"\mathscr{T}(-)"{name=TD},shift left=1.1ex] \arrow[phantom, from=TD, to=LDC, , "\dashv" rotate=0] \arrow[phantom, from=CC, to=CB, , "\dashv" rotate=90]
\end{tikzcd}
\] 
First notice that the adjunction on the left hand side is monadic, since we consider the opposite of a comonadic adjunction. Furthermore, all categories involved are complete and cocomplete. It is absolutely clear that $(-)^* \cdot \mathrm{U}^\mathsf{op} \cong \mathrm{U} \cdot (-)^*~.$ Thus we can apply the Adjoint Lifting Theorem \cite[Theorem 2]{AdjointLifting} to this situation, which concludes the proof. 
\end{proof}

\begin{Definition}[Sweedler dual]\label{def: Sweedler dual functor}
The \textit{Sweedler duality functor}
\[
\begin{tikzcd}[column sep=4pc,row sep=0pc]
\mathsf{dg}~\mathsf{upOp} \arrow[r,"(-)^\circ"] 
&\mathsf{dg}~\mathsf{upCoop}^{\mathsf{op}}
\end{tikzcd}
\]
is defined as the left adjoint of the linear dual functor. 
\end{Definition}

\begin{Remark}
The proof of the Adjoint Lifting Theorem \cite[Theorem 2]{AdjointLifting} gives an explicit construction of this left adjoint. Let $(\mathcal{P},\{\circ_i\},\eta,d_\PP)$ be a dg unital partial operad, and let 
\[
\gamma_\PP: \mathscr{T}(\mathcal{P}) \longrightarrow \PP
\]
be its structural morphism as an algebra over the tree monad. The Sweelder dual dg counital partial cooperad $\mathcal{P}^\circ$ is given by the following equalizer: 
\[
\begin{tikzcd}[column sep=4pc,row sep=4pc]
\mathrm{Eq}\Bigg(\mathscr{T}^\vee(\mathcal{P}^*) \arrow[r,"(\gamma_\PP)^*",shift right=1.1ex,swap]  \arrow[r,"\varrho"{name=SD},shift left=1.1ex ]
&\mathscr{T}^\vee\left((\mathscr{T}(\mathcal{P}))^*\right) \Bigg)~,
\end{tikzcd}
\]
where $\varrho$ is an arrow constructed using the comonadic structure of $\mathscr{T}^\vee$ and the canonical inclusio of a dg $\mathbb{S}$-module into its double linear dual.
\end{Remark}

\begin{Remark}
If we restrict to unital partial operads concentrated in arity $1$, that is, unital associative algebras, the Sweedler dual functor we have constructed is naturally isomorphic with the original Sweedler dual of \cite{Sweedler69}.
\end{Remark}

\begin{Remark}
Let $(\mathcal{P},\{\circ_i\},\eta,d_\PP)$ be a dg unital partial operad such that $\mathcal{P}(n)$ is degree-wise finite dimensional and bounded above or below. Its Sweedler dual $\mathcal{P}^\circ$ is simply given by $(\mathcal{P}^*,\{ \circ_i^*\},\eta^*,d_\PP^*)$. The adjunction constructed restricts to an anti-equivalence of categories between dg counital partial cooperads which are arity-wise degree-wise finite dimensional and bounded above or below and dg unital partial operads which are arity-wise degree-wise finite dimensional and bounded above or below.
\end{Remark}

Let's turn to the other side of the Koszul duality. We postpone the following proofs and constructions to the Appendix \ref{Appendix B}, where there is a detailed discussion of absolute partial operads and their properties. 

\begin{lemma}
Let $(\mathcal{C},\{\Delta_i\},d_\C,\Theta_\C)$ be a conilpotent curved partial cooperad. Then its linear dual $\C^*$ has inherits a structure of a complete curved absolute partial operad. This defines a functor 
\[
\begin{tikzcd}[column sep=4pc,row sep=0pc]
\left(\mathsf{curv}~\mathsf{pCoop}^{\mathsf{conil}}\right)^{\mathsf{op}} \arrow[r,"(-)^*"]
&\mathsf{curv}~\mathsf{abs}~\mathsf{pOp}^{\mathsf{comp}}~.
\end{tikzcd}
\]
\end{lemma}

\begin{proof}
See Lemma \ref{lemma: dual lin d'une curved conil coop}.
\end{proof}

\begin{Proposition}
The linear duality functor admits a left adjoint. There is an adjunction 
\[
\begin{tikzcd}[column sep=7pc,row sep=3pc]
\mathsf{curv}~\mathsf{abs}~\mathsf{pOp}^{\mathsf{comp}}  \arrow[r, shift left=1.1ex, "(-)^\vee"{name=F}] 
& \left(\mathsf{curv}~\mathsf{pCoop}^{\mathsf{conil}}\right)^{\mathsf{op}} ~. \arrow[l, shift left=.75ex, "(-)^*"{name=U}]
            \arrow[phantom, from=F, to=U, , "\dashv" rotate=-90]
\end{tikzcd}
\]
\end{Proposition}

\begin{proof}
See Proposition \ref{prop: adjonction topo et dual lin en curved}.
\end{proof}

\begin{Definition}[Topological dual functor]\label{def: topological dual functor}
The \textit{topological dual functor} 
\[
\begin{tikzcd}[column sep=4pc,row sep=0pc]
\mathsf{curv}~\mathsf{abs}~\mathsf{pOp}^{\mathsf{comp}} \arrow[r,"(-)^\vee"]
&\left(\mathsf{curv}~\mathsf{pCoop}^{\mathsf{conil}}\right)^{\mathsf{op}}
\end{tikzcd}
\]
is defined as the left adjoint of the linear dual functor.
\end{Definition}

This allows us to construct the first duality square of commuting functors.

\begin{theorem}[Duality square]\label{thm: carré magique}
The following square of adjunction 
\[
\begin{tikzcd}[column sep=5pc,row sep=5pc]
\left(\mathsf{dg}~\mathsf{upOp}\right)^{\mathsf{op}} \arrow[r,"\mathrm{B}^{\mathsf{op}}"{name=B},shift left=1.1ex] \arrow[d,"(-)^\circ "{name=SD},shift left=1.1ex ]
&\left(\mathsf{curv}~\mathsf{pCoop}^{\mathsf{conil}}\right)^{\mathsf{op}}  \arrow[d,"(-)^*"{name=LDC},shift left=1.1ex ] \arrow[l,"\Omega^{\mathsf{op}}"{name=C},,shift left=1.1ex]  \\
\mathsf{dg}~\mathsf{upCoop} \arrow[r,"\widehat{\Omega}"{name=CC},shift left=1.1ex]  \arrow[u,"(-)^*"{name=LD},shift left=1.1ex ]
&\mathsf{curv}~\mathsf{abs}~\mathsf{pOp}^{\mathsf{comp}}~, \arrow[l,"\widehat{\mathrm{B}}"{name=CB},shift left=1.1ex] \arrow[u,"(-)^\vee"{name=TD},shift left=1.1ex] \arrow[phantom, from=SD, to=LD, , "\dashv" rotate=0] \arrow[phantom, from=C, to=B, , "\dashv" rotate=-90]\arrow[phantom, from=TD, to=LDC, , "\dashv" rotate=0] \arrow[phantom, from=CC, to=CB, , "\dashv" rotate=-90]
\end{tikzcd}
\] 
commutes in the following sense: right adjoints going from the top right to the bottom left are naturally isomorphic.
\end{theorem}

\begin{proof}
Let us show the commutativity of this square. We have, for any graded $\mathbb{S}$-module $M$, a natural isomorphism of graded counital partial cooperads 
\[
\mathscr{T}^\vee(sM^*) \cong \left(\mathscr{T}(s^{-1}M) \right)^\circ~.
\]
This isomorphism is obtained as the mate of the obvious isomorphism $(-)^* \cdot \mathrm{U} \cong \mathrm{U}^\mathrm{op} \cdot (-)^*~.$ Let $(\C, \{\Delta_i\}, d_\C, \Theta_\C)$ be a conilpotent curved partial cooperad. One can show by direct inspection that the isomorphism of graded counital partial cooperads 
\[
\mathscr{T}^\vee(s\C^*) \cong \left(\mathscr{T}(s^{-1}\C) \right)^\circ
\]
extends to an isomorphism of dg counital partial cooperads
\[
\widehat{\mathrm{B}}(\C^*) \cong \left(\Omega(\C) \right)^\circ~.
\]
\end{proof}

\begin{Remark}
Let $(\C,\{\Delta_i\},\epsilon,d_\C)$ be a dg counital partial cooperad. Then, by the above theorem, there is an isomorphism
\[
\mathrm{B}(\C^*) \cong \left(\widehat{\Omega}(\C)\right)^\vee~,
\]
which is natural in $\C$.
\end{Remark}

\begin{Proposition}\label{prop: finite dual commutes}
Let $(\mathcal{P},\{\circ_i\},\eta,d_\PP)$ be a dg unital partial operad which is arity-wise degree-wise finite dimensional and bounded above or bounded below. There is an isomorphism
\[
\widehat{\Omega}(\mathcal{P}^*) \cong \left(\mathrm{B}(\mathcal{P})\right)^*
\]
of complete curved absolute partial operads.
\end{Proposition}

\begin{proof}
Let $M$ be a pdg $\mathbb{S}$-module which is arity-wise degree-wise finite dimensional and bounded above or bounded below. There is an isomorphism of complete pdg absolute partial operads:
\[
\overline{\mathscr{T}}^\wedge(s^{-1}M^* \oplus \nu) \cong \left(\overline{\mathscr{T}}(sM \oplus \nu)\right)^*~. 
\]
One can show that this isomorphism extends to an isomorphism of complete curved absolute partial operads
\[
\widehat{\Omega}(\mathcal{P}^*) \cong \left(\mathrm{B}(\mathcal{P})\right)^*
\]
by direct inspection, looking at the morphisms that induce the pre-differentials on each of those constructions.
\end{proof}

\begin{Remark}[Beck-Chevalley condition]
In fact, one can show that there is a monomorphism of complete curved absolute partial operads
\[
\lambda_\PP: \widehat{\Omega}(\mathcal{P}^\circ) \hookrightarrow \left(\mathrm{B}(\mathcal{P})\right)^*
\]
which is natural in $\mathcal{P}$. And $\lambda_\PP$ is an isomorphism if $\PP$ is arity-wise degree-wise finite dimensional and bounded above or below. The sub-category of dg unital partial operad which satisfy the \textit{Beck-Chevalley condition} contains the sub-category of arity-wise degree-wise finite dimensional and bounded above or below dg unital partial operads. 
\end{Remark}

\begin{Example}
Let $\ucom$ be the unital partial operad encoding dg unital commutative algebras. There is an isomorphism of complete curved absolute partial operads: 
\[
(\text{B}\ucom)^* \cong \widehat{\Omega}\ucomd~.
\]
The complete curved absolute partial operad $\widehat{\Omega}\ucomd$ encodes the notion of mixed curved $\mathcal{L}_\infty$-algebras. See Section \ref{Section: curved HTT} for more details on this. 
\end{Example}

\subsection{Homotopical duality square} We now construct the homotopical version of the duality square by replacing (co)unital partial (co)operads by their "up to homotopy" counterparts. Before making those constructions, we state results that are a direct consequence of the results stated in \cite{grignou2021}. They give Quillen equivalence between unital partial operads up to homotopy and conilpotent curved partial cooperads.

\begin{theorem}[\cite{Hinich97}]
There is a cofibrantly generated model structure on the category of unital partial operads up to homotopy with strict morphisms defined by the following classes of morphisms:
\begin{enumerate}
\item The class of weak-equivalences $\mathrm{W}$ is given by strict morphisms of unital partial operads up to homotopy $f: \mathcal{P} \longrightarrow \mathcal{Q}$ such that $f(n): \mathcal{P}(n) \longrightarrow \mathcal{Q}(n)$ is a quasi-isomorphism for all $n \geq 0$. 
\vspace{0.5pc}
\item The class of fibrations is $\mathrm{Fib}$ is given by strict morphisms of unital partial operads up to homotopy $f: \mathcal{P} \longrightarrow \mathcal{Q}$ such that $f(n): \mathcal{P}(n) \twoheadrightarrow \mathcal{Q}(n)$ is a degree-wise epimorphism for all $n \geq 0$. 
\item The class of cofibrations $\mathrm{Cof}$ is given by strict morphisms of unital partial operads up to homotopy which have the left lifting property against all morphism in $\mathrm{W} \cap \mathrm{Fib}$. 
\end{enumerate}
\end{theorem}

\begin{proof}
It is obtained by right transfer theorem using the free-forgetful adjunction.
\end{proof}

The author of \cite{grignou2021} endows the category of conilpotent curved partial cooperads with a transferred structure from unital partial operads, using the Bar-Cobar construction of Subsection \ref{subsection: classical bar-cobar relative to kappa}.

\begin{theorem}[{\cite[Theorem 5]{grignou2021}}]\label{thm: structure de modèles conil curved part coop}
There is a cofibrantly generated model structure on the category of conilpotent curved partial cooperads defined by the following classes of morphisms:
\begin{enumerate}
\item The class of weak-equivalences $\mathrm{W}$ is given by morphisms of conilpotent curved partial cooperads $f: \mathcal{C} \longrightarrow \mathcal{D}$ such that $\Omega(f)(n): \Omega(\mathcal{C})(n) \longrightarrow \Omega(\mathcal{D})(n)$ is a quasi-isomorphism for all $n \geq 0$. \vspace{0.00001pc}
\item The class of cofibrations is $\mathrm{Cof}$ is given by morphisms of conilpotent curved partial cooperads $f: \mathcal{C} \longrightarrow \mathcal{D}$ such that $\Omega(f): \Omega(\mathcal{C}) \longrightarrow \Omega(\mathcal{D})$ is a cofibration of unital partial operads up to homotopy.
\vspace{0.5pc}
\item The class of fibrations $\mathrm{Fib}$ is given by morphisms of conilpotent curved partial cooperads which have the right lifting property against all morphism in $\mathrm{W} \cap \mathrm{Cof}$. 
\end{enumerate}
\end{theorem}

\begin{Proposition}[{\cite[Proposition 13]{grignou2021}}]
In the model structure of conilpotent curved partial cooperads, the class of cofibrations is given by morphisms of conilpotent curved partial cooperads $f: \mathcal{C} \longrightarrow \mathcal{D}$ such that $f(n):\mathcal{C}(n) \longrightarrow \mathcal{D}(n)$ is a monomorphism for all $n \geq 0$. 
\end{Proposition}

The following result is a straightforward consequence of the above.

\begin{Proposition}
There is a Bar-Cobar adjunction
\[
\begin{tikzcd}[column sep=7pc,row sep=3pc]
           \mathsf{upOp}_\infty \arrow[r, shift left=1.1ex, "\Omega_\iota"{name=F}] &\mathsf{curv}~\mathsf{pCoop}^{\mathsf{conil}}~. \arrow[l, shift left=.75ex, "\mathrm{B}_\iota"{name=U}]
            \arrow[phantom, from=F, to=U, , "\dashv" rotate=-90]
\end{tikzcd}
\]
which is a Quillen equivalence.
\end{Proposition}

\begin{proof}
This adjunction is induced by the curved twisting morphism $\iota:\mathcal{S} \otimes c\mathcal{O}^\vee \longrightarrow \Omega_\mathbb{S}\left(\mathcal{S} \otimes c\mathcal{O}^\vee \right)$ using the same methods that in Subsection \ref{subsection: classical bar-cobar relative to kappa}. 

\medskip

Using analogous arguments as in Theorem \ref{Koszulity of uO}, one can show that this curved twisting morphism is Koszul. In turn, using the same arguments as in the proof of Corollary \ref{corollary: koszulness implies Quillen equivalence}, one can deduce that the Bar-Cobar adjunction relative to $\iota$ is a Quillen equivalence. 
\end{proof}

\begin{lemma}\label{lemma: adjoint à droite Sweedler à homotopies près}
The linear duality functor 
\[
\begin{tikzcd}[column sep=4pc,row sep=0pc]
\left(\mathsf{upCoop}_\infty\right)^{\mathsf{op}} \arrow[r,"(-)^*"] 
&\mathsf{upOp}_\infty
\end{tikzcd}
\]
admits a left adjoint.
\end{lemma}

\begin{proof}
Recall that there exists a dg unital partial $\mathbb{S}$-colored operad $\Omega_\mathbb{S}\left(\mathcal{S} \otimes c\mathcal{O}^\vee \right)$ which encodes unital partial operads up to homotopy as its algebras and counital partial cooperads up to homotopy as its coalgebras. Thus the first category is monadic and the second is comonadic. The rest of the proof is \textit{mutatis mutandis} the same as the proof of Lemma \ref{lemma: adjoint à droite}.
\end{proof}

\begin{Definition}[Homotopical Sweedler dual]\label{def: Sweedler dual functor homotopy}
The \textit{homotopical Sweedler duality functor}
\[
\begin{tikzcd}[column sep=4pc,row sep=0pc]
\mathsf{upOp}_\infty \arrow[r,"(-)_h^\circ"] 
&\left(\mathsf{upCoop}_\infty\right)^{\mathsf{op}}
\end{tikzcd}
\]
is defined as the left adjoint of the linear dual functor. 
\end{Definition}

\begin{lemma}
The adjunction 
\[
\begin{tikzcd}[column sep=7pc,row sep=3pc]
            \mathsf{upCoop}_\infty \arrow[r, shift left=1.1ex, "(-)^*"{name=F}] &\left(\mathsf{upOp}_\infty\right)^{\mathsf{op}} ~. \arrow[l, shift left=.75ex, "(-)_h^\circ"{name=U}]
            \arrow[phantom, from=F, to=U, , "\dashv" rotate=-90]
\end{tikzcd}
\]
is a Quillen adjunction.
\end{lemma}

\begin{proof}
The left adjoint $(-)^*$ sends degree-wise monomorphisms to degree-wise epimorphisms. Thus it preserves cofibrations. It also preserves quasi-isomorphisms. Therefore we have a Quillen adjunction.
\end{proof}

\begin{theorem}[Homotopical duality square]\label{thm: homotopical carré magique}
The following square of adjunction 
\[
\begin{tikzcd}[column sep=5pc,row sep=5pc]
\left(\mathsf{upOp}_\infty\right)^{\mathsf{op}} \arrow[r,"\mathrm{B}_\iota^{\mathsf{op}}"{name=B},shift left=1.1ex] \arrow[d,"(-)_h^\circ "{name=SD},shift left=1.1ex ]
&\left(\mathsf{curv}~\mathsf{pCoop}^{\mathsf{conil}}\right)^{\mathsf{op}}  \arrow[d,"(-)^*"{name=LDC},shift left=1.1ex ] \arrow[l,"\Omega_\iota^{\mathsf{op}}"{name=C},,shift left=1.1ex]  \\
\mathsf{upCoop}_\infty \arrow[r,"\widehat{\Omega}_\iota"{name=CC},shift left=1.1ex]  \arrow[u,"(-)^*"{name=LD},shift left=1.1ex ]
&\mathsf{curv}~\mathsf{abs}~\mathsf{pOp}^{\mathsf{comp}}~, \arrow[l,"\widehat{\mathrm{B}}_\iota"{name=CB},shift left=1.1ex] \arrow[u,"(-)^\vee"{name=TD},shift left=1.1ex] \arrow[phantom, from=SD, to=LD, , "\dashv" rotate=0] \arrow[phantom, from=C, to=B, , "\dashv" rotate=-90]\arrow[phantom, from=TD, to=LDC, , "\dashv" rotate=0] \arrow[phantom, from=CC, to=CB, , "\dashv" rotate=-90]
\end{tikzcd}
\] 
commutes in the following sense: right adjoints going from the top right to the bottom left are naturally isomorphic. Furthermore, this square is a square of Quillen adjunctions. 
\end{theorem}

\begin{proof}
The commutativity of this square of functors can be shown using the same arguments as in Theorem \ref{thm: carré magique}. In order to check that it is a square of Quillen adjunctions, we need to check that the adjunction 
\[
\begin{tikzcd}[column sep=7pc,row sep=3pc]
\mathsf{curv}~\mathsf{pOp}^{\mathsf{comp}}  \arrow[r, shift left=1.1ex, "(-)^\vee"{name=F}] 
& \left(\mathsf{curv}~\mathsf{pCoop}^{\mathsf{conil}}\right)^{\mathsf{op}} ~. \arrow[l, shift left=.75ex, "(-)^*"{name=U}]
            \arrow[phantom, from=F, to=U, , "\dashv" rotate=-90]
\end{tikzcd}
\]
is a Quillen adjunction. The right adjoint $(-)^*$ sends monomorphisms to epimorphisms, thus preserves fibrations. Let $f: \mathcal{C} \longrightarrow \mathcal{D}$ be a weak equivalence between two conilpotent curved partial cooperads. This is equivalent to $\Omega_\iota(f): \Omega_\iota\mathcal{C} \longrightarrow \Omega_\iota \mathcal{D}$ being an arity-wise quasi-isomorphism. Since $(-)_h^\circ$ is right Quillen, it preserves quasi-isomorphisms between cofibrant unital partial operads up to homotopy (fibrant in the opposite category). Thus $\Omega_\iota(f)_h^\circ: \Omega_\iota(\mathcal{C})_h^\circ \longrightarrow \Omega_\iota(\mathcal{D})_h^\circ$ is an arity-wise quasi-isomorphism. Using the commutativity of this adjunction, we get that
\[
\widehat{\mathrm{B}}_\iota(f^*): \widehat{\mathrm{B}}_\iota(\mathcal{C}^*) \longrightarrow \widehat{\mathrm{B}}_\iota(\mathcal{D}^*)~,
\]
is an arity-wise quasi-isomorphism. Which is equivalent to $f^*: \mathcal{C}^* \longrightarrow \mathcal{D}^*$ being a weak equivalence in the transferred model structure. Thus $(-)^*$ is a right Quillen functor. This concludes the proof.
\end{proof}

The above homotopical square allows us, using the curved Koszul duality established in \cite{HirshMilles12}, to compute explicit cofibrant resolutions for a vast class of complete curved absolute partial operads.

\begin{Proposition}\label{Prop: cofibrant resolutions of absolutes}
Let $(\mathcal{P},\{\circ_i\},\eta)$ be a unital partial operad, given by some quadratic data $(V,R)$, and which is arity-wise degree-wise finite dimensional and bounded above or below. Let $\mathcal{P}^{\mathrm{\ac}}$ be its Koszul dual conilpotent curved partial cooperad. If $\mathcal{P}$ is Koszul, then 
\[
\widehat{\Omega}(\mathcal{P}^*) \qi (\mathcal{P}^{\mathrm{\ac}})^*
\]
is a cofibrant resolution of $(\mathcal{P}^{\mathrm{\ac}})^*$ in the model category of complete curved absolute partial operads.
\end{Proposition}

\begin{proof}
Recall that $\kappa: \mathcal{P}^{\mathrm{\ac}} \longrightarrow \mathcal{P}$ is Koszul if and only if $f_\kappa: \Omega(\mathcal{P}^{\ac}) \longrightarrow \mathcal{P}$ is an arity-wise quasi-isomorphism. This is equivalent to
\[
g_\kappa: \mathcal{P}^{\ac} \longrightarrow \mathrm{B}\mathcal{P}
\]
being a weak-equivalence in the category of conilpotent curved partial cooperads, since the two model categories are Quillen equivalent. Every conilpotent curved partial cooperad is cofibrant in the model structure of Theorem \ref{thm: structure de modèles conil curved part coop} (thus fibrant in the opposite category), therefore the linear dual functor $(-)^*$ preserves all weak equivalences. This implies that
\[
(g_\kappa)^*: (\mathrm{B}\mathcal{P})^* \longrightarrow (\mathcal{P}^{\ac})^*
\]
is a weak equivalence of complete curved absolute partial operads. Using Proposition \ref{prop: finite dual commutes}, we know that there is an isomorphism
\[
\widehat{\Omega}(\mathcal{P}^*) \cong (\mathrm{B}\mathcal{P})^*
\]
of complete curved absolute partial operads, which concludes the proof.
\end{proof}

\begin{Example}[Cofibrant resolution of $c\mathcal{L}ie^\wedge$]\label{Example: Omega(ucomd)}
Let $c\mathcal{L}ie^\wedge$ be the complete curved absolute partial operad encoding curved Lie algebras of Definition \ref{def: Lie curved absolute}. There is a weak equivalence
\[
\widehat{\Omega}(\ucomd) \qi c\mathcal{L}ie^\wedge
\]
of complete curved absolute partial operads and $\widehat{\Omega}\ucomd$ is the minimal cofibrant resolution of $c\mathcal{L}ie^\wedge$ in this model structure. 
\end{Example}

\begin{Example}[Cofibrant resolution of $c\mathcal{A}ss^\wedge$]
Let $c\mathcal{A}ss^\wedge$ be the complete curved absolute partial operad encoding curved associative algebras of Definition \ref{def: Ass curved absolute}. There is a weak equivalence
\[
\widehat{\Omega}(u\mathcal{A}ss^*) \qi c\mathcal{A}ss^\wedge
\]
of complete curved absolute partial operads and $\widehat{\Omega}(u\mathcal{A}ss^*)$ is the minimal cofibrant resolution of $c\mathcal{A}ss^\wedge$ in this model structure. 
\end{Example}

\section{Application: Homotopy transfer theorem for curved algebras}\label{Section: curved HTT}
The Homotopy Transfer Theorem is a fundamental tool in homological algebra, as it allows to transfer algebraic structures up to homotopy using contractions of dg modules. The operadic reformulation of this theorem was given in \cite[Section 10.3]{LodayVallette12} in terms of a morphism between the Bar constructions of the endomorphisms operads associated to the contraction. This morphism is called the Van der Laan morphism in \textit{loc.cit}. The notion of a contraction is \textit{homotopical} not \textit{homological}, and extends easily to the setting of pdg modules. In this section, we prove a version of the Homotopy Transfer Theorem for curved algebraic structures up to homotopy, extending the Van der Laan morphism to the complete Bar construction of Section \ref{Section: Constructions Bar-Cobar operadiques}. We recover, in the case of "pro-nilpotent" curved $\mathcal{L}_\infty$-algebras in the sense of \cite{getzler2018maurercartan}, the Homotopy Transfer Theorem constructed by Fukaya using fixed-point equations in \cite{FukayaHTT}.

\subsection{Complete filtrations on pdg modules and complete curved absolute endomorphisms operad}
A complete filtration on a pdg modules allows to endow its curved endomorphisms operad with a structure of complete curved absolute partial operad. 

\begin{Definition}[Filtered pdg module]
A \textit{filtered pdg module} $V$ amounts to the data $(V,\mathrm{F}_\bullet V ,d_V)$ of a pdg module $(V,d_V)$ together with a degree-wise decreasing filtration:
\[
V = \mathrm{F}_0V \supset \mathrm{F}_1V \supset \mathrm{F}_2V \supset \cdots \supset \mathrm{F}_n V \supset \cdots~,
\]
such that $d_V(\mathrm{F}_nV) \subset \mathrm{F}_{n+1}V$ for all $n$ in $\mathbb{N}$. Morphisms $f: V \longrightarrow W$ of filtered pdg modules are morphisms of pdg modules which are compatible with the filtrations, that is, $f(\mathrm{F}_nV) \subset \mathrm{F}_n W$ for all $n$.
\end{Definition}

\begin{Remark}
We ask that $d_V$ raises the filtration degree by one for simplicity. These are not the optimal assumptions. In fact, we only need to impose that $d_V^2(\mathrm{F}_nV) \subset \mathrm{F}_{n+1}V$. These would correspond to "gr-dg filtered modules" in \cite{JoanCurved}. In any case, notice that the associated graded of these filtrations are dg modules. 
\end{Remark}

The category of filtered pdg modules can be endowed with a closed symmetric monoidal structure by considering the following filtration on the tensor product. For $V$ and $W$ two filtered pdg modules, their tensor product $V \otimes W$ is endowed with 
\[
\mathrm{F}_n \left(V \otimes W \right) \coloneqq \sum_{p+q = n} \mathrm{Im}\left(F_pV \otimes F_qW \rightarrow V \otimes W\right)~.
\]
The internal hom functor, denoted $\mathrm{hom}(A,B)$, is given by the internal hom of pdg modules endowed with the filtration 
\[
\mathrm{F}_n \left(\mathrm{hom}(V,W)\right) \coloneqq \Big\{ ~f ~\text{in}~ \mathrm{hom}(V,W)~~|~~ f(\mathrm{F}_kV) \subset \mathrm{F}_{k+n}W\Big\} ~.
\]
\begin{Definition}[Complete pdg module]\label{def: complete pdg module}
Let $V$ be a filtered pdg module. It is a \textit{complete} pdg module if the canonical morphism 
\[
\pi_V: V \longrightarrow \lim_{n} V/\mathrm{F}_n V 
\]
is an isomorphism of filtered pdg modules. 
\end{Definition}

Complete pdg modules form a reflexive subcategory of filtered pdg modules. The reflector is simply given by the completion functor, which sends a filtered pdg module $V$ to its completion
\[
\widehat{V} \coloneqq \lim_{n} A/\mathrm{F}_nA~.
\]
By setting $V \widehat{\otimes} W \coloneqq \widehat{V \otimes W}~,$ the category of complete pdg modules is endowed with a closed symmetric monoidal structure, the internal hom being the same as the one defined above. One can define operads in this context, \cite[Section 2]{DSV18} for more details. Here, we will only use the following construction. 

\begin{Definition}[Complete curved endomorphisms partial operad]
Let $V$ be a complete pdg module. Its \textit{complete curved endomorphisms partial operad} is given by the complete pdg $\mathbb{S}$-module 
\[
\mathrm{end}_V(n) \coloneqq \mathrm{hom}^{(\geq 1)}(V^{\otimes n}, V)~,
\]
where we only consider here morphisms of pdg modules which raise the filtration degree by at least one. The operad structure is given by the partial composition of functions. Its pre-differential is given by $\partial \coloneqq [d_V,-]$ and its curvature determined by $\Theta_V(\mathrm{id}) \coloneqq d_V^2~.$
\end{Definition}

\begin{lemma}
Let $V$ be a complete pdg module. Its complete curved endomorphisms partial operad is a complete curved \textit{absolute} partial operad. 
\end{lemma}

\begin{proof}
We postpone this proof to Lemma \ref{lemma: curved endomorphisms is absolute} in the Appendix \ref{Appendix B}, where there is a detailed discussion of complete curved absolute partial operads. 
\end{proof}

\subsection{Homotopy contractions and Van der Laan morphisms}
The data of a homotopy contraction between two complete pdg modules which is compatible with their underlying filtrations induces a "Van der Laan morphism" between the complete Bar construction of their complete curved endomorphisms operads. From this, one recovers a Homotopy Transfer theorem for curved algebras over a vast class of cofibrant complete curved absolute partial operads. 

\begin{Definition}[Homotopy contraction]
Let $V$ and $H$ be two complete pdg modules. A \textit{homotopy contraction} amounts to the data of 
\[
\begin{tikzcd}[column sep=5pc,row sep=3pc]
V \arrow[r, shift left=1.1ex, "p"{name=F}] \arrow[loop left]{l}{h}
&H~, \arrow[l, shift left=.75ex, "i"{name=U}]
\end{tikzcd}
\]
where $p$ and $i$ are two morphisms of filtered pdg modules and $h$ is a morphism of filtered graded modules of degree $-1$. This data satisfies the following conditions:
\[
p i = \mathrm{id}_H~, \quad i p - \mathrm{id}_V = d_V h + h d_V~.
\]
\end{Definition}

\begin{Notation}
We say that $H$ is a \textit{homotopy retract} of $V$ if there exists a homotopy contraction as in the definition above.
\end{Notation}

\begin{Proposition}[Van der Laan morphism]
Let $V$ and $H$ be two complete pdg modules such $H$ is a homotopy retract of $V$. The data of this homotopy contraction induces a morphism of dg counital partial cooperads
\[
\mathrm{VdL}: \widehat{\mathrm{B}}(\mathrm{end}_V) \longrightarrow \widehat{\mathrm{B}}(\mathrm{end}_H)~. 
\]
This morphism is called the Van der Laan morphism associated to the contraction.
\end{Proposition}

\begin{proof}
Let 
\[
\mathrm{vdl}: \mathscr{T}^\wedge(s\mathrm{end}_V) \longrightarrow \mathrm{end}_H
\]
be the morphism of pdg $\mathbb{S}$-modules which sends a series of rooted trees labeled by elements of $\mathrm{end}_V$ to the converging series in $\mathrm{end}_H$ given by applying the Van der Laan map to each rooted tree of the series. For a recollection on the standard Van der Laan map, see \cite[Section 10.3.2]{LodayVallette12}. One can restrict $\mathrm{vdl}$ along the inclusion of the cofree counital partial cooperad inside the completed tree endofunctor. By the universal property of the cofree counital partial cooperad, this induces a morphism of counital partial cooperads:
\[
\mathrm{VdL}: \mathscr{T}^\vee(\mathrm{end}_V) \longrightarrow \mathscr{T}^\vee(\mathrm{end}_H)~.
\]
Let us show that $\mathrm{VdL}$ commutes with the differentials. It equivalent to the following diagram commuting 
\[
\begin{tikzcd}[column sep=3pc,row sep=3pc]
\mathscr{T}^\vee(\mathrm{end}_V) \arrow[r,"\mathrm{VdL}"] \arrow[d,"d_{\mathrm{bar}} ",swap] 
&\mathscr{T}^\vee(\mathrm{end}_H) \arrow[d,"\psi_1 + \psi_2"] \\
\mathscr{T}^\vee(\mathrm{end}_V) \arrow[r,"\mathrm{vdl}"] 
&\mathrm{end}_H~,
\end{tikzcd}
\]
where $\psi_1 + \psi_2$ is the map that induces the differential of the complete Bar construction as its unique coderivation extending it. One can show that this diagram commutes for elements in $\mathscr{T}^\vee(\mathrm{end}_V)$, which viewed as infinite sums of trees, have a trivial coefficient in front of the trivial tree $|~$, extending the same computations as in \cite[Section 10.3.2]{LodayVallette12} to infinite sums of rooted trees. For the trivial tree $|~$, a small computation shows that:
\[
\mathrm{vdl} \cdot d_{\mathrm{bar}}(|) = p \cdot d_V^2 \cdot i = d_H \cdot p \cdot i \cdot d_H = d_H^2 = (\psi_1 + \psi_2) \cdot \mathrm{VdL}(|)~, 
\]
using the fact that $p$ and $i$ are morphisms of pdg modules. 
\end{proof}

\begin{theorem}[Curved Homotopy Transfer Theorem]\label{thm: curved HTT}
Let $\mathcal{C}$ be a dg counital partial cooperad. Let $V$ and $H$ be two complete pdg modules such $H$ is a homotopy retract of $V$. Then any curved $\widehat{\Omega}(\mathcal{C})$-algebra structure on $V$ can be transferred along the homotopy contraction to a curved $\widehat{\Omega}(\mathcal{C})$-algebra structure on $H$. 
\end{theorem}

\begin{proof}
One has that
\[
\mathrm{Hom}_{\mathsf{curv}~\mathsf{abs}~\mathsf{pOp}^{\mathsf{comp}}}(\widehat{\Omega}\C,\mathrm{end}_V) \cong \mathrm{Tw}(\C,\mathrm{end}_V) \cong \mathrm{Hom}_{\mathsf{dg}~\mathsf{upCoop}}(\C,\widehat{\mathrm{B}}(\mathrm{end}_V))~,
\]
thus a curved $\widehat{\Omega}(\mathcal{C})$-algebra structure on $V$ is equivalent to a morphism of dg counital partial cooperads 
\[
\phi: \mathcal{C} \longrightarrow \widehat{\mathrm{B}}(\mathrm{end}_V)~.
\]
Any such morphism can be pushed forward by the Van der Laan morphism
\[
\mathrm{VdL} \cdot \phi: \mathcal{C} \longrightarrow \widehat{\mathrm{B}}(\mathrm{end}_H)~,
\]
thus inducing a curved $\widehat{\Omega}(\mathcal{C})$-algebra structure on $H$. 
\end{proof}

\begin{Remark}
All cofibrant complete curved absolute partial operads can be written as $\widehat{\Omega}(\mathcal{C})$, where $\mathcal{C}$ is a counital partial cooperad \textit{up to homotopy}. By restricting to the case where $\C$ is a dg counital partial cooperad, we are restricting to the case of Koszul resolutions. See Proposition \ref{Prop: cofibrant resolutions of absolutes}. In order to prove the statement for all cofibrant complete curved absolute partial operad, one should prove that a homotopy contraction induces a morphism between the complete Bar constructions \textit{relative to $\iota$} of Proposition \ref{Prop: complete Bar-Cobar adjunction relative to iota}.
\end{Remark}

Let us conclude with some examples of how this framework can be applied. Let $\widehat{\Omega}(\ucomd)$ be the complete curved absolute partial operad of Example \ref{Example: Omega(ucomd)}.

\begin{Example}
Let $V$ be a complete dg modules. Then a curved $\widehat{\Omega}(\ucomd)$-algebra structure on $V$ corresponds to a \textit{pro-nilpotent curved} $\mathcal{L}_\infty$-algebra structure on $V$ in the sense of \cite{getzler2018maurercartan}.
\end{Example}

\begin{Example}[Fukaya's Homotopy Transfer Theorem]
Let $V$ and $H$ be two complete dg modules and let 
\[
\begin{tikzcd}[column sep=5pc,row sep=3pc]
V \arrow[r, shift left=1.1ex, "p"{name=F}] \arrow[loop left]{l}{h}
&H~, \arrow[l, shift left=.75ex, "i"{name=U}]
\end{tikzcd}
\]
be homotopy contraction. Given a pro-nilpotent curved $\mathcal{L}_\infty$-algebra structure on $V$, one can transfer it to $H$ using a version of the Homotopy Transfer Theorem given in \cite{FukayaHTT}. The transferred structure is given as the solution of a fixed-point equation. If one computes the solution of this fixed point equation using the methods of \cite[Appendix A]{robertnicoud2020higher}, one obtains the same formulae for the transferred structure as the ones obtained from Theorem \ref{thm: curved HTT}.
\end{Example}

\section*{Appendix A: What is an absolute partial operad ?}\label{Appendix B}
The notion of an \textit{absolute partial operad} is an example of a vast class of algebraic structures that emerge as algebras over a cooperad. See Subsection \ref{subsection : algebras and coalgebras} for an introduction to this notion. In our particular case, absolute partial operads appear as algebras over the conilpotent partial $\mathbb{S}$-colored cooperad $\mathcal{O}^*$. Here $\mathcal{O}^*$ denotes the linear dual of the partial $\mathbb{S}$-colored operad $\mathcal{O}$ which encodes partial operads as its algebras. The goal of this Appendix is to give a somewhat explicit definition of what absolute partial operads are and how to characterize them. Then to compare them with standard partial operads. Afterwards, we will generalize these results to the curved case. 

\subsection{Absolute partial operads}
We work in the underlying category of $\mathbb{S}$-modules for simplicity. This subsection admits a straightforward generalization to dg $\mathbb{S}$-modules or pdg $\mathbb{S}$-modules.

\begin{lemma}\label{lemma: iso avec complete tree monad}
There is an isomorphism of endofunctors in the category of $\mathbb{S}$-modules:
\[
\widehat{\mathscr{S}}_{\mathbb{S}}^c (\mathcal{O}^*)(-) \cong \overline{\mathscr{T}}^\wedge(-)~,
\]
where $\overline{\mathscr{T}}^\wedge(-)$ is the reduced completed tree endofunctor.
\end{lemma}

\begin{proof}
We have that:
\begin{align*}
\widehat{\mathscr{S}}_{\mathbb{S}}^c (\mathcal{O}^*)(M)(n) \coloneqq \prod_{(n_1,\cdots,n_r) \in \mathbb{N}^{r}} \mathrm{Hom}_{\left(\mathbb{S}_{n_1} \times \cdots \times \mathbb{S}_{n_r}\right)~\wr~ \mathbb{S}_r } \left( \mathcal{O}^*(n_1,\cdots,n_r;n), M(n_1) \otimes \cdots \otimes M(n_r) \right) \cong \\
\cong \prod_{(n_1,\cdots,n_r) \in \mathbb{N}^{r}} \mathcal{O}(n_1,\cdots,n_r;n) \otimes_{\left(\mathbb{S}_{n_1} \times \cdots \times \mathbb{S}_{n_r}\right)~\wr~ \mathbb{S}_r }  M(n_1) \otimes \cdots \otimes M(n_r)~,
\end{align*}
since each $\mathcal{O}^*(n_1,\cdots,n_r;n)$ is finite dimensional over $\kk$ and that we are working over a field of characteristic, thus invariants and coinvariants turn out to be canonically isomorphic. Using the same bijection as in the proof of Lemma \ref{lemma with the trees bijection}, we identify this last term with the reduced completed tree endofunctor.
\end{proof}

\begin{Corollary}
There is a monad structure on the reduced completed tree endofunctor $\overline{\mathscr{T}}^\wedge(-)$. Its structural morphism 
\[
\mathsf{Sub}^\wedge: \overline{\mathscr{T}}^\wedge \circ \overline{\mathscr{T}}^\wedge (-) \longrightarrow \overline{\mathscr{T}}^\wedge(-)
\]
is given by the substitution of infinite series of rooted trees. 
\end{Corollary}

\begin{proof}
Let $M$ be an $\mathbb{S}$-module. An element of $\overline{\mathscr{T}}^\wedge \circ \overline{\mathscr{T}}^\wedge(M)$ amounts to the data of a series of rooted trees which are labeled in the following way: if $v$ is a vertex of with $k$ inputs of a rooted tree $\tau$, then it is labeled by a series of rooted trees of arity $k$ which are themselves labeled by elements of $M$ in the usual way. Let us describe the monad structure transported via the isomorphism of Lemma \ref{lemma: iso avec complete tree monad}. This monad structure is given by first distributing all the labeling series and then by substituting each of the node by the rooted tree it is labeled with.
Pictorially, the substitution is given by 

\begin{center}
\includegraphics[width=130mm,scale=1]{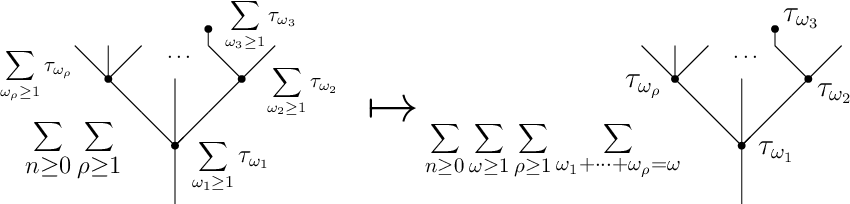}
\end{center}

where one should substitute each node by the corresponding labeling rooted tree $\tau_{\omega_i}$. 
\end{proof}

Thus we can rewrite the definition of an $\mathcal{O}^*$-algebra in terms of the reduced completed tree monad. 

\begin{Definition}[Absolute partial operad]
An \textit{absolute partial operad} $\mathcal{Q}$ amounts to the data $(\mathcal{Q},\gamma_\Q)$ of an algebra structure over the reduced completed tree monad 
\[
\gamma_\Q: \overline{\mathscr{T}}^\wedge (\Q) \longrightarrow \Q~.
\]
\end{Definition}

The goal of this section is to make sense of this definition.

\begin{Remark}
The structural morphism $\gamma_\Q$ of an absolute partial operad gives a well-defined composition of any \textit{infinite sums} for rooted trees labeled by elements of $\mathcal{Q}$, \textit{without presupposing any underlying topology} on $\mathcal{Q}$. Let 
\[
\sum_{n \geq 0} \sum_{\omega \geq 1} \sum_{\tau \in \mathrm{RT}_n^\omega} \tau
\]
be a infinite sum of rooted trees with vertices labeled by elements of $\Q$. Notice that, in general:
\[
\gamma_\Q \left(\sum_{n \geq 0} \sum_{\omega \geq 1} \sum_{\tau \in \mathrm{RT}_n^\omega} \tau \right) = \sum_{n \geq 0} \gamma_\Q \left( \sum_{\omega \geq 1} \sum_{\tau \in \mathrm{RT}_n^\omega} \tau \right) \neq \sum_{n \geq 0} \sum_{\omega \geq 1} \sum_{\tau \in \mathrm{RT}_n^\omega} \gamma_\Q(\tau)~,
\]
as the later sum is not even well-defined in $\Q$ in general.
\end{Remark}

\begin{Proposition}
Let $M$ be an $\mathbb{S}$-module such that $M(0) = M(1) = 0$, then the data of a partial operad structure on $M$ is equivalent to the data of an absolute partial operad structure on $M$. 
\end{Proposition}

\begin{proof}
In this case, there are only a finite amount of rooted tree labeled by elements of $M$ for each possible arity, thus the reduced completed tree monad coincides with the reduced tree monad. 
\end{proof}

\begin{lemma}
Let $(\mathcal{Q},\gamma_\Q)$ be an absolute partial operad. The structure map $\gamma_\Q$ induces a partial operad structure on $\Q$. This defines a faithful functor
\[
\begin{tikzcd}[column sep=4pc,row sep=0pc]
\mathsf{Res}: \mathsf{abs}~\mathsf{pOp} \arrow[r]
&\mathsf{pOp}~.
\end{tikzcd}
\]
\end{lemma}

\begin{proof}
It is straightforward to check that the inclusion of endofunctors
\[
\overline{\mathscr{T}}(-) \hookrightarrow \overline{\mathscr{T}}^\wedge(-)
\]
is in fact an inclusion of monads. Thus, by pulling back along this inclusion, any absolute partial operad structure gives a partial operad structure. This pullback amounts to restrict $\gamma_\Q$ to finite sums of rooted trees inside the reduced completed tree monad.
\end{proof}

\begin{Remark}
Let $(\mathcal{Q},\gamma_\Q)$ be an absolute partial operad. It is in particular endowed with partial composition maps $\{\circ_i\}$ which satisfy the usual axioms of a partial operad. These maps are \textit{part of the structure} of an absolute partial operad.
\end{Remark}

\begin{Proposition}[Absolute envelope of a partial operad]
There is an adjunction
\[
\begin{tikzcd}[column sep=7pc,row sep=3pc]
\mathsf{pOp}  \arrow[r, shift left=1.1ex, "\mathsf{Abs}"{name=F}] 
&\mathsf{abs}~\mathsf{pOp}~. \arrow[l, shift left=.75ex, "\mathsf{Res}"{name=U}] \arrow[phantom, from=F, to=U, , "\dashv" rotate=-90]
\end{tikzcd}
\]
We call the left adjoint \textit{the absolute envelope} of a partial operad, and we denote it by $\mathsf{Abs}$.
\end{Proposition}

\begin{proof}
Notice that both categories are categories of algebras over accessible monads, thus are presentable. Since $\mathsf{Res}$ is accessible and preserves all limits, it has a left adjoint, by \cite[Theorem 1.66]{AdamekRosicky}.
\end{proof}

There is a canonical topology on absolute partial operads induced by the structural morphism.

\begin{Definition}[Canonical filtration on an absolute partial operad]\label{def: canonical topology on absolutes}
Let $(\mathcal{Q},\gamma_\Q)$ be an absolute partial operad. The \textit{canonical filtration} is the decreasing filtration given by
\[ 
\overline{\mathscr{F}}_{\omega} \Q \coloneqq \mathrm{Im}\left(\gamma_\Q^{(\geq \omega)}: \overline{\mathscr{T}}^{\wedge~(\geq \omega)}(\Q) \longrightarrow \Q \right)
\]
for all $\omega \geq 0$, where $\overline{\mathscr{T}}^{\wedge~(\geq \omega)}(\Q)$ denotes the possibly infinite sums of rooted trees that have terms of at least $\omega$ internal edges. Each $\overline{\mathscr{F}}_{\omega} \Q$ defines an ideal of $\Q$. Notice that: 
\[
\Q = \overline{\mathscr{F}}_{0} \Q \supseteq \overline{\mathscr{F}}_{1} \Q \supseteq \cdots \supseteq \overline{\mathscr{F}}_{\omega} \Q \supseteq \cdots.
\]
\end{Definition}

\begin{Definition}[Nilpotent absolute partial operad]
Let $(\mathcal{Q},\gamma_\Q)$ be an absolute partial operad. It is said to be \textit{nilpotent} if there exists an $\omega \geq 1$ such that 
\[
\Q/\overline{\mathscr{F}}_{\omega} \Q \cong \Q~.
\]
The absolute partial operad is said to be $\omega_0$\textit{-nilpotent} if $\omega_0$ is the smallest integer such that the above isomorphism exits.
\end{Definition}

\begin{Definition}[Completion of an absolute partial operad]\label{def: completion functor for absolute operads}
Let $(\mathcal{Q},\gamma_\Q)$ be an absolute partial operad, its \textit{completion} $\widehat{\Q}$ is given by the following limit
\[
\widehat{\Q} \coloneqq \lim_{\omega} \Q/\overline{\mathscr{F}}_{\omega} \Q
\]
taken in the category of absolute partial operads.
\end{Definition}

The completion is functorial. There is a canonical morphism of absolute partial operads
\[
\psi: \Q \twoheadrightarrow \widehat{\Q}~,
\]
which is always an epimorphism. See Remark \ref{Remark: varphi is an epimorphism} as to why this is the case in the general setting.

\begin{Definition}[Complete absolute partial operad]\label{def: complete absolute operads}
Let $(\mathcal{Q},\gamma_\Q)$ be an absolute partial operad. It is \textit{complete} if the canonical morphism $\psi: \Q \twoheadrightarrow \widehat{\Q}$ is an isomorphism of absolute partial operads. 
\end{Definition} 

\begin{Example}
Any nilpotent absolute partial operad is also a complete absolute partial operad. Any complete absolute partial operad is the limit of a tower of nilpotent absolute partial operads. Any free absolute partial operad is also complete.
\end{Example}

\begin{Remark}
An absolute partial operad is complete if and only if the topology induced by its canonical filtration is Hausdorff.
\end{Remark}

\begin{Remark}
Let $(\Q, \gamma_\Q)$ be a \textit{complete} absolute partial operad and let 
\[
\sum_{n \geq 0} \sum_{\omega \geq 1} \sum_{\tau \in \mathrm{RT}_n^\omega} \tau
\]
be a infinite sum of rooted trees with vertices labeled by elements of $\Q$. In this case we have that
\[
\gamma_\Q \left(\sum_{n \geq 0} \sum_{\omega \geq 1} \sum_{\tau \in \mathrm{RT}_n^\omega} \tau \right) = \sum_{n \geq 0} \sum_{\omega \geq 1} \sum_{\tau \in \mathrm{RT}_n^\omega} \gamma_\Q(\tau)~,
\]
since $\gamma_\Q$ commutes with the sum over the weight of rooted trees and since there are only a finite amount of rooted trees of arity $n$ and of weight $\omega$.
\end{Remark}

\begin{Proposition}
There is an isomorphism of categories between the category of $\omega_0$-nilpotent absolute partial operads and the category of $\omega_0$-nilpotent partial operads, for all $\omega_0 \geq 1$. 
\end{Proposition}

\begin{proof}
Let $\overline{\mathscr{T}}^{(\leq \omega_0)}$ denote the reduce tree monad truncated at rooted trees of with more than $\omega_0$ internal edges. The data of an $\omega_0$-nilpotent partial operad amounts to the data of an algebra over the reduced $\omega_0$-truncated tree monad. The data of an $\omega_0$-nilpotent absolute partial operad also amounts to the data of an algebra over the reduced $\omega_0$-tree monad. Indeed, since there are only a finite number of rooted trees of weight less than $\omega_0$ at any given arity, the direct sum and the product coincide.
\end{proof}

The notion of a complete absolute partial operad appears naturally when one takes the linear dual of a conilpotent partial cooperad.

\begin{Proposition}\label{prop: linear dual of a conil coop}
Let $(\C, \{\Delta_i\})$ be a conilpotent partial cooperad. Then its linear dual $\C^*$ has an induced absolute partial operad structure. Furthermore, since $\C$ is conilpotent, then $\C^*$ is a complete absolute partial operad. It defines a functor
\[
\begin{tikzcd}[column sep=7pc,row sep=3pc]
\left(\mathsf{pCoop}^{\mathsf{conil}}\right)^{\mathsf{op}}  \arrow[r, "(-)^*"{name=F}] 
&\mathsf{abs}~\mathsf{pOp}^{\mathsf{comp}}~. 
\end{tikzcd}
\]
\end{Proposition}

\begin{proof}
Let 
\[
\Delta_\C: \C \longrightarrow \overline{\mathscr{T}}^c(\C)
\]
denote the structural map of $\C$ as a coalgebra over the reduced tree comonad, given by Proposition \ref{prop: reduced tree comonad gives conil part coop}. By taking the linear dual we obtain a map
\[
\Delta_\C^*: \left(\overline{\mathscr{T}}^c(\C)\right)^* \longrightarrow \C^*~.
\]
There is a canonical inclusion
\[
\iota_\C:  \overline{\mathscr{T}}^{\wedge}(\C^*) \hookrightarrow \left(\overline{\mathscr{T}}^c(\C)\right)^*
\]
which is an isomorphism if and only if $\C$ is arity-wise finite dimensional. Thus by pulling back along $\iota_\C$, we obtain a map $\gamma_\C \coloneqq \Delta_\C^* \cdot \iota_\C$. One can check that this endows $\C^*$ with the structure of an algebra over the reduced complete tree monad. Since $\C$ is conilpotent, it can be written as 
\[
\C \cong \colim_{\omega} \mathscr{R}_\omega \C~,
\]
where $\mathscr{R}_\omega \C$ is the sub-cooperad given by operations which admit non-trivial $\omega$-iterated partial decompositions. Thus 
\[
\C^* \cong \lim_{\omega} \left(\mathscr{R}_\omega \C\right)^*~,
\]
where one can check that
\[
\left(\mathscr{R}_\omega \C\right)^* \cong \C^* /\overline{\mathscr{F}}_{\omega} \C^*~.
\]
Therefore the resulting absolute partial operad $\C^*$ is indeed complete.
\end{proof}

\begin{Proposition}\label{prop: adjonction topo dual lin}
The linear functor admits a left adjoint
\[
\begin{tikzcd}[column sep=7pc,row sep=3pc]
\left(\mathsf{pCoop}^{\mathsf{conil}}\right)^{\mathsf{op}}  \arrow[r,shift left=1.1 ex, "(-)^*"{name=F}] 
&\mathsf{abs}~\mathsf{pOp}^{\mathsf{comp}}~. \arrow[l, shift left=.75ex, "(-)^\vee"{name=U}] \arrow[phantom, from=F, to=U, , "\dashv" rotate=90]
\end{tikzcd}
\]
We denote its left adjoint $(-)^\vee$ and call it the \textit{topological dual functor}.
\end{Proposition}

\begin{proof}
Consider the following square of functors
\[
\begin{tikzcd}[column sep=4pc,row sep=4pc]
\left(\mathsf{pCoop}^{\mathsf{conil}}\right)^{\mathsf{op}} \arrow[r,"(-)^*"]  \arrow[d,"\mathrm{U}^\mathsf{op}"{name=SD},shift left=1.1ex ]
&\mathsf{abs}~\mathsf{pOp}^{\mathsf{comp}} \arrow[d,"\mathrm{U}"{name=LDC},shift left=1.1ex ] \\
\mathbb{S}\text{-}\mathsf{mod}^{\mathsf{op}} \arrow[r,"(-)^*"{name=CC},,shift left=1.1ex] \arrow[u,"\left(\overline{\mathscr{T}}^c(-)\right)^\mathsf{op}"{name=LD},shift left=1.1ex ] \arrow[phantom, from=SD, to=LD, , "\dashv" rotate=0]
&\mathbb{S}\text{-}\mathsf{mod}~.  \arrow[l,"(-)^*"{name=CB},shift left=1.1ex] \arrow[u,"\overline{\mathscr{T}}^\wedge(-)"{name=TD},shift left=1.1ex] \arrow[phantom, from=TD, to=LDC, , "\dashv" rotate=0] \arrow[phantom, from=CC, to=CB, , "\dashv" rotate=90]
\end{tikzcd}
\] 
The adjunction on the left hand side is monadic. All these categories are bicomplete. It is clear that $(-)^* \cdot \mathrm{U}^\mathsf{op} \cong \mathrm{U} \cdot (-)^*~.$ Thus we can apply the Adjoint Lifting Theorem \cite[Theorem 2]{AdjointLifting}, which concludes the proof. 
\end{proof}

\begin{Remark}
There is an explicit description of this left adjoint. Let $(\Q,\gamma_Q)$ be a complete absolute partial operad, its topological dual $\Q^\vee$ is given by the equalizer
\[
\begin{tikzcd}[column sep=4pc,row sep=4pc]
\mathrm{Eq}\Bigg(\overline{\mathscr{T}}^c(\mathcal{Q}^*) \arrow[r,"(\gamma_\Q)^*",shift right=1.1ex,swap]  \arrow[r,"\varrho"{name=SD},shift left=1.1ex ]
&\overline{\mathscr{T}}^c\left((\overline{\mathscr{T}}^\wedge(\mathcal{P}))^*\right) \Bigg)~,
\end{tikzcd}
\]
where $\varrho$ is a map constructed using the comonad structure of $\overline{\mathscr{T}}^c$ and the canonical inclusion into the double linear dual.
\end{Remark}

Lastly, one can compare the canonical filtrations induced by the structure of an absolute partial operad and by the structure of a partial operad. In general, they do not agree.

\begin{lemma}
Let $(\mathcal{Q},\gamma_\Q)$ be an absolute partial operad. Let $(\mathsf{Res}(\Q), \mathsf{Res}(\gamma_\Q))$ the partial operad given by its restriction. Then there is an inclusion of $\mathbb{S}$-modules:
\[
\mathscr{F}_{\omega}\mathsf{Res}(\Q) \subseteq \overline{\mathscr{F}}_{\omega} \Q
\]
for all $\omega \geq 0$, between its canonical filtration as a partial operad and its canonical filtration as an absolute partial operad. 
\end{lemma}

\begin{proof}
Elements in $\overline{\mathscr{F}}_{\omega} \Q$ are given by the images of the structural morphism $\gamma_\Q$ of series of rooted trees of at least weight $\omega$. Since any finite sum of at least weight $\omega$ is a particular example of this, any element in $\mathscr{F}_{\omega}\mathsf{Res}(\Q)$ is also included in $\overline{\mathscr{F}}_{\omega} \Q$.
\end{proof}

\begin{Corollary}
Let $(\mathcal{Q},\gamma_\Q)$ be an absolute partial operad. Then the canonical topology of $(\mathsf{Res}(\Q),\allowbreak \mathsf{Res}(\gamma_\Q))$ as a partial operad is Hausdorff.
\end{Corollary}

\begin{proof}
We have that
\[
\bigcap_{\omega \in \mathbb{N}} \mathscr{F}_{\omega}\mathsf{Res}(\Q) \subset \bigcap_{\omega \in \mathbb{N}} \overline{\mathscr{F}}_{\omega} \Q = 0~,
\]
hence the canonical topology on $\mathsf{Res}(\Q)$ is indeed Hausdorff.
\end{proof}

\begin{Counterexample}
Let $\overline{\mathscr{T}}^\wedge(M)$ be the free complete absolute partial operad generated by an $\mathbb{S}$-module $M$. Then 
\[
\mathsf{Res}\left(\overline{\mathscr{T}}^\wedge(M)\right)
\]
is not complete as a partial operad. 
\end{Counterexample}

\subsection{Complete curved absolute partial operads}
Recall Definition \ref{def: cO vee} of the conilpotent curved $\mathbb{S}$-colored partial cooperad $c\mathcal{O}^\vee$ which encodes conilpotent curved partial cooperads as its coalgebras. The goal of this subsection is to understand what algebras over this $\mathbb{S}$-colored cooperad are. For this purpose, we now consider pdg $\mathbb{S}$-modules as our ground category.

\begin{lemma}\label{lemma: dual Schur of cOvee}
Let $M$ be an pdg $\mathbb{S}$-module. There is an isomorphism of pdg $\mathbb{S}$-modules
\[
\widehat{\mathscr{S}}_\mathbb{S}^c(c\mathcal{O}^\vee)(M) \cong \overline{\mathscr{T}}^\wedge(M \oplus \nu)~,
\]
where $\nu$ is an arity $1$ generator of degree $-2$. This isomorphism is natural in $M$. 
\end{lemma}

\begin{proof}
This is a straightforward extension of Lemma \ref{lemma: iso avec complete tree monad}, using a similar bijection to the one defined in the proof of Proposition \ref{Prop: cO-cogebres.}.
\end{proof}

\begin{Definition}[Complete curved absolute partial operad]
A \textit{complete curved absolute partial operad} $(\Q,\gamma_\Q , d_\Q , \Theta_\Q)$ amounts to the data of $(\Q,\gamma_\Q ,d_\Q )$ a complete pdg absolute partial operad, as defined in the preceding subsection, endowed with a morphism of pdg $\mathbb{S}$-modules $\Theta_\Q: \I \longrightarrow \overline{\mathscr{F}}_{1}\Q$ of degree $-2$ such that the following diagram commutes
\[
\begin{tikzcd}[column sep=7.5pc,row sep=3pc]
\Q \arrow[r,"\mathrm{diag}"] \arrow[rrd,"(d_\Q)^2", bend right =10]
&\Q \oplus \Q \cong (\I \circ \Q) \oplus (\Q \circ \I) \arrow[r,"(\Theta_\Q~ \circ~ \mathrm{id})~ -~(\mathrm{id}~ \circ'~ \Theta_\Q)"] 
& \Q \circ_{(1)} \Q \arrow[d,"\gamma_{(1)}"]\\
&
&\Q~,
\end{tikzcd}
\]
where $\mathrm{diag}$ is given by $\mathrm{diag}(\mu) \coloneqq (\mu,\mu)$.
\end{Definition}

\begin{Proposition}\label{Prop: cO algebres vrai}
The category of complete curved $c\mathcal{O}^\vee$-algebras is equivalent to the category of complete curved absolute partial operads.
\end{Proposition}

\begin{proof}
Let $(\Q,\gamma_\Q , d_\Q , \Theta_\Q)$ be a complete curved absolute partial operad. One can extend $\gamma_\Q$ into a morphism of pdg $\mathbb{S}$-modules 
\[
\gamma_\Q^+: \overline{\mathscr{T}}^\wedge(\Q \oplus \nu) \longrightarrow \Q
\]
as follows 
\begin{enumerate}
\item It sends $\nu$ to $\Theta_\Q(\mathrm{id})$ in $\overline{\mathscr{F}}_{1}\Q~.$

\item It sends a rooted tree $\tau$ with vertices labeled by elements of $\Q$ and possibly to containing some unary vertices labeled by $\nu$ to the corresponding compositions of operations in $\Q$ where the unary vertices labeled by $\nu$ are replaced by $\Theta_\Q$. 

\item It defined on infinite sums of rooted trees in $\overline{\mathscr{T}}^\wedge(\Q \oplus \nu)$ as follows
\[
\sum_{n \geq 0} \sum_{\omega \geq 1} \sum_{\tau \in \mathrm{RT}_n^\omega} \tau \mapsto \sum_{n \geq 0} \sum_{\omega \geq 1} \sum_{\tau \in \mathrm{RT}_n^\omega} \gamma_\Q^+(\tau)~,
\]
which is well-defined since $\Theta_\Q(\mathrm{id})$ is in $\overline{\mathscr{F}}_{1}\Q$ and since $\Q$ is a complete absolute partial operad.
\end{enumerate}
This endows $\Q$ with the structure of a pdg $c\mathcal{O}^\vee$-algebra. Furthermore, one checks that it is complete for its canonical filtration as a $c\mathcal{O}^\vee$-algebra, since both filtrations are the same. It is straightforward to check that 
\[
\begin{tikzcd}[column sep=4.5pc,row sep=3pc]
\Q \cong \widehat{\mathscr{S}}_{\mathbb{S}}^c(\I_\mathbb{S})(\Q) \arrow[r,"\widehat{\mathscr{S}}_{\mathbb{S}}^c(\Theta_{c\mathcal{O}^\vee})(\mathrm{id}) "] \arrow[rd,"- d_\Q^2",swap]
&\widehat{\mathscr{S}}_{\mathbb{S}}^c(c\mathcal{O}^\vee)(\Q) \arrow[d,"\gamma_\Q"]\\
&\Q 
\end{tikzcd}
\]
commutes, making it a complete curved $c\mathcal{O}^\vee$-algebra. The other way around, let $(\Q,\gamma_\Q^+,d_\Q)$ be a complete curved $c\mathcal{O}^\vee$-algebra. Restricting $\gamma_\Q^+$ along the obvious inclusion
\[
\overline{\mathscr{T}}^\wedge(\Q) \hookrightarrow \overline{\mathscr{T}}^\wedge(\Q \oplus \nu)~,
\]
endows $\Q$ with an absolute partial operad structure. Furthermore, since its canonical filtration as an absolute partial operad is contained in its canonical filtration as a $c\mathcal{O}^\vee$-algebra, this implies that $\Q$ is a complete absolute partial operad. Define $\Theta_\Q(\mathrm{id}) \coloneqq \gamma_\Q^+(\nu)$, and the commutative of the above triangle implies that $\Q$ forms indeed a complete curved absolute partial operad.
\end{proof}

\begin{lemma}\label{lemma: dual lin d'une curved conil coop}
Let $(\mathcal{C},\{\Delta_i\},d_\C,\Theta_\C)$ be a conilpotent curved partial cooperad. Then its linear dual $\C^*$ has inherits a structure of a complete curved absolute partial operad. This defines a functor 
\[
\begin{tikzcd}[column sep=4pc,row sep=0pc]
\left(\mathsf{curv}~\mathsf{pCoop}^{\mathsf{conil}}\right)^{\mathsf{op}} \arrow[r,"(-)^*"]
&\mathsf{curv}~\mathsf{abs}~\mathsf{pOp}^{\mathsf{comp}}~.
\end{tikzcd}
\]
\end{lemma}

\begin{proof}
We know that $\C^*$ is a complete absolute partial operad from Proposition \ref{prop: linear dual of a conil coop}. Therefore $\Theta_\C^*$ endows $\C^*$ with a complete curved absolute partial operad structure.
\end{proof}

\begin{Proposition}\label{prop: adjonction topo et dual lin en curved}
There is an adjunction 
\[
\begin{tikzcd}[column sep=7pc,row sep=3pc]
\mathsf{curv}~\mathsf{pOp}^{\mathsf{comp}}  \arrow[r, shift left=1.1ex, "(-)^\vee"{name=F}] 
& \left(\mathsf{curv}~\mathsf{pCoop}^{\mathsf{conil}}\right)^{\mathsf{op}} ~. \arrow[l, shift left=.75ex, "(-)^*"{name=U}]
            \arrow[phantom, from=F, to=U, , "\dashv" rotate=-90]
\end{tikzcd}
\]
\end{Proposition}

\begin{proof}
This is a particular case of Proposition \ref{prop: adjonction topo dual lin}. The only thing to prove is that if $(\Q,\gamma_\Q , d_\Q , \allowbreak \Theta_\Q)$ is a complete curved absolute partial operad, its topological dual $\Q^\vee$ conilpotent partial cooperad is indeed a conilpotent curved partial operad. Notice that the square constructed in the proof of \textit{loc.cit} induces a natural monomorphism of pdg $\mathbb{S}$-modules 
\[
(-)^\vee \hookrightarrow (-)^*~,
\]
thus $\Q^\vee$ is a sub-object of $\Q^*$ and its pre-differential is induced by restriction. This in turn implies that $\Theta_\Q^\vee$ endows $\Q^\vee$ with a conilpotent curved partial cooperad structure. 
\end{proof}

\subsection{Examples}
Here are some examples of complete curved absolute partial operads. As first examples, let us construct the "absolute analogues" of the curved operads $\Liec$ and $\Assc$ constructed in Section \ref{Section: Curved Operads}.

\medskip

Let $M$ be the pdg $\mathbb{S}$-module given by $(\mathbb{K}.\zeta,0, \mathbb{K}.\beta,0,\cdots)$ with zero pre-differential, where $\zeta$ is an arity $0$ operation of degree $-2$, and $\beta$ is an arity $2$ operation of degree $0$, basis of the signature representation of $\mathbb{S}_2$. 

\begin{Definition}[$\Liec^\wedge$ absolute operad]\label{def: Lie curved absolute}
The \textit{complete curved absolute partial operad} $\Liec^\wedge$ is given by the free pdg absolute partial operad generated by $M$ modulo the ideal generated by the Jacobi relation on the generator $\beta$. It is endowed with the curvature $\Theta$ given by $\Theta(\mathrm{id}) \coloneqq \beta \circ_1 \zeta~.$  
\end{Definition}

\begin{lemma}
The data $(\Liec^\wedge, 0, \Theta)$ forms a complete curved absolute partial operad. 
\end{lemma}

\begin{proof}
One can show by direct computation that this absolute partial operad is complete. The rest of the proof is analogous to Lemma \ref{lemmalie}.
\end{proof}

\begin{Proposition}
Let $u\mathcal{C}om$ be the unital partial operad encoding unital commutative algebras and let $u\mathcal{C}om^{\ac}$ be its Koszul dual conilpotent curved partial cooperad. There is an isomorphism of complete curved partial operads
\[
\left(u\mathcal{C}om^{\ac}\right)^* \cong \Liec^\wedge.
\]
\end{Proposition}

\begin{proof}
By direct inspection.
\end{proof}

Let $N$ be the pdg $\mathbb{S}$-module given by $(\mathbb{K}.\phi,0,\mathbb{K}[\mathbb{S}_2].\mu,0,\cdots)$ with zero pre-differential, where $\phi$ is an arity $0$ operation of degree $-2$, and $\mu$ is an binary operation of degree $0$, basis of the regular representation of $\mathbb{S}_2$.

\begin{Definition}[$\Assc^\wedge$ absolute operad]\label{def: Ass curved absolute}
The \textit{complete curved absolute partial operad} $\Assc^\wedge$ is given by the free pdg absolute partial operad generated by $N$ modulo the ideal generated by the associativity relation on the generator $\mu$. It is endowed with the curvature $\Theta$ given by $\Theta(\mathrm{id}) \coloneqq \mu \circ_1 \phi - \mu \circ_2 \phi.$ 
\end{Definition}

\begin{lemma}
The data $(\Assc^\wedge, 0, \Theta)$ forms a complete curved absolute partial operad.
\end{lemma}

\begin{proof}
Analogous to the proof of the previous lemma.
\end{proof}

\begin{Proposition}
Let $u\mathcal{A}ss$ be the unital partial operad encoding unital associative algebras and let $u\mathcal{A}ss^{\ac}$ be its Koszul dual conilpotent curved partial cooperad. There is an isomorphism of complete curved partial operads
\[
\left(u\mathcal{A}ss^{\ac}\right)^* \cong \Assc^\wedge.
\]
\end{Proposition}

\begin{proof}
By direct inspection.
\end{proof}

\begin{Proposition}\label{absolute assliem}
There is a morphism of complete curved absolute partial operads $\Liec^\wedge \longrightarrow \Assc^\wedge$ which is determined by sending $\beta \mapsto \mu - \mu^{(12)}~,$ and $\zeta$ to $\phi$. 
\end{Proposition}

\begin{proof}
The proof is identical to the proof of Proposition \ref{assliem}.
\end{proof}

\begin{Remark}
As one can guess by the above examples, one can construct "absolute analogues" of well-known operads when they are given by generators and relations. Furthermore, they show that if $\mathcal{P}$ is a unital partial operad, then its Koszul dual \textit{operad} $\mathcal{P}^!$ is in fact an \textit{absolute operad}.
\end{Remark}

Another class of examples is given by object which already carry an underlying filtration that makes them "complete". Let $V$ be a complete pdg module as defined in Definition \ref{def: complete pdg module}, its curved endomorphism operad carries a natural structure of a complete curved absolute partial operad.

\begin{Definition}[Complete curved endomorphisms partial operad]
Let $V$ be a complete pdg module. Its \textit{complete curved endomorphisms partial operad} is given by the complete pdg $\mathbb{S}$-module 
\[
\mathrm{end}_V(n) \coloneqq \mathrm{hom}^{(\geq 1)}(V^{\otimes n}, V)~,
\]
where we only consider here morphisms of pdg modules which raise the filtration degree by at least one. The operad structure is given by the partial composition of functions. Its pre-differential is given by $\partial \coloneqq [d_V,-]$ and its curvature determined by $\Theta_V(\mathrm{id}) \coloneqq d_V^2~.$
\end{Definition}

\begin{lemma}\label{lemma: curved endomorphisms is absolute}
Let $V$ be a complete pdg module. Its complete curved endomorphisms partial operad is a complete curved \textit{absolute} partial operad. 
\end{lemma}

\begin{proof}
Forgetting the filtration, $\mathrm{end}_V$ is in particular a partial operad, thus an algebra over the reduced tree monad. Let 
\[
\gamma_V: \overline{\mathscr{T}}(\mathrm{end}_V) \longrightarrow \mathrm{end}_V
\]
be its structural morphism. The morphism $\gamma_V$ can be extended to the completed reduced tree monad using the completeness of the underlying filtration of $\mathrm{end}_V$. Notice that it is crucial to restrict ourselves to operations in the endomorphisms operad which raise the degree by at least one in order for this to be true. One can check that the extension of $\gamma_V$ to the complete reduced tree monad satisfies the axioms of an algebra over this monad. It is thus an absolute partial operad ; moreover it forms a curved absolute partial operad endowed with its curvature, see Lemma \ref{lemma: endo curved is curved}. 

\medskip

Let us show it is a complete as a curved absolute partial operad. Let $\mathrm{F}_\bullet \mathrm{end}_V$ be its underlying filtration and let $\overline{\mathscr{F}}_{\bullet} \mathrm{end}_V$ the canonical filtration induced by its curved absolute partial operad structure. See Definition \ref{def: canonical topology on absolutes} for more details. There is an obvious inclusion $\overline{\mathscr{F}}_{n}\mathrm{end}_V \subset \mathrm{F}_n \mathrm{end}_V$, since any operation which can be written as partial compositions iterated $n$ times also must raise the filtration degree by at least $n$. Thus 
\[
\bigcap_{n \in \mathbb{N}} \overline{\mathscr{F}}_{n} \mathrm{end}_V \subset \bigcap_{n \in \mathbb{N}} \mathrm{F}_n \mathrm{end}_V = 0~,
\]
which implies that $\mathrm{end}_V$ is a complete curved absolute partial operad.
\end{proof}

\begin{Remark}
This provides for a natural setting that allows the canonical filtration of a complete curved absolute partial operad to be reflected upon its algebras. See Section \ref{Section: curved HTT} for an application in this direction.
\end{Remark}

\bibliographystyle{alpha}
\bibliography{bibe}

\end{document}